\documentclass{book}
\usepackage{packages}

\title{Confluence of quantum $K$-theory to quantum cohomology for projective spaces}
\author[1]{Alexis Roquefeuil}
\affil[1]{Universit\'e d'Angers. LAREMA, D\'epartement de math\'ematiques, B\^atiment I, Facult\'e des Sciences, 2 Boulevard Lavoisier, F-49045 Angers cedex 01, France.}
\date{Ph.D. thesis (online version) \\ September 20, 2019}

\AtEndDocument{%
  \par
  \medskip
  \begin{tabular}{@{}l@{}}%
    \textsc{Alexis Roquefeuil, Universit\'e d'Angers. LAREMA, D\'epartement de math\'ematiques, B\^atiment I,}\\
    \textsc{Facult\'e des Sciences, 2 Boulevard Lavoisier, F-49045 Angers cedex 01, France.}\\
    \textit{E-mail address}: \texttt{roquefeuil@math.univ-angers.fr}
  \end{tabular}}

\begin{document}

\renewcommand{\thechapter}{\Roman{chapter}}

%
%



\maketitle

    

\pagebreak
\hspace{0pt}
\vfill
\begin{center}
    This thesis was financed by Université d'Angers (imputation budg\'etaire A900210).
    
    \bigskip
    
    The author also received financial support from  from Agence Nationale de la Recherche's projects ANR-13-IS01-0001 (SISYPH) and ANR-17-CE40-0014 (CatAG).
\end{center}

\vfill
\hspace{0pt}
\pagebreak


\tableofcontents

\chapter{Introduction}


\section{Outline}

\subsection*{Enumerative geometry and Gromov--Witten theory}

Enumerative geometry is a branch of mathematics which essentially deals with counting the number of solutions to a geometrical problem. One of the first enumerative geometry problem one encounters is the following: given two distinct points of the plane, how many lines going through these two points can we find?
This problem has a more difficult version: given five points of the plane in general position, how many conics go through these five points?


\textit{Gromov--Witten theory} gives tools to answer similar enumerative geometry problems. One of its feats was to solve the generalisation of our two first problems:

\begin{pb}
    Let $d \in \mathbb{Z}_{>0}$. Find the number $N_d$ of rational curves of degree $d$ in $\mathbb{P}^2_\mathbb{C}$ going through $3d-1$ points.
\end{pb}

The answer to this problem was given in 1994 by  M. Kontsevich and Y. Manin in \cite{Konts_Manin_CohFT}. They show that the numbers $N_d$ satisfy the recurrence relation (see Example \ref{HGW:ex_P2_computation_Nd})
\[
    N_d = \sum_{\substack{d_1+d_2=d \\ d_1, d_2>0}} N_{d_1} N_{d_2} \left( \binom{3d-4}{3d_1-2} d_1^2d_2^2 - \binom{3d-4}{3d_1-1}d_1^3 d_2 \right)
\]
This was a great step forward: before, we only knew the values $N_d$ for $d$ small, and no one could expect these numbers to satisfy such a recursive relation.

Gromov--Witten theory is a branch of algebraic geometry. It first came to birth in theoretical physicists' string theory.
The mathematical community would then realise its potential when a group of four physicists,  P. Candelas, X. de la Ossa, P. Green and L. Parkes \cite{Candelas_Ossa_Xenia_Green_MS} announced a mean to compute similar numbers $N_d$ obtained by replacing $\mathbb{P}^2_\mathbb{C}$ by an arbitrary quintic $X \subset \mathbb{P}^4_\mathbb{C}$ (\textit{Clemens' conjecture}).
The next step is to include these ideas in the context of geometry.

\subsection*{Gromov--Witten invariants and quantum differential equations}

One major difficulty we encounter is the actual definition of the numbers $N_d$. These numbers are examples of Gromov--Witten invariants. We begin by giving only an intuition for their definition.

\begin{defin*}[sketch]
    Let $X$ be a projective complex variety. Let $g,n \in \mathbb{Z}_{\geq 0}$, $d \in H_2(X;\mathbb{Z})$.
    Consider some cycles $Z_1, \dots, Z_n \in Z_*(X)$.
    The Gromov--Witten invariant associated to this data is the number
    \[
        \left\langle [Z_1] , \dots , [Z_n] \right\rangle_{g,n,d}^\textnormal{coh}
        =
        \left(
            \begin{gathered}
                \textnormal{Number of curves } C \subset X \textnormal{ of genus } g, \textnormal{ and homological class } d,\\
                \textnormal{satisfying for all } i, C \cap Z_i \neq \emptyset
            \end{gathered}
        \right)
        \in \mathbb{Z}_{\geq 0}?
    \]
\end{defin*}

To give this definition a true meaning, we have to realise this number as the degree of some intersection product (\cite{Fulton_intersection_theory, Vistoli_Intthy_stacks}) on a moduli space parametrising this data - called the \textit{moduli space of stable maps}.
The obstacle to defining and computing Gromov--Witten invariants comes from the geometry of this moduli space.
The construction of this space, due to M. Kontsevich, outputs not a scheme, but a Deligne--Mumford stack which is not of pure dimension in general. The dimension issue was fixed by B. Fantechi and K. Behrend's intrinsic normal cone \cite{BF_intnormalcone}, which defines the \textit{virtual fundamental class} $\left[ \overline{\mathcal{M}}_{g,n}(X,d) \right]^\text{vir}$. This virtual class allows us to define Gromov--Witten as an integral on the cycle of the correct dimension. 

\begin{defin*}[\ref{HGW:def_GWI}]
    Let $g,n,d$ as above, and let $\alpha_i$ be the Poincaré dual of $[Z_i]$. The associated Gromov--Witten invariant is defined by
    \[
        \langle \alpha_1 , \dots , \alpha_n \rangle_{g,n,d}^\textnormal{coh}
        = \int_{\left[ \overline{\mathcal{M}}_{g,n}(X,d) \right]^\text{vir}} \bigcup_i  \text{ev}_i^\star(\alpha_i) \in \mathbb{Q}
    \]
\end{defin*}

Now, we would like to actually compute these integrals. 
We proceed to define some generating series, a product and a bundle with connection.
Properties of the moduli spaces above can be translated to properties on these constructions (see e.g. Proposition \ref{HGW:prop_qprod_properties}).

We fix a basis of the cohomology $H^*(X;\mathbb{Q}) = \textnormal{Span}(T_i)_{i \in I}$, and we associate to $T_i$ a coordinate $t_i$. Thus, an arbitrary class in cohomology can be written as $\tau = \sum_i t_i T_i$. We denote by $g$ the metric on $H^*(X; \mathbb{C})$ given by Poincaré duality: $g(T_i,T_j)=\int_X T_i \cup T_j$.
We encode the Gromov--Witten invariants in the generating series
\begin{defin*}[\ref{HGW:def_potential}]
    The Gromov--Witten potential is the formal power series of variables $t_0, \dots, t_N$ defined by
     \[
         \mathcal{F}(t_i) = \sum_{\substack{n \geq 0 \\d \in H_2(X;\mathbb{Z})}} \frac{1}{n!} \langle \tau, \dots, \tau \rangle^\textnormal{coh}_{0,n,d}
     \]
\end{defin*}
We assume from now on that there exists some open set $U$ on which this series converges.

\begin{example*}[\ref{HGW:ex_potential_P2}]
    For $X=\mathbb{P}^2$, the potential is expressed with the numbers $N_d$ defined above by
    \[
        \mathcal{F}(t_1,t_2,t_3) =
        \frac{1}{2} (t_0 t_1^2 + t_0^2 t_2) + \sum_{d=1}^\infty e^{dt_1} N_d \frac{t_2^{3d-1}}{(3d-1)!}
    \]
\end{example*}

\begin{defin*}[\ref{HGW:def_qprod}]
     The quantum product $\bullet_\tau$ is a deformation of the classic product $\cup$ on cohomology, which depends on parameters $(t_i)$. It is defined by the relation
     \[
         g(T_i \bullet_\tau T_j, T_k) = \partial_{t_i} \partial_{t_j} \partial_{t_k} \mathcal{F}
     \]
\end{defin*}

We can now construct the quantum $\mathcal{D}$-module, which is the data of a bundle with a connection $(F,\nabla)$ (cf. \cite{Dub_Frob}).
\begin{defin*}[\ref{HGW:def_qdm}]
     Let $F$ be the trivial bundle on $H^*(X;\mathbb{C}) \times \mathbb{P}^1$ of fibre $H^*(X;\mathbb{C})$. We denote by $z$ the local coordinate $\mathbb{P}^1$ at $0$.
     Dubrovin's connection $\nabla$ is defined by the following formula:
     \[
         \left\{\begin{aligned}
             & \nabla_{\partial_{t_i}} T_j = \left(\partial_{t_i} + \frac{1}{z} T_i \bullet_\tau\right) T_j, && 0 \leq i \leq N\\
             & \nabla_{\partial_z} T_j = \left(\partial_z - \frac{1}{z^2} \mathfrak{E} \bullet_\tau + \frac{1}{z} \frac{\textnormal{deg}_{H^*(X)}}{2} \right) T_j
         \end{aligned}\right.
     \]
     where $\mathfrak{E}$ is some section of the bundle $F$, which will be made explicit in another chapter.
\end{defin*}

We define a new generating series of Gromov--Witten invariants
\begin{defin*}[\ref{HGW:def_J_fn}]
    Givental's $J$-function is defined by $J=S^{-1} \mathds{1}$, where $S$ is a fundamental solution of the quantum $\mathcal{D}$-module and $\mathds{1}$ is the constant section of the bundle $F$ whose value on the fibre is the unit of cohomology $\mathds{1} \in H^0(X;\mathbb{C})$.
\end{defin*}
On one hand, it turns out that Givental's $J$-function is the solution of some differential equation. In this thesis, we will focus on the case $X=\mathbb{P}^N$. Under this condition, the equation satisfied by $J$ is
\[
    [\left(z \partial_{t_1}\right)^{N+1} - e^{t_1}]J = 0
\]

On the other hand, we are able to build an explicit solution to this differential equation, called Givental's $I$-function. These functions are related by the

\begin{thm*}[\cite{Givental_EquivariantGW}]
    For $X=\mathbb{P}^N$, the functions $I$ and $J$ satisfy $I=J$
\end{thm*}

\begin{remark*}
    There are two interpretations for the definition of $I$ :
    \begin{enumerate}
        \item By mirror symmetry via the GKZ $\mathcal{D}$-modules.
        \item Via the localisation theorem applied to Givental's equivariant $J$-function.
    \end{enumerate}
    However, we will not mention anything else on the function $I$.
\end{remark*}

\subsection*{Quantum $K$-theory and $q$-difference equations}

More recently in 2004, Y.P. Lee and A. Givental gave \cite{Lee_qk,Giv_qk_wdvv} defined new enumerative invariants inspired by previous Gromov--Witten invariants. These invariants are defined by replacing cohomological definitions by their $K$-theoretical analogues. Y.P. Lee constructed a virtual structure sheaf $\mathcal{O} ^ {\text{vir}} _ {\overline{\mathcal{M}}}$. We define the $K$-theoretical Gromov--Witten invariants by the following:

\begin{defin*}[\ref{KGW:def_KGWI}]
    Let $X$ be a projective complex variety. Let $g,n \in \mathbb{Z}_{\geq 0}$, $d \in H_2(X;\mathbb{Z})$.
    Consider some classes $\phi_1, \dots, \phi_n \in K(X)$. The $K$-theoretical Gromov--Witten invariant associated to this data is the number
    \[
        \left\langle
            \phi_1, \cdots, \phi_n
        \right\rangle_{g,n,d}^{K\textnormal{th}}
        =
        \chi \left(
            \overline{\mathcal{M}}_{g,n}\left(X,d \right);
            \mathcal{O} ^ {\text{vir}} _ { \overline{\mathcal{M}}_{g,n}\left(X,d \right) }
            \bigotimes_{i=1}^n  \text{ev}_i^*(\phi_i)
        \right)
        \in \mathbb{Z}
    \]
\end{defin*}

\begin{qns*}
    \textbf{(Q1)}   \quad
    Can we build the analogue in $K$-theory of the quantum product and the quantum $\mathcal{D}$-module?
    
    \textbf{(Q2)}   \quad
    Can we relate together cohomological Gromov--Witten and $K$-theoretical Gromov--Witten invariants ?
\end{qns*}

To answer the first question \textit{\textbf{(Q1)}}, the naive analogue of the product using the classic metric $g(L_1,L_2)=\chi(L_1 \otimes L_2)$ is not associative.
To fix this problem, Givental--Lee \cite{Lee_qk,Giv_qk_wdvv} introduce a modified metric with the of $K$-theoretical Gromov--Witten invariants.
Once the metric is changed, we are more or less back in the same situation: we can define a product, the operators $\nabla_{\partial_{t_i}}$ and the $J$-function.
However, we observe something new: instead of the operator $\nabla_{\partial_z}$, Givental--Iritani--Milanov--Tonita \cite{Iri_Mil_Ton_qk,Giv_Ton_qk_HRR} obtain $q$-difference operators.

To answer the second question \textit{\textbf{(Q2)}}, there are two approaches.

The first approach is to look for an application of a Riemann--Roch theorem.
Unfortunately, a Hirzebruch--Riemann--Roch formula we could use for Deligne--Mumford stack, due to B. Toën \cite{Toen_GRR_stacks}, only works when these stacks are smooth - which our moduli spaces are not in general (though this is the case when the target space $X$ is a projective space).
In 2011, A. Givental and V. Tonita \cite{Giv_Ton_qk_HRR} prove a theorem relating the two Gromov--Witten theories. The statements are nonetheless quite technical.

The second approach is the aim of this thesis. It is inspired by the computations of the $K$-theoretical $J$-function by Givental--Lee \cite{Giv_Lee_qk}. In the case $X=\mathbb{P}^N$, the $J$-function can be linked to the $q$-hypergeometric series
\[
    f_q(Q) = 
    \sum_{d \geq 0}
    \frac{1}{\prod_{r=1}^d \left( 1 - q^r \right)^{N+1}} Q^d
\]
which is solution not of a differential equation, but of a $q$-difference equation
\[
    \left[ \left( 1 - \qdeop{Q} \right)^{N+1} - Q \right]  J^{K \textnormal{th}}(q,Q) = 0
\]

Given a $q$-difference equation, we are able to obtain a differential equation using a phenomenon called \textit{confluence}. We can therefore wonder if we can compare the confluence of $q$-difference equations in quantum $K$-theory with the differential equations in quantum cohomology.

\subsection*{Confluence and Gromov--Witten theories}

The confluence phenomenon for $q$-difference equations was studied first by J. Sauloy in 2000 \cite{Sauloy_qde_regsing}.
This property says that a $q$-difference equation can admit a differential equation as a limit by doing "$q \to 1$".
Notice that for any $k \in \mathbb{Z}$ we have
\[
    \frac{\qdeop{Q} - \textnormal{Id}}{q-1} \cdot Q^k
    =
    \left(
        1 + q + q^2 + \cdots + q^{k-1}
    \right)
    Q^k
\]
Because of this, we have the formal limit
\[
    \lim_{q \to 1} \frac{\qdeop{Q} - \textnormal{Id}}{q-1} \cdot Q^k
    =
    k Q^k
    =
    Q \partial_Q \cdot Q^k
\]
This principle generalises to general $q$-difference equations, as long as we allow ourselves to specify further what "$q \to 1$" means. A $q$-difference equation that has a well defined limit when $q$ tends to $1$ is said to be \textit{confluent} (see Definition \ref{qde:def_confluentsys}).
Then, the limit of this $q$-difference equation defines a differential equation. 
It makes sense to compare the $q$-difference equation and its limit: for example, the solutions of the $q$-difference equation give solutions to the differential equation by taking their limit when $q$ tends to $1$ (see Theorem \ref{qde:thm_confluence_of_sol_with_initial_cond}).

The main result of this thesis adapts the confluence of $q$-difference equation in the context of quantum $K$-theory to obtain the theorem below.

\begin{thm*}[\ref{qkqde:thm_JK_sol_confluence_statement}]
    For $X=\mathbb{P}^N$, let $J^\textnormal{coh}$ (resp. $J^{K \textnormal{th}}$) be the small cohomological (resp. $K$-theoretical) $J$-function. Then,
    \begin{enumerate}
        \item   The $q$-difference equation satisfied by $J^{K \textnormal{th}}$ will degenerate through confluence to the differential equation satisfied by $J^\textnormal{coh}$.
        \item    We denote by $\textnormal{ch} : K \left( \mathbb{P}^N \right) \otimes \mathbb{Q} \to H^* \left( \mathbb{P}^N ; \mathbb{Q} \right)$ the Chern character.
        Let $\textnormal{confluence}\left(J^{K \textnormal{th}}\right)$ the result of confluence applied to the solution $J^{K \textnormal{th}}$. Then, we have
        \[
            \textnormal{ch}\left(\textnormal{confluence}\left(J^{K \textnormal{th}}\right)\right)
            =
            J^\textnormal{coh}
        \]
    \end{enumerate}
\end{thm*}

\section{Plan of the thesis}


In Chapter \ref{chapter:stablemaps}, we define the moduli space of stable maps and briefly expose some constructions and geometrical properties needed for the later chapters.

\noindent In Chapter \ref{chapter:H_GW}, we define Gromov--Witten invariants of a target space $X$ and list their properties.
Then, we use these invariants to define a deformation of the cohomology ring of the target space $X$, called quantum cohomology.
The remaining of this chapter is dedicated to the definition and the study of the quantum $\mathcal{D}$-module, from which we construct Givental's $J$-function.

\noindent In Chapter \ref{chapter:QK}, we construct the $K$-theoretic analogues of the previous chapter.
More precisely, we define $K$-theoretic Gromov--Witten invariants and quantum $K$-theory as a deformation of $K$-theory of the target space.
Then, we try to construct the analogue of the quantum $\mathcal{D}$-module.
Lastly, we explicit the $q$-difference operators acting on quantum $K$-theory.

\noindent In Chapter \ref{chapter:regsingqde}, we give the necessary background on $q$-difference equations to be able to state our main theorem.
We explain how to construct the fundamental solution of a regular singular $q$-difference equation.
Then, we explain the confluence of these systems.

\noindent In Chapter \ref{chapter:QKqDE}, we state and prove our main theorem, which uses confluence of $q$-difference equation to relate quantum $K$-theory with quantum cohomology.


We recommend the reader familiar with Gromov--Witten to focus on Chapters \ref{chapter:regsingqde} and \ref{chapter:QKqDE}. The construction of the $q$-difference module in quantum $K$-theory, is recalled in Subsection \ref{kgw:subsection_q_shift_operators_in_QK}.

We suggest the reader unacquainted with Gromov--Witten theory to skip the technical details of the three first chapters in a first lecture.
In Chapter \ref{chapter:H_GW}, this reader should focus instead on the properties of the quantum $\mathcal{D}$-module as a meromorphic connection (Section \ref{hgw:section_QH_and_QDM}) while keeping the examples in mind. The construction of Givental's $J$-function will also be important.
In Chapter \ref{chapter:QK}, we suggest to focus on the construction of the $J$-function (Definition \ref{kgw:def_j_function}). The construction of the $q$-difference module in Subsection \ref{kgw:subsection_q_shift_operators_in_QK} can be skipped, and we refer instead to the $q$-difference equations exposed in the Chapter \ref{chapter:QKqDE}, starting with Proposition \ref{qkqde:prop_JK_sol_qde_pre}.

\chapter{Moduli space of stable maps}\label{chapter:stablemaps}

In this introductory chapter we review the construction of the moduli space of stable maps, which is an essential ingredient to the define Gromov--Witten invariants.
Then, we define the evaluation maps and the tautological cotangent line bundles which will also be useful for defining Gromov--Witten invariants.

\section{Stable maps and their moduli}

\subsection{Stable maps}

\begin{defin}[\cite{KontEnum}]\label{HGW:def_stablemap}
Let $n \in \mathbb{Z}_{\geq 0}$ and $X$ be a complex projective variety with even cohomology. A \textit{stable map} is the data of a connected proper curve $C$ with $n$ markings $p_1,\dots,p_n$ and a morphism $f: C \to X$ such that
\begin{enumerate}[label=(\roman*)]
    \item The singularities of $C$ are of nodal type at worst
    \item The markings $p_1,\dots,p_n \in C$ are distinct smooth points of the curve.
    \item If $C$ has an irreducible component $C_0$, such that $C_0$ is of genus 0 and $f$ is constant on $C_0$,  then $C_0$ must contain three points which are either singularities or markings. If $C$ has genus 1 and there are no marking, then $f$ must not be constant.
\end{enumerate}
\end{defin}

\begin{defin}
    Let $\left( f:(C;\underline{p}) \to X \right)$ and $\left( g:(C',\underline{p'}) \to X\right)$ be two stable maps. An \textit{isomorphism of stable maps} $\varphi : \left( f:(C;\underline{p}) \to X \right) \to \left( g:(C',\underline{p'}) \to X\right)$ is the data of an isomorphism $\varphi : C \to C'$ such that for all $i\in\{ 1,\dots,n \}, \varphi(p_i)=p'_i$ and the triangle below is commutative:
    \begin{center}
        \begin{tikzcd}
                (C;p_1,\dots,p_n)
            \arrow[dd,"\sim","\varphi"']
            \arrow[dr,"f"]
                                                \\& X
                \\(C;p'_1,\dots,p'_n)
            \arrow[ur,"g"']
        \end{tikzcd}
    \end{center}
\end{defin}
The condition (iii) in the Definition \ref{HGW:def_stablemap} is equivalent to the condition that the stable map $\left( f: C \to X \right)$ has a finite amount of automorphisms. We will thus refer to these conditions as \textit{stability conditions}.

\begin{defin}
    Let $d \in H_2(X,\mathbb{Z})$. We say that the stable map $\left( f:(C;\underline{p}) \to X \right)$ has class $d$ if the fundamental class $[C] \in H_2(C;\mathbb{Z})$ satisfies $f_*([C]) = d$.
\end{defin}

\begin{example}
    Let us give an example of a stable map to $\mathbb{P}^2$ of genus 1 and degree 3. We consider the curve $C$ with three irreducible components $C_1, C_2, C_3$, with $g(C_1)=g(C_2)=0, g(C_3)=1$. We consider the maps
    \[
            f_{|C_1} : \left|
            \begin{tikzcd}[cramped,row sep = 0cm]
                C_1 \simeq \mathbb{P}^1
                    \arrow[r]
                &\mathbb{P}^2\\
                \left[ x : y \right]
                    \arrow[r, mapsto]
                &\left[ x^3 : x y^2 : y^3 \right]
            \end{tikzcd}\right.
    \]
    And we take $f_{|C_2}, f_{|C_3}$ to be constant equal to $[0:0:1] \in \mathbb{P}^2$.
    
    To make the map $f$ stable, we need $C_2$ to contain one marking since it has already two nodal points of $C$, and we need $C_3$ to contain one marking.
\end{example}

\subsection{Moduli space of stable maps}

\begin{defin}
    Let $S$ be a scheme over $\mathbb{C}$. A \textit{family of stable maps over $S$} is a flat proper morphism $\pi : \mathcal{C} \to S$ with $n$ sections $s_1,\dots,s_n$ and a map $f: \mathcal{C} \to X$ such that for all point $t \in S$, denoting $\mathcal{C}_t = \pi^{-1}(t)$, the map $\left[f_{|\mathcal{C}_t} : (\mathcal{C}_t, s_1(t), \dots, s_n(t)) \to X \right]$ is a stable map.
\end{defin}

\begin{defin}[Moduli space of stable curves]\label{HGW:def_mod_sp}
    Let $X$ be a complex projective variety and fix the parameters $g,n \in \mathbb{Z}_{\geq 0}, d\in H_2(X,\mathbb{Z})$.
    We denote by $\overline{\mathcal{M}}_{g,n}(X,d)$ the contravariant functor $\overline{\mathcal{M}}_{g,n}(X,d) : \left(\textnormal{Schemes over }\mathbb{C}\right) \to (\textnormal{Groupoids})$
    which sends the $\mathbb{C}$-scheme $S$ to the isomorphism class $[\pi : \mathcal{C} \to S]$ of family of stable maps over $S$ of genus $g$, with $n$ markings and of degree $d$.
    This functor is called the \textit{moduli space of stable maps}.
\end{defin}

\begin{thm}[\cite{KontEnum}]
    Let $X$ be a complex projective variety and fix the parameters $g,n \in \mathbb{Z}_{\geq 0}, d\in H_2(X,\mathbb{Z})$. The functor $\overline{\mathcal{M}}_{g,n}(X,d)$ is an algebraic Deligne--Mumford stack over $\mathbb{C}$ which is proper. 
\end{thm}
In general, this space is not equidimensional, nonetheless, these moduli spaces have a \textit{virtual fundamental class} $\left[ \overline{\mathcal{M}}_{g,n}(X,d) \right]^\text{vir}$ (see \cite{BF_intnormalcone,Behrend_Manin_stable_maps_stack}) and a \textit{virtual structure sheaf} $\mathcal{O} ^ {\text{vir}}$ (see \cite{Lee_qk}). These two virtual objects satisfy a collection of properties called the Behrend--Manin axioms, see \cite{Behrend_Manin_stable_maps_stack} for the virtual class, \cite{Lee_qk} for the virtual sheaf.



\begin{remark}
    The Deligne--Mumford stack $\overline{\mathcal{M}}_{g,n}(X,d)$ has virtual dimension
    \[
        \textnormal{vdim}_\mathbb{C}
        \left( \overline{\mathcal{M}}_{g,n}(X,d) \right) 
        = 
        (1-g)(\textnormal{dim}(X)-3)+n-\int_X c_1(TX)
    \]
    This number can be reobtained by checking the deformation theory of a stable map and using the Hirzebruch--Riemann--Roch formula, see \cite{CK_book}, 7.14.
\end{remark}

Let us give a class of varieties $X$ for which the genus 0 moduli spaces are well behaved.




\begin{prop}[\cite{FP_qh_notes}, Theorem 2]\label{MSM:prop_if_convex_then_smooth}
    If $X$ is a homogeneous space (e.g. for $X=\mathbb{P}^N$), then the moduli spaces $\overline{\mathcal{M}}_{0,n}(X,d)$ are smooth stacks.
\end{prop}

\section{Towards Gromov--Witten classes}

In this section we will give some additional constructions on the moduli of stable maps to construct Gromov--Witten invariants.
Our main objective is to construct the evaluation maps $\textnormal{ev}_i$ and some line bundles $\mathcal{L}_i$.
For completeness, we also give some properties satisfied by the virtual fundamental class which we may use in the next chapter.

\subsection{Universal curve}

\begin{defin}\label{MSM:def_ev_forget_maps}
    Let $i \in \{ 1, \dots, n\}$.
    \begin{itemize}
        \item The $i^\textnormal{th}$ \textit{evaluation map} is the map (in the category of Deligne--Mumford stacks)
        \[
            \functiondesc
            {\textnormal{ev}_i}
            {\overline{\mathcal{M}}_{g,n}(X,d)}
            {X}
            {\left[f:(C;p_0, \dots, p_n) \to X \right]}
            {f(p_i)}
        \]
        \item We denote by $\pi_i$ the map which forgets the $i^\textnormal{th}$ marking then contracts the eventually unstable components :
        \[
            \functiondesc
            {\pi_i}
            {\overline{\mathcal{M}}_{g,n}(X,d)}
            {\overline{\mathcal{M}}_{g,n-1}(X,d)}
            {\left[f:(C;p_1,\dots,p_n) \to X \right]}
            {\left[f:(C;p_1,\dots,\widehat{p_i},\dots,p_n) \to X \right]^\textnormal{(stabilised)}}
        \]
        Where $\widehat{\cdot}$ means that element is omitted. We explain what stabilised means below.
    \end{itemize}
\end{defin}

If a $n$-pointed morphism $f : (C;p_1, \dots, p_n) \to X$ satisfies the properties (i) and (ii) but not (iii), it is possible to stabilise this morphism into a stable map by contracting the irreducible components of $C$ on which $f$ fails the stability conditions, eventually killing some markings. We refer to the resulting stable map as the \textit{stabilised map}.

\begin{prop}\label{MSM:rmk_comparison_virtual_class}\textbf{}
    \textit{(i)}    \quad
    The forget map $\pi_{n+1}: \overline{\mathcal{M}}_{g,n+1}(X,d) \to \overline{\mathcal{M}}_{g,n}(X,d) $ is the universal moduli space, i.e. for any family of stable $\mathcal{C}$ over a scheme $S$, we have the cartesian diagram
    \begin{center}
        \begin{tikzcd}
            \mathcal{C}
                \arrow[r]
                \arrow[d]
            &
            \overline{\mathcal{M}}_{g,n+1}(X,d)
                \arrow[d, "\pi_{n+1}"]
            \\
            S
                \arrow[r]
            &
            \overline{\mathcal{M}}_{g,n}(X,d)
        \end{tikzcd}
    \end{center}
    In particular, if $f:(C;p_1,\dots,p_n) \to X$ is a stable map, let $\mathcal{C}=C$ , $S=pt$ and the bottom map is $pt \mapsto \left[f:(C;p_1,\dots,p_n) \to X \right]$ the fibres of $\pi_{n+1}$ satisfy
    \[
        \pi_{n+1}^{-1} (\left[f:(C;p_1,\dots,p_n) \to X \right])
        \simeq
        C / \textnormal{Aut}(f)
    \]
    
    \textit{(ii)} \quad
    The virtual fundamental class satisfies 
    \[
    \left[ \overline{\mathcal{M}}_{g,n+1}(X,d) \right]^\text{vir} = \pi_{n+1}^* \left[ \overline{\mathcal{M}}_{g,n+1}(X,d) \right]^\text{vir}
    \]
\end{prop}

\begin{defin}
    Let $i \in \{ 1, \dots, n\}$.
    The $i^\textit{th}$ \textit{tautological section }
    $s_i : \overline{\mathcal{M}}_{g,n}(X,d) \to \overline{\mathcal{M}}_{g,n+1}(X,d)$
    is the section of the universal moduli space which takes a a stable map 
    $\left[f:(C;\underline{p}) \to X \right] \in \overline{\mathcal{M}}_{g,n}(X,d)$ and sends it to the class in $\overline{\mathcal{M}}_{g,n+1}(X,d)$ defined by replacing the $i^\textnormal{th}$ marked point $p_i$ with an irreducible component $C_i \simeq \mathbb{P}^1$ with two markings $p_i', p_{n+1}'$, and defining $f_{|C_i} := f(p_i)$.
\end{defin}

\begin{defin}
    Let $i \in \{ 1, \dots, n\}$.
    Let $\omega_{n+1}$ be the relative dualizing sheaf of the universal curve 
    $\pi_{n+1} : \overline{\mathcal{M}}_{g,n+1}(X,d) \to \overline{\mathcal{M}}_{g,n}(X,d)$.
    The \textit{cotangent line bundle at the $i^\textit{th}$ marking} $\mathcal{L}_i \to \overline{\mathcal{M}}_{g,n}(X,d)$ is the line bundle defined by $s_i^* \omega_{n+1}$.
\end{defin}
The fibres of this bundle at the point $[f : (C;p_1,\dots,p_n)\to X] \in \overline{\mathcal{M}}_{g,n}(X,d)$ is given by the cotangent space at the $i^\textit{th}$ marked point $T^*_{p_i} C$.

\subsection{Gluing stable maps}

In this subsection, we detail what happens when we try to glue two stable maps together. This leads us to a technical property satisfied by the virtual fundamental class. We suggest to the reader that is unfamiliar with the quantum $\mathcal{D}$-module to skip this subsection in a first lecture and assume instead that the quantum product is associative (cf. Proposition \ref{HGW:prop_qprod_properties}).

First, we give a scheme theoretic heuristic to construct a gluing map
\[
    \varphi : \overline{\mathcal{M}}_{g_1,n_1+1}(X,d_1)
    \times_{_{\underline{X}}}
    \overline{\mathcal{M}}_{g_2,n_2+1}(X,d_2)
    \to \overline{\mathcal{M}}_{g,n}(X,d)
\]
This map takes two stable maps $[f_1: (C_1;p_1, \dots, p_{n_1}, a) \to X] \in \overline{\mathcal{M}}_{g_1,n_1+1}(X,d_1)$, $[f_2: (C_2; b,p'_1, \dots, p'_{n_2}) \to X] \in \overline{\mathcal{M}}_{g_2,n_2+1}(X,d_2)$
and sends them to the stabilised map given by the gluing together the markings $a \in C_1, b \in C_2$.
This results in a new stable map of genus $g_1+g_2$, with $n_1+n_2$ markings and of degree $d_1 + d_2$.
    
For this application to be well defined, the last point $a \in C_1$ of the first stable map and the first point $b \in C_2$ of the second stable map to satisfy $f_1(a)=f_2(b) \in X$.
Therefore, the maps to the base $\underline{X}$ in the fibre product are the maps 
$\textnormal{ev}_{n_1+1} : \overline{\mathcal{M}}_{g_1,n_1+1}(X,d_1) \to \underline{X}$ and 
$\textnormal{ev}_1 : \overline{\mathcal{M}}_{g_2,n_2+1}(X,d_2) \to \underline{X}$, as in the diagram
\begin{center}
    \begin{tikzcd}
        \overline{\mathcal{M}}_{g_1,n_1+1}(X,d_1)
        \times_{_{\underline{X}}}
        \overline{\mathcal{M}}_{g_2,n_2+1}(X,d_2)
            \arrow[r]
            \arrow[d]
        &
        \overline{\mathcal{M}}_{g_1,n_1+1}(X,d_1)
            \arrow[d,"\textnormal{ev}_{n_1+1}"]
        \\  
        \overline{\mathcal{M}}_{g_2,n_2+1}(X,d_2)
            \arrow[r,"\textnormal{ev}_1"]
        &
        \underline{X}
    \end{tikzcd}
\end{center}

There is also a second Cartesian diagram we should consider, which is obtained by using the gluing of stable curves instead of trying to glue stable maps.
We denote by $g : \overline{\mathcal{M}}_{g_1,n_1+1} \times \overline{\mathcal{M}}_{g_2,n_2+1} \to \overline{\mathcal{M}}_{g,n}$ the map which glues two stable curves as described above.
We have
\begin{center}
    \begin{tikzcd}
        \mathcal{Z}_d
            \arrow[r]
            \arrow[d]
        &
        \overline{\mathcal{M}}_{g,n}(X,d)
            \arrow[d,"\textnormal{Stab}"]
        \\  
        \overline{\mathcal{M}}_{g_1,n_1+1} \times \overline{\mathcal{M}}_{g_2,n_2+1}
            \arrow[r,"g"]
        &
        \overline{\mathcal{M}}_{g,n}
    \end{tikzcd}
\end{center}

\begin{prop}[\cite{Li_Tian_Virtual_in_AG}, Theorem 5.2]
    Let
    $
        \mathcal{Z}_{d_1,d_2}=
        \overline{\mathcal{M}}_{g_1,n_1+1}(X,d_1)
        \times_{_{\underline{X}}}
        \overline{\mathcal{M}}_{g_2,n_2+1}(X,d_2)
    $.
    There exists a canonical morphism
    \[
        \Psi : \bigcup_{d_1 + d_2 = d} \mathcal{Z}_{d_1,d_2} \to \mathcal{Z}_d
    \]
    which is proper, finite and dominant.
    Furthermore, we have the Cartesian diagram
    \begin{center}
        \begin{tikzcd}
            \mathcal{Z}_{d_1,d_2}
                \arrow[r]
                \arrow[d]
            &
            \overline{\mathcal{M}}_{g_1,n_1+1}(X,d_1) \times \overline{\mathcal{M}}_{g_2,n_2+1}(X,d_2)
                \arrow[d,
                "\textnormal{ev}_{n+1} \times \textnormal{ev}_1"]
            \\  
            \underline{X}
                \arrow[r,"\Delta"]
            &
            \underline{X} \times \underline{X}
        \end{tikzcd}
    \end{center}
    We define the class
    \[
        \left[
            \bigcup_{d_1 + d_2 = d} \mathcal{Z}_{d_1,d_2}
        \right]^\textnormal{vir}
        =
        \sum_{d_1+d_2=d} \Delta^!
        \left[
            \overline{\mathcal{M}}_{g_1,n_1+1}(X,d_1)
        \right]^\textnormal{vir}
        \otimes
        \left[
            \overline{\mathcal{M}}_{g_2,n_2+1}(X,d_2)
        \right]^\textnormal{vir}
    \]
    It satisfies the identity
    \[
        \Psi_*
        \left[
            \bigcup_{d_1 + d_2 = d} \mathcal{Z}_{d_1,d_2}
        \right]^\textnormal{vir}
        =
        g^! \left[ \overline{\mathcal{M}}_{g,n}(X,d) \right]^\textnormal{vir}
    \]
\end{prop}

%
%
%

\chapter{Cohomological Gromov--Witten Invariants}\label{chapter:H_GW}

%
%
%

In this chapter, we recall the definition of Gromov--Witten invariants and their main properties in Section \ref{HGW:section_coh_gw}. The next section is dedicated to how these invariants are encoded in a differential module. We define the quantum $\mathcal{D}$-module and study its properties. Then, we construct a fundamental solution and define Givental's $J$-function.

\section{Cohomological Gromov--Witten invariants}\label{HGW:section_coh_gw}

\subsection{Definitions}

\begin{notation}
    We will denote by $\mathds{1} \in H^0(X; \mathbb{Q})$ the unit in the cohomology ring of $X$. 
    We also recall that the virtual cycle $\left[ \overline{\mathcal{M}}_{g,n}(X,d) \right]^\text{vir}$ has homological degree $2(1-g)(\textnormal{dim}(X)-3)+2n+2\int_d c_1(TX)$.
\end{notation}

\begin{defin}\label{HGW:def_GWI}
    Let $g,n \in \mathbb{Z}_{\geq 0}$, $d \in H_2(X;\mathbb{Z})$. Let $k_1, \dots, k_n \in \mathbb{Z}_{\geq 0}$ be some integers, and let $\alpha_1, \dots \alpha_n \in H^*(X; \mathbb{Q})$. 
    We also introduce the cohomological class $\psi_i := c_1 \left( \mathcal{L}_i \right) \in H^2\left(\overline{\mathcal{M}}_{g,n}(X,d) ; \mathbb{Q} \right)$.
    The associated \textit{Gromov--Witten invariant} is defined by the following intersection formula in the moduli space of stable maps
    \[
        \langle \psi_1^{k_1} \alpha_1 , \dots , \psi_n^{k_n} \alpha_n \rangle^\textnormal{coh}_{g,n,d}
        = \int_{\left[ \overline{\mathcal{M}}_{g,n}(X,d) \right]^\text{vir}} \bigcup_i \left( \psi_i^{k_i} \cup \text{ev$_i^\star$}(\alpha_i)\right) \in \mathbb{Q}
    \]
    We extend this definition to $k_i \in \mathbb{Z}_{<0}$ by setting the invariant to be equal to zero if one of the $k_i$ is negative.
\end{defin}
For intersection on algebraic stacks, we refer to \cite{Vistoli_Intthy_stacks}. Note that in the left hand side, the expression $ \psi_i^{k_i} \alpha_i$ should not be considered as a product but merely as a notation, since these two classes live on different spaces - which are respectively $H^*(\overline{\mathcal{M}}_{g,n}(X,d);\mathbb{Q})$ and $H^*(X;\mathbb{Q})$.

If one integer $k_i$ is zero, we will replace $\psi_{k_i}\alpha_i$ by $\alpha_i$. We may also shorten the insertion $\psi_{k_i} \mathds{1}, k_i \neq 0$ by $\psi_{k_i}$.

The assertion that the invariant $\langle \psi_1^{k_1} \alpha_1 , \dots , \psi_n^{k_n} \alpha_n \rangle^\textnormal{coh}_{g,n,d}$ is only a rational and not an integer is a consequence that the moduli space $\overline{\mathcal{M}}_{g,n}(X,d)$ is a Deligne--Mumford stack and not always a scheme (e.g. in \cite{CK_book}, Example 10.1.3.3.).
The choice of taking the cohomology with rational coefficients is also the trace of the use of orbifold cohomology.
We will give an example of a non integer, non positive Gromov--Witten invariant below.
We also point that the cohomological class inside the integral has degree $\sum_i 2k_i+\textnormal{deg}(\alpha_i)$.


\subsection{Properties of cohomological Gromov--Witten invariants}\label{subsection:HGW_props}

In this subsection we give a list of some important properties of the Gromov--Witten invariants. We recommend to the reader mainly interested in the construction of the quantum $\mathcal{D}$-module to skip this subsection on the first read.

We refer to \cite[Sections 7.3 and 10.1.2]{gathmann_habilitation, CK_book} for detailed proofs of these statements.
We begin by giving some properties useful for computations.

\begin{prop}[Degree Axiom]\label{HGW:prop_degree_axiom}
    The Gromov--Witten invariant $\langle \psi_1^{k_1} \alpha_1 , \dots , \psi_n^{k_n} \alpha_n \rangle_{g,n,d}^\textnormal{coh}$ is zero unless the following integer condition is verified:
    \[
        \sum_i 2k_i+\textnormal{deg}(\alpha_i) = 2(1-g)(\textnormal{dim}(X)-3)+2n-2\int_X c_1(TX)
    \]
\end{prop}

\begin{prop}[Fundamental Class Axiom, \cite{gathmann_habilitation}, 1.3.3]\label{HGW:prop_string_eqn}
    This property can also be referred to as the String Equation. 
    Let $g,n,d$ such that $n+2g \geq 4$ or $n>1, d \neq 0$.
    If one of the insertions in a Gromov--Witten invariant is the unit in cohomology, we have
    \[
        \left\langle \psi_1^{k_1} \alpha_1 , \dots , \psi_{n-1}^{k_{n-1}} \alpha_{n-1}, \mathds{1} \right\rangle^\textnormal{coh}_{g,n,d}
        =
        \sum_{i=1}^{n-1} \langle \psi_1^{k_1} \alpha_1 , \dots , \psi_i^{k_i-1} \alpha_i , \dots, \psi_{n-1}^{k_{n-1}} \alpha_{n-1} \rangle^\textnormal{coh}_{g,n,d}
    \]
\end{prop}


\begin{prop}[Dilaton Axiom, \cite{gathmann_habilitation}, 1.3.4]\label{HGW:prop_dilaton_axiom}
    If one insertion in a Gromov--Witten invariant is $\psi^1 \mathds{1}$, we have
    \[
        \left\langle \psi_1^{k_1} \alpha_1 , \dots , \psi_{n-1}^{k_{n-1}} \alpha_{n-1}, \psi_n^1 \right\rangle^\textnormal{coh}_{g,n,d}
        =
        (2g-2+n) \left\langle \psi_1^{k_1} \alpha_1 , \dots , \psi_{n-1}^{k_{n-1}} \alpha_{n-1} \right\rangle^\textnormal{coh}_{g,n-1,d}
    \]
\end{prop}


\begin{prop}[Point Mapping Axiom, \cite{CK_book}, 7.35]\label{HGW:prop_point_mapping_axiom}
    In this proposition, we assume that both the genus and the degree of the stable curve are zero. We have
    \[
        \langle \alpha_1 , \dots , \alpha_n \rangle^\textnormal{coh}_{0,n,0}
        =
        \left\{
            \begin{aligned}
                \int_X  \alpha_1 \cup &\alpha_2 \cup \alpha_3    && \textnormal{if } n=3 \\
                &0   &&  \textnormal{otherwise.}
            \end{aligned}
        \right.
    \]
\end{prop}


\begin{prop}[Divisor Axiom, \cite{gathmann_habilitation}, 1.3.4]\label{HGW:prop_divisor_axiom}
    Let $g,n,d$ such that $n+2g \geq 4$ or $n>1, d \neq 0$.
    If one of the insertions in a Gromov--Witten invariant is the class of a divisor $D \in H^2(X;\mathbb{Q})$, then
    \begin{align*}
        \left\langle \psi_1^{k_1} \alpha_1 , \dots , \psi_{n-1}^{k_{n-1}} \alpha_{n-1}, D \right\rangle^\textnormal{coh}_{g,n,d}
        =
        &\left(
            \int_d D
        \right)
        \left\langle \psi_1^{k_1} \alpha_1 , \dots , \psi_{n-1}^{k_{n-1}} \alpha_{n-1} \right\rangle^\textnormal{coh}_{g,n-1,d} \\
        &+
        \sum_{j=1}^{n-1} \langle \psi_1^{k_1} \alpha_1 , \dots , \psi_i^{k_i-1} (D \cup \alpha_i) , \dots, \psi_{n-1}^{k_{n-1}} \alpha_{n-1} \rangle^\textnormal{coh}_{g,n-1,d}
    \end{align*}
\end{prop}

\section{Quantum cohomology and the quantum $\mathcal{D}$-module}\label{hgw:section_QH_and_QDM}

\subsection{Quantum cohomology}

In this subsection, we will follow \cite{CK_book}, Section 8.2.
We recall that $X$ denotes a complex projective variety with even cohomology.
%
%
%

\begin{notation}
    Let $N$ be the dimension of $H^*(X;\mathbb{Q})$ and $1 \leq r \leq N$. We will use the following notation:
    \begin{align*}
        &   T_0, \dots T_N
        &&  \textnormal{A basis of } H^*(X;\mathbb{Q}) \textnormal{, so that } T_1, \dots, T_r \textnormal{ span } H^2(X;\mathbb{Q}) \textnormal{ and } T_{r+1}, \dots, T_N \in H^{\geq 4}(X; \mathbb{Q})\\
        &   g
        && \textnormal{the pairing on } H^*(X;\mathbb{Q}) \textnormal{ given by Poincaré duality , } g_{ij} := g(T_i,T_j) = \int_X T_i \cup T_j\\
        &   (g^{ij})_{i,j}
        &&  \textnormal{the inverse of the Gram matrix of the metric } g\\
        &   T^i
        &&  \textnormal{the metric dual of } T_i \textnormal{ with respect to the metric } g, \, T^i := \sum_j g^{ij} T_j \\
        &   t_i
        &&  \textnormal{the coordinate associated to } T_i\\
        &   \tau
        &&  \textnormal{an arbitrary point in } H^*(X;\mathbb{Q}), \, \tau := \sum_i t_i T_i\\
        &   \tau_2(d)
        &&  \textnormal{for } d \in H_2(X;\mathbb{Z}), \tau_2 \in H^2(X;\mathbb{Q}) \textnormal{ the integral } \tau_2(d) := \int_d \tau_2\\
        & Q_1, \dots, Q_r
        && \textnormal{Novikov variables associated to } T_1, \dots, T_r \textnormal{. If }d \in H_2(X;\mathbb{Z}) \textnormal{, we also define } Q^d := Q_1^{T_1(d)} \cdots Q_r^{T_r(d)}
    \end{align*}
    
    If a definition depends on the coordinates $t_0, \dots, t_N$, we may say it depends on the symbol $\tau$ instead, e.g. we will denote the ring $H^*(X;\mathbb{Q}) \otimes \mathbb{C}\left[t_0, \dots, t_N \right]$ by $H^*(X;\mathbb{Q}) \otimes \mathbb{C}[\tau]$.
    
    The Novikov ring $\mathbb{C}[\![Q]\!]$ is defined by the ring of formal series in $Q_1, \dots, Q_r$ :
    \[
        \mathbb{C}[\![Q]\!] = \left\{
            \sum_{d \in H_2(X;\mathbb{Z})} f_d Q^d \, \middle| \, f_d \in \mathbb{C}
        \right\}
    \]
\end{notation}

\begin{defin}\label{HGW:def_potential}
    The \textit{genus zero Gromov--Witten potential} $\mathcal{F}$ is the formal power series defined by
    \[
        \mathcal{F}(\tau,Q) = \sum_{\substack{n \geq 0 \\ d \in H_2(X;\mathbb{Z})}} \frac{1}{n!} \langle \tau, \dots, \tau \rangle^\textnormal{coh}_{0,n,d} Q^d
        \in 
        \mathbb{Q}[ \![\tau] \! ] \otimes \mathbb{C}[\![Q]\!]
    \]
\end{defin}

\begin{remark}\label{HGW:rmk_large_radius_limit}
    Let $\tau' = t_0 T_0 + \sum_{i>r} t_i T_i$ and $\tau_2 = \sum_{j=1}^r t_j T_j$.
    Because of the Divisor Axiom (see Proposition \ref{HGW:prop_divisor_axiom}), we have
    \[
        \mathcal{F}(\tau,Q) = \sum_{d,n \geq 0} \frac{1}{n!} \langle \tau', \dots, \tau' \rangle^\textnormal{coh}_{0,n,d} e^{\tau_2(d)} Q^d
    \]
    When looking at Definition \ref{HGW:def_potential}, we may be tempted to see the genus zero potential as some local formal data defined near $t=0$. Through this new identity, we see that the potential also defined near the point $\tau'=0, \mathfrak{Re}(\tau_2(d)) \to -\infty$. This point is called the \textit{large radius limit}.
\end{remark}

\begin{remark}\label{HGW:lemma_potential_derivative}
    By the linearity of the Gromov--Witten invariants, we have
    \[
    \partial_{t_i} \mathcal{F}(\tau,Q) 
    =
    \sum_{\substack{d \in H_2(X;\mathbb{Z}) \\ n \geq 0}} \frac{1}{n!} \langle T_i, \tau, \dots, \tau \rangle^\textnormal{coh}_{0,n,d} Q^d
    \]
\end{remark}

\begin{example}\label{HGW:ex_potential_P2}
    For $X = \mathbb{P}^2$, let $H = c_1(\mathcal{O}(1)) \in H^2(\mathbb{P}^2 ; \mathbb{Q})$ be the hyperplane class and $[l] \in H_2(X;\mathbb{Z})$ be the class of a line.
    Consider the Gromov--Witten invariant
    \[
    N_d = \left \langle H^2, \dots, H^2 \right \rangle^\textnormal{coh}_{0,3d-1,d[l]}
    \]
    Notice that the class $H^2 \in H^4(\mathbb{P}^2 ; \mathbb{Q})$ corresponds the Poincaré dual of the class of a point $[pt] \in H_0(X;\mathbb{Q})$. We are going to show that the Gromov--Witten potential of $\mathbb{P}^2$ is given by
    \[
        \mathcal{F}(t_0,t_1,t_2,Q) = \frac{1}{2} (t_0 t_1^2 + t_0^2 t_2) + \sum_{d=1}^\infty N_d \frac{t_2^{3d-1}}{(3d-1)!} e^{dt_1}Q^d
    \]
\end{example}

In the expression above, the right hand side is written in two parts: first we have a degree 3 polynomial, then we get a power series in $e^{t_1}Q$.
The first part will be obtained by the Point Mapping Axiom, while the second comes from Divisor Axiom. We follow the strategy of \cite{CK_book}, Subsection 8.3.2.

\begin{proof}
    To make our expressions more compact, we will be denoting the Gromov--Written invariant whose entries are the classes $T_0, T_1, T_2$ appearing respectively $\alpha_0,\alpha_1,\alpha_2$ times by $\left\langle
        T_0^{(\alpha_0)}, T_1^{(\alpha_1)}, T_2^{(\alpha_2)}
    \right\rangle^\textnormal{coh}_{0,\alpha_0+\alpha_1+\alpha_2,d[l]}$.
    Since the classes $d [l]$ are effective if and only if $d \geq 0$, we have
    \[
        \mathcal{F}(t_0,t_1,t_2,Q)
        =
        \sum_{d \geq 0}
        \sum_{n=0}^\infty
        \sum_{\alpha_0 + \alpha_1 + \alpha_2 = n}
        \frac{t_0^{\alpha_0}}{\alpha_0!} \frac{t_1^{\alpha_1}}{\alpha_1!} \frac{t_2^{\alpha_2}}{\alpha_2!} 
        \left\langle
            T_0^{(\alpha_0)}, T_1^{(\alpha_1)} ,T_2^{(\alpha_2)}
        \right\rangle^\textnormal{coh}_{0,n,d[l]}
    \]
    Let us compute the zero degree $d=0$ part of the potential. Using the Point Mapping Axiom (see \ref{HGW:prop_point_mapping_axiom}), the degree $d=0$ part is given by
    \[
        \sum_{\alpha_0 + \alpha_1 + \alpha_2 = 3}
        \frac{t_0^{\alpha_0}}{\alpha_0!} \cdots \frac{t_2^{\alpha_2}}{\alpha_2!} \int_{\mathbb{P}^2} H^{\alpha_1 + 2 \alpha_2}
        =
        \frac{1}{2} \left(t_0 t_1^2 + t_0^2 t_2 \right)
    \]
    For the positive degree $d>0$ part of the potential, by String Equation (see \ref{HGW:prop_string_eqn}) we immediately set $\alpha_0=0$.
    Using the Divisor Axiom (see \ref{HGW:prop_divisor_axiom}), we have
    \[
        \left\langle
            T_1 ^ {(\alpha_1)}, T_2 ^ {(\alpha_2)}
        \right\rangle^\textnormal{coh}_{0,\alpha_1 + \alpha_2,d[l]}
        =
        d^{\alpha_1} \left\langle
            T_2 ^ {(\alpha_2)}
        \right\rangle^\textnormal{coh}_{0,\alpha_2,d[l]}
    \]
    By the Degree Axiom (\ref{HGW:prop_degree_axiom}), the invariant in the right hand side is non zero if and only if $\alpha_2 = 3d-1$.
    Therefore, recalling that $N_d = \langle T_2^{(3d-1)} \rangle^\textnormal{coh}_{0,3d-1,d[l]}$, we get
    \[
        \mathcal{F}(t_0,t_1,t_2,Q)
        =
        \frac{1}{2} \left(t_0 t_1^2 + t_0^2 t_2 \right)
        +
        \sum_{d > 0} \sum_{n \geq 0} \sum_{\alpha_1+3d-1=n}
        N_d \frac{(d \, t_1)^{\alpha_1}}{{\alpha_1}!} \frac{t_2^{3d-1}}{(3d-1)!} Q^d
    \]
    Rearranging the sums, the variable $n$ being mute, we make an exponential appear and obtain
    \[
        \mathcal{F}(t_0,t_1,t_2,Q) = \frac{1}{2} (t_0 t_1^2 + t_0^2 t_2) + \sum_{d=1}^\infty N_d \frac{t_2^{3d-1}}{(3d-1)!} e^{dt_1}Q^d
    \]
\end{proof}

\begin{defin}\label{HGW:def_qprod}
    The \textit{quantum product} $\bullet_\tau$ is a product on $H^*(X;\mathbb{Q})[ \![\tau] \! ] [\![Q]\!]$ defined on basis elements by
    \[
        T_i \bullet_\tau T_j = \sum_k \partial_{t_i} \partial_{t_j} \partial_{t_k} \mathcal{F}(\tau,Q) T^k
    \]
    We then extend bilinearly this product to $H^*(X;\mathbb{Q}) \otimes \mathbb{C}[\![t_0,\dots,t_N,Q_1,\dots,Q_r]\!]$.
    
    The ring $(H^*(X;\mathbb{Q}) \otimes \mathbb{C}[\![\tau ]\!][\![Q]\!],\bullet_\tau)$ formed by this product will be called \textit{quantum cohomology}, denoted by $QH^*(X)$.
\end{defin}

\begin{remark}\label{HGW:rmk_product_metric_frobenius}
    The quantum product $\bullet_\tau$ satisfies a compatibility relation with the metric $g$
    \[
        g(T_i \bullet_\tau T_j, T_k) = g(T_i, T_j \bullet_\tau T_k) = \partial_{t_i} \partial_{t_j} \partial_{t_k} \mathcal{F}(\tau,Q)
    \]
    Notice also that the relation $g(T_i, T_j \bullet_\tau T_k) = \partial_{t_i} \partial_{t_j} \partial_{t_k} \mathcal{F}(\tau,Q)$ can be taken as a definition of the quantum product.
\end{remark}

\begin{remark}
    The quantum product $\bullet_\tau$ can be seen as a deformation of the usual product $\cup$ on the cohomology, indexed by the parameters $t_0, \dots, t_N$. Notice that if we set $Q=0$, in the power series $\mathcal{F}$ remains only the terms corresponding to $d=0$. Then, by the Point Mapping Axiom (see Proposition \ref{HGW:prop_point_mapping_axiom}),
    \[
        \left(  T_i \bullet_\tau T_j  \right)_{|Q=0} = T_i \cup T_j
    \]
\end{remark}

\begin{prop}\label{HGW:prop_qprod_properties}
    The product $\bullet_\tau$ is unitary, commutative and associative. The unit of this product is the unit in cohomology, $\mathds{1} \in H^*(X;\mathbb{Q}) \otimes \mathbb{C}[\![\tau ]\!][\![Q]\!]$.
\end{prop}

Before giving a proof, we mention that the associativity of the quantum product can be translated into a property satisfied by the Gromov--Witten potential, which appears in the theorem below.

\begin{thm}[\cite{Konts_Manin_CohFT}, Theorem-Definition 4.5]\label{HGW:thm_potential_WDVV}
    The genus zero Gromov--Witten potential $\mathcal{F}$ satisfies the set of differential equations, called WDVV equations (Witten--Dijkgraaf--Verlinde--Verlinde), indexed by $i,j,k,l \in \{0,\dots,N\}$:
    \[
        \sum_{a,b = 0}^N
        \frac{\partial^3 \mathcal{F}}{\partial t_i \partial t_j \partial t_a}
        g^{ab}
        \frac{\partial^3 \mathcal{F}}{\partial t_b \partial t_k \partial t_l}
        =
        \sum_{a,b = 0}^N
        \frac{\partial^3 \mathcal{F}}{\partial t_j \partial t_k \partial t_a}
        g^{ab}
        \frac{\partial^3 \mathcal{F}}{\partial t_b \partial t_i \partial t_l}
    \]
    \qed
\end{thm}

\begin{proof}[Proof of Proposition \ref{HGW:prop_qprod_properties}]
\textbf{(i) The quantum product is unitary.}
We are going to show that the unit in cohomology $\mathds{1} \in H^0(X; \mathbb{Q})$ is also the unit of the quantum product.
Let $j \in \{ 0, \dots, N \}$. We have
\[
    \mathds{1} \bullet_\tau T_j
    =
    \sum_{k=0}^N \partial_{t_0} \partial_{t_j} \partial_{t_k} \mathcal{F}(\tau,Q) T^k
    =
    \sum_{k=0}^N \sum_{d,n \geq 0} \frac{1}{n!} \langle \mathds{1}, T_j, T_k, \tau, \dots, \tau \rangle^\textnormal{coh}_{0,n,d} T^k Q^d 
\]
By the String Equation (see \ref{HGW:prop_string_eqn}), all the Gromov--Witten invariants in the rightmost expression are zero unless $d=0$.
In this case, we can apply the Point Mapping Axiom (see \ref{HGW:prop_point_mapping_axiom}), which means we only consider the invariants with parameters $d=0, n=3$. Finally, since the classes $T^k$ and $T_k$ are dual with respect to the Poincaré pairing, we get
\[
    \mathds{1} \bullet_\tau T_j = \sum_{k=0}^N T^k \int_X T_j \cup T_k = \sum_{k=0}^N g_{jk} T^k = T_j
\]

\textbf{(ii) The quantum product is commutative.}
This follows from the definition of the quantum product.
We have $T_i \bullet_\tau T_j = \sum_k \partial_{t_i} \partial_{t_j} \partial_{t_k} \mathcal{F}(\tau,Q) T^k$.
Since the derivatives commute, we have $\partial_{t_j} \partial_{t_i} \partial_{t_k} \mathcal{F}(\tau,Q) = \partial_{t_i} \partial_{t_j} \partial_{t_k} \mathcal{F}(\tau,Q)$.

\textbf{(iii) The quantum product is associative.}
Using $T^a = \sum_b g^{ab} T_b$, we get
\begin{align*}
    (T_i \bullet_\tau T_j) \bullet_\tau T_k
    &= \sum_{a,b,l} \partial_{t_i} \partial_{t_j} \partial_{t_a} \mathcal{F}(\tau) g^{ab} \partial_{t_b} \partial_{t_k} \partial_{t_l} \mathcal{F}(\tau) T^l
    \\
    T_i \bullet_\tau (T_j \bullet_\tau T_k)
    &= \sum_{a,b,l} \partial_{t_j} \partial_{t_k} \partial_{t_a} \mathcal{F}(\tau) g^{ab} \partial_{t_b} \partial_{t_i} \partial_{t_l} \mathcal{F}(\tau) T^l
\end{align*}
Using Theorem \ref{HGW:thm_potential_WDVV}, the two series on the right hand sides are equal.
\end{proof}

Continuing the example of $X = \mathbb{P}^2$, the WDVV equations allow us to recursively compute the coefficients of the Gromov-Witten potential.

\begin{example}\label{HGW:ex_P2_computation_Nd}
    Let $X=\mathbb{P}^2$. In Example \ref{HGW:ex_potential_P2}, we had showed that the Gromov-Witten potential is given by
    \[
        \mathcal{F}(t_0,t_1,t_2,Q) = \frac{1}{2} (t_0 t_1^2 + t_0^2 t_2) + \sum_{d=1}^\infty N_d \frac{t_2^{3d-1}}{(3d-1)!} e^{dt_1}Q^d
    \]
    For simplicity, we will set $Q=1$. We can do this because, if we consider the potential as a power series in $Q$, the ratio test implies it has infinite convergent ray.
    By the Theorem \ref{HGW:thm_potential_WDVV}, we have
        \[
        \sum_{a,b = 0}^2
        \frac{\partial^3 \mathcal{F}}{\partial t_i \partial t_j \partial t_a}(t_0,t_1,t_2)
        g^{ab}
        \frac{\partial^3 \mathcal{F}}{\partial t_b \partial t_k \partial t_l}(t_0,t_1,t_2)
        =
        \sum_{a,b = 0}^2
        \frac{\partial^3 \mathcal{F}}{\partial t_i \partial t_l \partial t_a}(t_0,t_1,t_2)
        g^{ab}
        \frac{\partial^3 \mathcal{F}}{\partial t_b \partial t_j \partial t_k}(t_0,t_1,t_2)
    \]
    We recall that the matrices $\left(g_{ij}\right), \left(g^{ij}\right)$ are given by
    \[
        \left(g_{ij}\right)
        =
        \left(g^{ij}\right)
        =
        \begin{pmatrix}
            0   &   0   &   1   \\
            0   &   1   &   0   \\
            1   &   0   &   0   
        \end{pmatrix}
    \]
    Using String Equation and Point Mapping Axiom, we also have (see proof of Proposition \ref{HGW:prop_qprod_properties})
    \[
        \partial_{t_0} \partial_{t_a} \partial_{t_b} \mathcal{F}(t_0,t_1,t_2,Q) = g_{ab}
    \]
    We choose to compute the values in the WDVV equations for a quadruplet $(i,j,k,l)$ such that the sums have non zero terms.
    We choose $(i,j,k,l)=(1,1,2,2)$.
    The corresponding WDVV equation reduces to
    \begin{equation}\label{HGW:eqn_WDVV_P2_reduced}
        \frac{\partial ^3 \mathcal{F}}{\partial t_2^3}
        +
        \frac{\partial^3 \mathcal{F}}{\partial t_1^3}
        \frac{\partial^3 \mathcal{F}}{\partial t_1 \partial t_2^2}
        =
        \left(
            \frac{\partial ^3 \mathcal{F}}{\partial t_1^2 \partial t_2}
        \right)^2
    \end{equation}
    We have
    \[
        \frac{\partial ^3 \mathcal{F}}{\partial t_2^3} = \sum_{d=2}^\infty N_d e^{dt_1} \frac{1}{(3d-4)!}t_2^{3d-4}
    \]
    Therefore, in (\ref{HGW:eqn_WDVV_P2_reduced}), we look for the coefficient in front of $e^{dt_1}t_2^{3d-4}$.
    For the left hand side, we find
    \[
        \frac{1}{(3d-4)!} N_d
        +
        \sum_{d_1+d_2=d} N_{d_1}  N_{d_2} d_1^3 d_2 \frac{1}{(3d_1-1)!} \frac{1}{(3d_2-3)!}
    \]
    And for the right hand side, we find
    \[
        \sum_{d_1+d_2=d} N_{d_1}  N_{d_2} d_1^2 d_2^2 \frac{1}{(3d_1-2)!} \frac{1}{(3d_2-2)!}
    \]
    Using (\ref{HGW:eqn_WDVV_P2_reduced}), we thus get
    \[
        N_d = (3d-4)! \sum_{d_1+d_2=d}
            N_{d_1}  N_{d_2} 
            \left(
                -d_1^3 d_2 \frac{1}{(3d_1-1)!} \frac{1}{(3d_2-3)!}
                +
                d_1^2 d_2^2 \frac{1}{(3d_1-2)!} \frac{1}{(3d_2-2)!}
            \right)
    \]
    Using $d_2=d-d_1$, this can be rewritten as
    \[
        N_d = \sum_{d_1+d_2=d} N_{d_1} N_{d_2} \left( \binom{3d-4}{3d_1-2} d_1^2d_2^2 - \binom{3d-4}{3d_1-1}d_1^3 d_2 \right)
    \]
    This allows us to compute the values of $N_d$ recursively.
    The first values of the sequence $(N_d)$ are (see OEIS' sequence \href{http://oeis.org/A013587}{A013587})
    \[
        1,1,12,620,87304,26312976,14616808192,13525751027392,\dots
    \]
\end{example}

Since we know the value of the Gromov--Witten potential for $X=\mathbb{P}^2$, we can also give the value of the quantum product.

\begin{example}
    Let $X= \mathbb{P}^2$ and $T_i = H^i$ for $i \in \{0,1,2\}$.
    For simplicity, we set $Q=1$ again.
    We know that $T_0$ is the unit of the quantum product. The other values of this product are given by
    \begin{align*}
        T_1 \bullet_\tau T_1
        &=
        \left(
            \sum_{d=1}^\infty N_d d^2 e^{dt_1} \frac{t_2^{3d-2}}{(3d-2)!}
        \right) T_0
        +
        \left(
            \sum_{d=1}^\infty N_d d^3 e^{dt_1} \frac{t_2^{3d-1}}{(3d-1)!}
        \right) T_1
        +
        T_2
        \\
        T_1 \bullet_\tau T_2
        &=
        \left(
            \sum_{d=1}^\infty N_d d e^{dt_1} \frac{t_2^{3d-3}}{(3d-3)!}
        \right) T_0
        +
        \left(
            \sum_{d=1}^\infty N_d d^2 e^{dt_1} \frac{t_2^{3d-2}}{(3d-2)!}
        \right) T_1
        \\
        T_2 \bullet_\tau T_2
        &=
        \left(
            \sum_{d=2}^\infty N_d e^{dt_1} \frac{t_2^{3d-4}}{(3d-4)!}
        \right) T_0
        +
        \left(
            \sum_{d=1}^\infty N_d d e^{dt_1} \frac{t_2^{3d-3}}{(3d-3)!}
        \right) T_1
    \end{align*}
\end{example}

As seen as the formulas above, the (big) quantum product $\bullet_\tau$ can be hard to compute. We introduce a second product, which is easier to compute and depends on less variables.

\begin{defin}
    Let $\tau' = t_0 T_0 + \sum_{i>r} t_i T_i$ and $\tau_2 = \sum_{j=1}^r t_j T_j$.
    The \textit{small quantum product} $\circ_{\tau_2}$ is a product on $H^*(X;\mathbb{Q})\left[t_1, \dots, t_r \right][\![Q]\!]$ defined by
    \[
        T_i \circ_{\tau_2} T_j = \left( T_i \bullet_\tau T_j \right)_{|\tau'=0}
    \]
    We denote by $SQH^*(X)$ the ring $(H^*(X;\mathbb{Q})[ \![\tau] \! ][\![Q]\!],\circ_{\tau_2})$, called \textit{small quantum cohomology}.
\end{defin}

\begin{example}\label{HGW:ex_sqh_projectivespace}
    For $X= \mathbb{P}^N$, 
    let $H=c_1\left( \mathcal{O}(1) \right) \in H^2 \left( \mathbb{P}^N \right)$ be the hyperplane class. We have 
    \[
        H^*(\mathbb{P}^N;\mathbb{Q}) \simeq \mathbb{Q}[H]\left/ \left(H^{N+1}\right) \right.
    \]
    Set $T_i = H^i$, so that $T^i = H^{N-i}$.
    We have
    \[
        SQH^* \left( \mathbb{P}^N \right)
        \simeq
        \mathbb{Q}[H,e^{t_1}Q]\left/\left( H^{N+1} - e^{t_1}Q \right)\right.
    \]
\end{example}

\subsection{The quantum $\mathcal{D}$-module}\label{HGW:subsection_QDM}

\subsubsection*{Big quantum $\mathcal{D}$-module}

The references for this subsection are \cite{CK_book}, Section 10.2 and \cite{Iritani_integral_structure}, Section 2.2.

\begin{assumption}
    We can recall that following Remark \ref{HGW:rmk_large_radius_limit}, the quantum product $T_i \bullet_\tau T_j$ can be seen as a formal power series in the parameters $\tau'$ and $e^{\tau_2} Q$.
    From now on, we assume that the potential $\mathcal{F}$ is convergent on some open set $U$, neighbourhood of the large radius limit. On this open set $U$, we can now replace the expression $e^{\tau_2(d)} Q^d$ by $e^{\tau_2(d)}$ to drop the Novikov variables. We will still denote $T_i \bullet_\tau T_j = \left(T_i \bullet_\tau T_j\right)_{|Q=1}$.
\end{assumption}


\begin{defin}\label{HGW:def_qdm}
    Let $z$ be a local coordinate on $\mathbb{P}^1$ at $0 \in \mathbb{A}^1$. The \textit{quantum $\mathcal{D}$-module} is the bundle with connection $QDM(X)=(F,\nabla)$ where $F$ is the trivial bundle 
    \[
        U \times \mathbb{P}^1 \times H^*(X) \to U \times \mathbb{P}^1
    \]
    Since $F$ is a trivial bundle, the space of its sections is spanned as a $\mathcal{O}_{U \times \mathbb{P}^1}$-module by the constant sections of value $T_i$ on the fibres for $i \in \{0,\dots,N\}$.
    We will also denote by $F$ the space of sections of $F$, and by $(T_i)$ the basis of sections of $F$.
    
    We denote by $t_i$ the local coordinates on the open set $U$. The connection $\nabla$ on the bundle $F$ is called the \textit{Dubrovin connection}, and is defined by
    \[
        \left\{\begin{aligned}
            & \nabla_{\partial_{t_i}} T_j = \left(\partial_{t_i} + \frac{1}{z} T_i \bullet_\tau\right) T_j, && 0 \leq i \leq N    \\
            & \nabla_{\partial_z} T_j = \left(\partial_z - \frac{1}{z^2} \mathfrak{E} \bullet_\tau + \frac{1}{z} \mu \right) T_j
        \end{aligned}\right.
    \]
    Where the \textit{Euler field} $\mathfrak{E}$ is the section of the bundle $F$ defined by
    \[
        \mathfrak{E} =
        c_1(TX) + \sum_i \left( 1 - \frac{1}{2}\textnormal{deg}_{H^*(X)}(T_i) \right) t_i T_i
    \]
    and the \textit{Hodge grading operator} is the endomorphism $\mu \in \textnormal{End}(H^*(X))$ given by
    \[
        \mu \left( T_j \right ) =
        \frac{1}{2}\left( \textnormal{deg}_{H^*(X)}(T_j) - \textnormal{dim}_\mathbb{C}(X) \right) T_j
    \]
\end{defin}

\begin{prop}\label{HGW:prop_qdm_cyclic}
    The quantum $\mathcal{D}$-module $(F, \nabla)$ is a cyclic $\mathcal{D}$-module, generated by the constant section $\mathds{1}$.
    In other words, if we denote by $F$ the space of sections of the bundle $F$, and consider the application
    \[
        \functiondesc{\varphi}
        {\mathbb{C}[t_0, \dots, t_N, z] \langle z\partial t_0, \dots, z \partial t_N \rangle}
        {F}
        {P(t_0,\dots,t_N,z,z\partial t_0, \dots, z \partial t_N)}
        {P(t_0,\dots,t_N,z,\nabla_{z\partial t_0}, \dots, \nabla_{z\partial t_N}) \cdot \mathds{1}}
    \]
    Then, the application $\varphi$ is surjective.
\end{prop}

\begin{proof}
    The space of sections $\mathcal{F}$ is spanned by the constant sections $T_0, \dots, T_N$.
    Using the definition of Dubrovin's connection, we have for any $i \in \{ 0, \dots, N \}$,
    \[
        \nabla_{z \partial t_i} \mathds{1} = z \partial_{t_i} \mathds{1} + T_i \bullet_\tau \mathds{1} = 0 + T_i \bullet_\tau \mathds{1}
    \]
    According to Proposition \ref{HGW:prop_qprod_properties}, the cohomological class $\mathds{1}$ is the unit of the quantum product, so we get
    \[
        \nabla_{z \partial t_i} \mathds{1} = T_i
    \]
    So the application $\varphi$ is surjective.
\end{proof}


\begin{prop}[\cite{Dub_Frob}]\label{HGW:prop_QDM_flat}
    Dubrovin's connection $\nabla$ is flat, i.e. its curvature satisfies $\nabla ^2 = 0$
\end{prop}

    
    

We now define a metric on the bundle $F$.

\begin{defin}
    Let $\iota : U \times \mathbb{P}^1 \to U \times \mathbb{P}^1$ be the involution given by $z \mapsto -z$. We define a pairing $\textbf{g}$ on $F$ by
    \[
    \functiondesc
    {\textbf{g}}
    {\iota^* (F,\nabla) \times (F,\nabla)}
    { \mathcal{O}_{U \times \mathbb{P}^1}}
    {(s_1,s_2)}
    {\int_X s_1(t,-z) \cup s_2(t,z)}
    \]
\end{defin}

\begin{prop}
    This pairing is $\nabla$-flat, i.e. if $s_1, s_2$ are sections of $F$, and $\xi$ is a vector field on $U \times \mathbb{P}^1$, then
    \[
        \partial_\xi \textbf{g}(s_1,s_2) = 
        \textbf{g}(\nabla_\xi s_1,s_2)
        +
        \textbf{g}(s_1,\nabla_\xi s_2)
    \]
\end{prop}

\begin{proof}
    We have for all $i,j,k \in \{0,\dots, N\}$,
    \[
        \textbf{g}(\nabla_{\partial t_k}T_i,T_j)+\textbf{g}(T_i,\nabla_{\partial t_k}T_j)
        =
        \frac{1}{-z} \textbf{g} (T_k,T_i\bullet T_j ) + \frac{1}{z}\textbf{g} (T_i, T_j \bullet T_j)
        =
        0
        =
        \partial_{t_k} \textbf{g}(T_i, T_j)
    \]
    Where we used the Remark \ref{HGW:rmk_product_metric_frobenius} in the second equality.
    
    For the derivative with respect to coordinate $z$, recall that $g(T_i,T_j) = 0$ unless $T_i$ and $T_j$ are Poincaré dual, i.e. $T_j=T^i$, and then $\text{deg }T_i + \text{deg }T_j = 2 \textnormal{dim}_\mathbb{C}(X)$.
    As a consequence, we have
    \begin{align*}
        &\textbf{g}(\nabla_{z \partial z}T_i,T_j)+ \textbf{g}(T_i,\nabla_{z \partial z}T_j) \\
        &=
        \frac{1}{z}\textbf{g}(\mathfrak{E} \bullet T_i, T_j) + \frac{\text{deg }T_i}{2}\textbf{g}(T_i,T_j) + \frac{1}{-z}\textbf{g}(T_i,\mathfrak{E} \bullet T_j) + \frac{\text{deg }T_j}{2}(T_i,T_j) -\textnormal{dim}_\mathbb{C}(X) \textbf{g}(T_i, T_j)
        \\
        &=0
        = z\partial_z \textbf{g}(T_i,T_j)
    \end{align*}
    Where we used the Remark \ref{HGW:rmk_product_metric_frobenius} again in the second equality.
\end{proof}

\subsubsection*{Small quantum $\mathcal{D}$-module}

The small analogue of the quantum $\mathcal{D}$-module is obtained by replacing the quantum product $\bullet_\tau$ with the small quantum product $\circ_{\tau_2}$ and restricting the directions for which we define $\nabla$ to variables $t_j$ of degree 2.

\begin{defin}
    Let $\tau' = t_0 T_0 + \sum_{i>r} t_i T_i$, $\tau_2 = \sum_{j=1}^r t_j T_j$ and $U_2=U \cap H^2(X;\mathbb{C})$.
    The \textit{small quantum $\mathcal{D}$-module} is the bundle with connection obtained by restriction of the quantum $\mathcal{D}$-module $(F,\nabla)$ to $U_2 \times \mathbb{P}^1$.
    \begin{center}
        \begin{tikzcd}
            j^*(F,\nabla)
                \arrow[r]
                \arrow[d]
            &
            (F,\nabla)
                \arrow[d]
            \\
            U_2 \times \mathbb{P}^1
                \arrow[r,"j",hookrightarrow]
            &
            U \times \mathbb{P}^1
        \end{tikzcd}
    \end{center}
    Keeping the same notations, we obtain a connection $\nabla$ defined by
    \[
    \left\{\begin{aligned}
        & \nabla_{\partial_{t_j}} T_k = \left(\partial_{t_j} + \frac{1}{z} T_i \circ_{\tau_2} \right) T_k, && 1 \leq j \leq r\\
        & \nabla_{\partial_z} T_k = \left(\partial_z - \frac{1}{z^2} c_1(\textnormal{T}X) \circ_{\tau_2} + \frac{1}{z} \left( 1 - \frac{1}{2}\textnormal{dim}(X) \right) \right) T_k
    \end{aligned}\right.
    \]
\end{defin}

\begin{remark}
    In general, Dubrovin's connection  $\nabla$ has a regular singularity at $z=\infty$ and an irregular singularity at $z=0$. If $X$ is Calabi--Yau, then $c_1(\textnormal{T}X)=0$ and the small quantum $\mathcal{D}$-module's singularity $z=0$ becomes regular.
    
    While the big quantum $\mathcal{D}$-module was cyclic spanned by the section $\mathds{1}$ (cf Proposition \ref{HGW:prop_qdm_cyclic}), it is not necessarily the case for the small quantum $\mathcal{D}$-module.
    Although, it is cyclic if, for example, we can generate the cohomology ring $H^*(X;\mathbb{Q})$ from the unit $\mathds{1}$ and products with elements of $H^2(X;\mathbb{Q})$.
    This is true for smooth toric varieties, see \cite{Fulton_toric_book}, Proposition p.106.
\end{remark}

\begin{example}\label{HGW:ex_sqdm_projectivespace}
    For $X= \mathbb{P}^N$, 
    let $H=c_1\left( \mathcal{O}(1) \right) \in H^2 \left( \mathbb{P}^N \right)$ be the hyperplane class, and set $T_i = H^i$.
    
    Using the Euler sequence (\cite{hartshorne_hartshorne}, Theorem 8.13),
    \begin{center}
        \begin{tikzcd}
            0 
                \arrow{r} 
            & \mathcal{O}_{\mathbb{P}^n} 
                \arrow{r} 
            & \mathcal{O}_{\mathbb{P}^n}(1)^{\oplus n+1}
                \arrow{r} 
            & \text{T}\mathbb{P}^n 
                \arrow{r} 
            & 0
        \end{tikzcd}
    \end{center}
    We obtain $c_1(T\mathbb{P}^N) = (N+1) H$.
    
    The small quantum connection is the connection on the bundle
    \[
        H^2\left( \mathbb{P}^N \right) \times \mathbb{P}^1 \times H^*\left( \mathbb{P}^N \right) \to H^2\left( \mathbb{P}^N \right)_{(t_1)} \times \mathbb{P}^1_{(z)}
    \]
    Given by
    \[
        \left\{
            \begin{aligned}
                \nabla_{z \partial_{t_1}} &= z \partial_{t_1} + H \circ_{\tau_2}
                \\
                \nabla_{z \partial_z} &= z \partial_z - \frac{1}{z} (N+1)H \circ_{\tau_2} + \left(1-\frac{N+1}{2} \right)
            \end{aligned}
        \right.
    \]
    Considering only the first direction, we get a module $SQDM \left( \mathbb{P}^N \right)$ on $\mathbb{C}[t_1,z,e^{t_1}]\langle z \partial_{t_1} \rangle$. We have
    \[
        \mathbb{C}[t_1,z,e^{t_1}]\langle z \partial_{t_1} \rangle \left/\left( (z \partial_{t_1})^{N+1} - e^{t_1} \right)\right.
        \simeq
        SQDM \left( \mathbb{P}^N \right)
    \]
    This isomorphism is given by
    \[
        P(t_1,z, e^{t_1},z\partial_{t_1}) \mapsto P(t_1,z, e^{t_1},\nabla_{z\partial_{t_1}}) \cdot \mathds{1} \in SQDM \left( \mathbb{P}^N \right)
    \]
    Thus $SQDM \left( \mathbb{P}^N \right)$ is cyclic, its unique generator is  the constant section $\mathds{1}$.
\end{example}

\subsection{Fundamental solution and Givental's $J$-function}

The aim of this subsection is to build a formal fundamental solution to the quantum $\mathcal{D}$-module. This fundamental solution is related to a cohomological function, called Givental's $J$-function, which plays an essential role: this function can be used to compute Gromov--Witten invariants and obtain relations in quantum cohomology.

\subsubsection*{Fundamental solution}


We will follow \cite{Iritani_integral_structure}, Section 2.2. We begin by introducing a formal function which will be our main tool to construct a fundamental solution.
\begin{defin}\label{HGW:def_S_fn}
    We define the formal function $S^\textnormal{coh}$ by the expression, for $\alpha \in H(X)$
    \[
        S^\textnormal{coh}(\tau,z) ( \alpha ) = e^{-\tau_2/z}\alpha
        -\sum_{\substack{d \in H_2(X;\mathbb{Z})-\{0\} \\ l \geq 0 }} \,
        \sum_{k=0}^N \frac{1}{l!} \left\langle T_k, \tau ', \dots , \tau ', \frac{e^{-\tau_2/z} \alpha}{z+\psi} \right\rangle^\textnormal{coh}_{0,l+2,d} T^k e^{\tau_2(d)} \in QH(X) \otimes \mathbb{C}( \! (z) \! )
    \]
    Where the Gromov--Witten invariant $\left\langle T_k, \tau ', \dots , \tau ', \frac{e^{-\tau_2/z} \alpha}{z+\psi} \right\rangle^\textnormal{coh}_{0,l+2,d}$ is actually a shortcut for the expression
    \[
    \left\langle T_k, \tau ', \dots , \tau ', \frac{e^{-\tau_2/z} \alpha}{z+\psi} \right\rangle^\textnormal{coh}_{0,l+2,d}
    :=
    \sum_{n,m \geq 0} \frac{(-1)^{n+m}}{z^{n+m+1}}
    \left\langle T_k, \tau ', \dots , \tau ', \psi_{l+2}^n \tau_2^m \cup \alpha \right\rangle^\textnormal{coh}_{0,l+2,d}
    \]
\end{defin}

Now we can construct the fundamental solution. This is done in essentially two parts: first we show that the formal function $S^\textnormal{coh}$ is a fundamental solution for the directions $\partial_{t_i}$, then we modify $S^\textnormal{coh}$ to obtain a full solution.

\begin{thm}[\cite{Iritani_integral_structure}, Proposition 2.4]\label{HGW:S_fund_sol}
     We recall that $\mu$ denoted the Hodge grading operator, was defined by
        \[
            \mu \left( T_j \right ) =
            \frac{1}{2}\left( \textnormal{deg}_{H^*(X)}(T_j) - \textnormal{dim}_\mathbb{C}(X) \right) T_j
        \]
    We also denote by $\rho$ the endormorphism of $H^*(X), \rho = c_1(TX) \cup$.
    
    \textit{(i)} \quad
        Let $\alpha \in H^\star(X)$ be a section of F. For all $i \in \{0, \dots ,N\}$, we have
        \begin{align*}
            \nabla_{\partial t_i} S^\textnormal{coh}(\tau,z) \alpha &= 0  \\
            \nabla_{z \partial_z } S^\textnormal{coh}(\tau,z) \alpha &= S^\textnormal{coh}(\tau,z) \left(\mu(\alpha)-\frac{1}{z}\rho(\alpha) \right) 
        \end{align*}
        
    \textit{(ii)} \quad
        We define the endomorphism $z^{-\mu} z^\rho \in \textnormal{End}(QH(X) \otimes \mathbb{C}( \! (z) \! ))$ by
        \[
            z^{-\mu} z^\rho(T_i) = \exp (-\mu \log(z)) \exp ( \rho \log(z)) \cdot T_i
        \]
        Then, the function $S^\textnormal{coh}(\tau,z)z^{-\mu}z^\rho$ is a fundamental solution of the quantum $\mathcal{D}$-module.

    \textit{(iii)} \quad
         The function $S^\textnormal{coh}(\tau,z)$ is an isometry, i.e. for all $i,j \in \{0, \dots, N\}$
        \[
            \textbf{g} \left( S^\textnormal{coh}(\tau,z)(T_i), S^\textnormal{coh}(\tau,z)(T_j) \right) = \textbf{g}(T_i,T_j)
        \]
\end{thm}

The proof of the first equality in \textit{(i)} is adapted from \cite{CK_book}, Propositions 10.2.1 and 10.2.3.
For the other results, we refer to \cite{Iritani_integral_structure}, Proposition 2.4.
Before giving the proof, we introduce a few lemmas.
The first two lemmas are intermediary steps to prove that $\nabla_{\partial t_i} S^\textnormal{coh}(\tau,z) = 0$.

\begin{notation}
We will use the compact expression
\[
        \coupl
            \psi_1^{k_1} \alpha_1 , \dots , \psi_n^{k_n} \alpha_n
        \coupr^\textnormal{coh}_{0,n,\tau}
        :=
        \sum_{\substack{l \geq 0 \\ d \in H_2(X;\mathbb{Z})}}
        \frac{1}{l!}
        \langle
            \psi_1^{k_1} \alpha_1 , \dots , \psi_n^{k_n} \alpha_n, \underbrace{\tau, \dots, \tau}_{l \textnormal{ times}}
        \rangle^\textnormal{coh}_{0,n+l,d}
        Q^d
\]
\end{notation}

\begin{lemma}[Topological Recursion Relations]\label{HGW:lemma_fund_sol_TRR}
        For all $k_1, k_2, k_3 \geq 0$\\ and $0 \leq j_1, j_2, j_3\ \leq r$,
    \[
        \coupl
        \psi_1^{d_1+1} T_{j_1}, \psi_2^{d_2} T_{j_2}, \psi_3^{d_3} T_{j_3}
        \coupr_{0,3,\tau}
        =
        \sum_{a=0}^N 
        \coupl \psi_1^{d_1} T_{j_1}, T_a \coupr_{0,2,\tau}
        \coupl T^a, \psi_2^{d_2} T_{j_2}, \psi_3^{d_3} T_{j_3} \coupr_{0,3,\tau}
    \]
    \qed
\end{lemma}
For a proof of this lemma, see \cite{CK_book}, Lemma 10.2.2.

\begin{lemma}\label{HGW:lemma_fund_sol_compact_form}
    The fundamental solution $S^\textnormal{coh}$ can be written as the compact expression
    \begin{equation}\label{HGW:eqn_fund_sol_cpct}
        S^\textnormal{coh}(\tau,z)(\alpha) = \alpha - \sum_{j=0}^N T^j  \coupl \frac{\alpha}{z+\psi},T_j \coupr_{0,2,\tau}
    \end{equation}
\end{lemma}

\begin{proof}
We start from the right hand side.
By writing $\tau = \tau_2 + \tau'$ inside the Gromov--Witten invariants in $\coupl \frac{\alpha}{z+\psi},T_j \coupr_0$, we can make a careful use of the linearity of the invariants and the Divisor Axiom.
By linearity, we have
\[
    \left\langle \psi^n T_a, T_j, \tau^{(k)} \right\rangle^\textnormal{coh}_{0,k+2,d} 
    = 
    \sum_{u+v=k} \frac{k!}{u!v!}
    \left\langle
        \psi^n T_a, T_j, \tau_2^{(u)}, \tau'^{(v)}
    \right\rangle^\textnormal{coh}_{0,k+2,d}
\]
Where we recall that the notation $ \tau^{(k)}$ means the entry $\tau$ appears $k$ times.
Before applying the Divisor Axiom, we notice that we have to separate two cases: either $d \neq 0$ and we can apply the axiom $u$ times, or $d=0$ and we can apply the axiom $u-1$ times.
When $d \neq 0$, we have
\[
    \left\langle 
        \psi_1^n T_a, T_j, \tau^{(k)} 
    \right\rangle^\textnormal{coh}_{0,k+2,d} 
    =
    \sum_{\substack{u+v=k \\ x+y=u}} \frac{k!}{v!x!y!} (\tau_2(d))^x 
    \left\langle 
        \psi_1^{n-y} (T_a \cup \tau_2^y), T_j, \tau'^{(v)}
    \right\rangle^\textnormal{coh}_{0,v+2,d}
\]
Therefore, in the right hand side of (\ref{HGW:eqn_fund_sol_cpct}), the coefficient in front of $Q^d, d>0$ is given by the sum on the following parameters
\begin{align*}
    &j \in \{0, \dots N\}
    && \textnormal{parameters that runs the basis in cohomology}
    \\
    &n \in \mathbb{Z}_{\geq 0}
    && \textnormal{from developing } \frac{1}{z+\psi}
    \\
    &k \in \mathbb{Z}_{\geq 0}
    && \textnormal{from developing the expression } \coupl \psi^n T_a,T_j \coupr_{0,2,\tau} 
    \\
    &u, v \in \mathbb{Z}_{\geq 0}
    && \textnormal{from using the linearity axiom } 
    \\
    &x, y \in \mathbb{Z}_{\geq 0}
    && \textnormal{from using the divisor axiom } 
\end{align*}
Moreover, we have the relations $u+v=k$ and $x+y=u=k-v$, so $u$ is a mute parameter.
The general term of this whole sum is given by
\[
    \frac{(-1)^n}{z^{n+1}}\frac{1}{v!}\frac{1}{x!}\frac{1}{y!}(\tau_2(d))^x
    \left\langle \psi^{n-y}
        (T_a \cup \tau_2^y), T_j, \tau'^{(v)}
    \right\rangle^\textnormal{coh}_{0,v+2,d} T^j
\]
Let $l=n-y$, the general term becomes
\[
    \frac{1}{v!}\left( \frac{(\tau_2(d))^x}{x!}\right) \frac{(-1)^l}{z^{l+1}}
    \left\langle
        \psi^l \left( \frac{(-1)^y}{y!}\frac{(\tau_2)^y}{z^y}\right)\cup T_a, T_j, \tau'^{(v)}
    \right\rangle^\textnormal{coh}_{0,v+2,d} T^j
\]
The series in the parameters $x$ and $y$ can now be identified with the Taylor series of an exponential, while the series in the parameter $l$ can be identified with the expansion of $\frac{1}{z+\psi}$.
The coefficient in front of $Q^d, d>0$ in the right hand side of of (\ref{HGW:eqn_fund_sol_cpct}) is finally given by
\[
    \sum_{j=0}^N \sum_{l \geq 0} \frac{1}{v!} e^{\tau_2(d)} 
    \left\langle \frac{e^{-\tau_2/z}}{z+\psi},T_k,(\tau ')^v 
    \right\rangle^\textnormal{coh}_{0,l+2,d} T^j
\]
which is precisely the coefficient in front of $Q^d, d>0$ in the left hand side of of (\ref{HGW:eqn_fund_sol_cpct}).

Now we move to the $d=0$ case. In that case, we have by the Divisor Axiom
\[
    \left\langle 
        \psi_1^n T_a, T_j, \tau^{(k)} 
    \right\rangle^\textnormal{coh}_{0,k+2,0}
    =
    \sum_{u+v=k}
    \sum_{x+y=u-1} \frac{k!}{v! x! y!} \tau_2(0)^x
    \left\langle
        \psi_1^{n-y+1} (T_a \cup \tau_2^{y-1}), T_j, \tau_2, \tau'^{(v)}
    \right\rangle^\textnormal{coh}_{0,v+3,0}
\]
When $d=0$, we have $\overline{\mathcal{M}}_{0,n}(X,0) \simeq \overline{\mathcal{M}}_{0,n} \times X$.
Therefore, the line bundle $\mathcal{L}_1$ is trivial and its first Chern class is zero.
Consequently, 
$
    \left\langle
        \psi^{n-y+1} (T_a \cup \tau_2^{y-1}), T_j, \tau_2, \tau'^{(v)}
    \right\rangle^\textnormal{coh}_{0,v+3,0}
$
are zero unless $n-y+1 = 0$.
So the non zero Gromov--Witten invariants are of the form $\left\langle T_a \cup \tau_2^n, T_j, \tau_2, (\tau')^v \right\rangle_{0,v+3,0}$.
Since there no psi classes, we can apply the Point Mapping Axiom. The remaining non zero invariants are
\[
    \left\langle
        (T_a \cup \tau_2^n), T_j, \tau_2,
    \right\rangle^\textnormal{coh}_{0,3,0}
    =
    \int_X
        T_a \cup T_j \cup \tau_2^{n+1}
\]
Finally, in the right hand side of (\ref{HGW:eqn_fund_sol_cpct}), the coefficient in front of $Q^0$ is given by
\[
    T_a + \sum_{n=0}^\infty \sum_j \frac{(-1)^{n+1}}{(n+1)!} \frac{1}{z^{n+1}} \left( \int_X T_a \cup \tau_2^{n+1} \cup T_j \right) T^j
    =
    e^{-\tau_2/z}T_a
\]
This completes the proof of the identity (\ref{HGW:eqn_fund_sol_cpct}).
\end{proof}

The next result will be helpful for the computation of $\nabla_{z \partial_z } S^\textnormal{coh}(\tau,z)$.

\begin{notation}
    We are going to associate to the Euler field $\mathfrak{E}$ a vector field $\mathfrak{E}_{(\partial)}$ by replacing the generator $T_j \in H^*(X; \mathbb{Q})$ in the expression of the section $\mathfrak{E}$ by the differential operator $\partial_{t_j}$.
    The motivation for introducing this operator is that once we have shown that $\nabla_{\partial t_i} S^\textnormal{coh}(\tau,z) = 0$, then we have $ \mathfrak{E} \bullet_\tau S^\textnormal{coh}(\tau,z) = -z\mathfrak{E}_{(\partial)} \cdot S^\textnormal{coh}(\tau,z) $, which is easier to compute.
    Explicitly, if $c_1(TX) = \sum_{i=1} \omega_i T_i$, then
    \[
        \mathfrak{E}_{(\partial)} := \sum_{i=1}^r \omega_i \partial_{t_i} + \sum_{k=0}^N \left(1 - \frac{\text{deg }T_k}{2}t_k\right) \partial_{t_k}
    \]
\end{notation}

Notice that the expression of $\mathfrak{E}_{(\partial)}$ consists of two sums, the first acting on the variables of cohomological degree two ($\tau_2$), the second acting on the other variables ($\tau'$).

\begin{lemma}\label{hgw:lemma_fundsol_commutator_shift_zdz}
    Recall that $\mu$ is the Hodge grading operator and $\rho = c_1(TX) \cup $.
    Let $\alpha \in H^*(X; \mathbb{Q})$. The fundamental solution $S^\textnormal{coh}$ satisfies the following commutativity property:
    \[
    (z\partial_z + \mathfrak{E}_{(\partial)} + \mu) \, _\circ \, S^\textnormal{coh}(\tau,z)(e^{\tau_2/z} \alpha) = S^\textnormal{coh}(\tau,z) \, _\circ \, e^{\tau_2/z} \, _\circ \, (z\partial z + \mathfrak{E}_{(\partial)} + \mu) (\alpha)
    \]
\end{lemma}

\begin{proof}
    We introduce in this proof the notation $\textbf{S}(t,z)(\alpha) := S(t,z)(e^{\tau_2/z} \alpha)$.
    Therefore, we have
    \[
        \textbf{S}(t,z)(\alpha)
        =
        \alpha
        -\sum_{\substack{d \in H_2(X;\mathbb{Z})-\{0\} \\ l \geq 0 }} \,
        \sum_{k=0}^N \frac{1}{l!} \left\langle T_k, \tau ', \dots , \tau ', \frac{ \alpha}{z+\psi} \right\rangle^\textnormal{coh}_{0,l+2,d} T^k e^{\tau_2(d)}
    \]
    We begin by computing the right hand side in the identity we want to prove. We have
    \[
    \textbf{S}(t,z)(z\partial z + \mathfrak{E}_{(\partial)} + \mu)\alpha
    =
    \textbf{S}(t,z) \left( \frac{\text{deg}(\alpha)}{2}-\frac{\textnormal{dim}_\mathbb{C}(X)}{2} \right) \alpha
    =\left( \frac{\text{deg }\alpha}{2}-\frac{\textnormal{dim}_\mathbb{C}(X)}{2} \right) S(t,z) (\alpha)
    \]
    Thus, we want to show that
    \[
        (z\partial_z + \mathfrak{E}_{(\partial)} + \mu)\textbf{S}(t,z)(\alpha)
        =
        \left( \frac{\text{deg}(\alpha)}{2}-\frac{\textnormal{dim}_\mathbb{C}(X)}{2} \right) \textbf{S}(t,z) (\alpha)
    \]
    
    We have the expansion
    \[
        \left\langle T_k, (\tau ')^l, \frac{\alpha}{z+\psi} \right\rangle_{0,l+2,d}
        = \sum_{u \geq 0} \frac{(-1)^u}{z^{u+1}} \left\langle \psi^u \alpha, T_k, (\tau ')^l \right\rangle_{0,l+2,d}
    \]
    We are going to evaluate the action of the operators $z\partial_z, \mu, \mathfrak{E}_{(\partial)}$ individually on $\textbf{S}(t,z)(\alpha)$.
    We have
    \begin{align*}
        z\partial_z \cdot \textbf{S}(t,z)(\alpha) 
        &=
        \sum_{k} \sum_{\substack{d \neq 0 \\ l,u \geq 0}} (-1)^{u+1} \frac{1}{l!}(u+1)\frac{1}{z^{u+1}} \left\langle \psi^u \alpha, T_k, \tau '^{(l)} \right\rangle^\textnormal{coh}_{0,l+2,d} T^k e^{\tau_2(d)}
        \\
        \mu \cdot \textbf{S}(t,z)(\alpha) 
        &=
        \left(\frac{\text{deg}(\alpha)}{2} - \frac{\textnormal{dim}_\mathbb{C}(X)}{2}\right)\alpha + \sum_{k} \sum_{\substack{d \neq 0 \\ l,u \geq 0}} (-1)^u \frac{1}{l!}\frac{(\text{deg }T^k)/2}{z^{u+1}} \left\langle \psi^u \alpha, T_k, (\tau ')^l \right\rangle_{0,l+2,d} T^k e^{\tau_2(d)}
    \end{align*}
    Next, we will compute
    $
        \mathfrak{E}_{(\partial)} \cdot \textbf{S}(t,z)(\alpha)
    $. In a first time, we have
    \begin{equation}\label{HGW:eqn_lemma_dz_fund_sol_1}
        \mathfrak{E}_{(\partial)} e^{\tau_2(d)}
        =
        \left(\sum_{i=1}^r \omega_i\partial_{t_i}\right) e^{\tau_2(d)}
        =
        \sum_{i=1}^r \omega_i T_i(d) e^{\tau_2(d)}
    \end{equation}
    In a second time, for a fixed $l \in \mathbb{Z}_{>0}$, we develop the Gromov--Witten invariant $\left\langle \psi^u \alpha, T_k, (\tau ')^l \right\rangle^\textnormal{coh}_{0,l+2,d}$.
    For a multi-index $\underline{a'}=(a_0, a_{r+1}, \dots, a_N) \in (\mathbb{Z}_{\geq 0})^{N+1-r}$, we introduce the following notations:
    \begin{align*}
        |\underline{a'}|
        &:=
        a_0 + a_{r+1} + \cdots + a_N
        \\
        \underline{t}^{\underline{a'}}
        &:=
        t_0^{a_0} t_r^{a_{r+1}} \cdots t_N^{a_N}
        \\
        \left\langle \underline{T}^{(\underline{a'})} \right\rangle^\textnormal{coh}_{0,|\underline{a}|,d}
        &:=
        \left\langle T_0^{(a_0)}, T_{r+1}^{(a_{r+1})} \dots, T_N^{(a_N)} \right\rangle^\textnormal{coh}_{0,|\underline{a}|,d}
        \\
        \underline{a'}!
        &:=
        (a_0!) \cdots (a_N!)
    \end{align*}
    Using this notation, we have
    \[
        \frac{1}{l!} \left\langle \psi^u \alpha, T_k, (\tau ')^l \right\rangle_{0,l+2,d}
        =
        \sum_{\substack{\underline{a'} \in (\mathbb{Z}_{\geq 0})^{N+1-r} \\ |\underline{a}|=l}}
        \frac{\underline{t}^{\underline{a'}}}{\underline{a'}!}
        \left\langle
            \psi^u \alpha, T_k,\underline{T}^{(\underline{a'})}
        \right\rangle^\textnormal{coh}_{0,l+2,d}
    \]
    We can compute the action of $\mathfrak{E}_{(\partial)}$ on this function. For $k \in \mathbb{Z}_{\geq 0}$, we set by convention $a_k = 0$ if $\textnormal{deg}(T_k) = 2$.
    \begin{equation}\label{HGW:eqn_lemma_dz_fund_sol_2}
        \mathfrak{E}_{(\partial)} \cdot \frac{1}{l!} \left\langle \psi^u \alpha, T_k, (\tau ')^l \right\rangle_{0,l+2,d}
        =
        \sum_{i=0}^N \left( 1-\frac{\textnormal{deg}(T_k)}{2}\right)
        \sum_{\substack{\underline{a'} \in (\mathbb{Z}_{\geq 0})^{N+1-r} \\ |\underline{a}|=l}}
        a_i \frac{\underline{t}^{\underline{a'}}}{\underline{a'}!}
        \left\langle
            \psi^u \alpha, T_k,\underline{T}^{(\underline{a'})}
        \right\rangle^\textnormal{coh}_{0,l+2,d}
    \end{equation}
    Plugging (\ref{HGW:eqn_lemma_dz_fund_sol_1}) and (\ref{HGW:eqn_lemma_dz_fund_sol_2}) together, we obtain
    \[
        \mathfrak{E}_{(\partial)} \cdot
        \frac{1}{l!}
        \left\langle
            \psi^u \alpha, T_k, (\tau ')^l
        \right\rangle^\textnormal{coh}_{0,l+2,d}
        e^{\tau_2(d)}
        =
        \sum_{|\underline{a}|=l} \left(
        \sum_{i=0}^N
        \left( 1-\frac{\textnormal{deg}(T_k)}{2} \right) a_i
        + \sum_{j=1}^r 
        \omega_j T_j(d)
        \right)
        \frac{\underline{t}^{\underline{a'}}}{\underline{a'}!}
        \left\langle
            \psi^u \alpha, T_k,\underline{T}^{(\underline{a'})}
        \right\rangle^\textnormal{coh}_{0,l+2,d}
        e^{\tau_2(d)}
    \]
    However, applying the Degree Axiom (see \ref{HGW:prop_degree_axiom}) to the Gromov--Witten invariant here gives the relation
    \[
        \sum_{i=0}^N
        \left( 1-\frac{\textnormal{deg}(T_k)}{2} \right) a_i
        + \sum_{j=1}^r 
        \omega_j T_j(d)
        =
        - \textnormal{dim}_\mathbb{C}(X) + \frac{\text{deg }T_k}{2} + \frac{\text{deg }\alpha}{2} + u + 1
    \]
    Which means we can use the linearity of Gromov--Witten invariants to get rid of the sum on multi-indices $\underline{a'}$ and obtain
    \[
        \mathfrak{E}_{(\partial)} \cdot
        \textbf{S}(t,z)(\alpha)
        =
        \sum_{k=0}^N \sum_{\substack{d \neq 0 \\ l,u \geq 0}}
        \left(-\textnormal{dim}_\mathbb{C}(X) + \frac{\text{deg }T_k}{2} + \frac{\text{deg }\alpha}{2} + u + 1 \right)
        (-1)^u
        \frac{1}{l!}
        \left\langle
            \psi^u \alpha, T_k,\underline{T}^{(\underline{a'})}
        \right\rangle^\textnormal{coh}_{0,l+2,d}
        T^k e^{\tau_2(d)}
    \]
    
    At last, we have that $(z\partial_z + \mathfrak{E}_{(\partial)} + \mu) _\circ \textbf{S}(t,z)(\alpha)$ is equal to
    \begin{align*}
        \left(
            \frac{\text{deg}(\alpha)}{2} - \frac{\textnormal{dim}_\mathbb{C}(X)}{2}
        \right) \alpha
        +
        \sum_{k=0}^N \sum_{\substack{d \neq 0 \\ l,u \geq 0}}
        &
        \left(
            -(u+1)
            -
            \textnormal{dim}_\mathbb{C}(X) + \frac{\text{deg }T_k}{2} + \frac{\text{deg }\alpha}{2} + u + 1
            +
            \frac{\textnormal{deg}(T^k) - \textnormal{dim}_\mathbb{C} (X) }{2}
        \right)
        \times \\ \times
        &
        \frac{(-1)^u}{z^{u+1}}
        \frac{1}{l!}
        \left\langle
            \psi^u \alpha, T_k, (\tau ')^l
        \right\rangle^\textnormal{coh}_{0,l+2,d}
        e^{\tau_2(d)}
    \end{align*}
    Using $\textnormal{deg}(T^k) = 2 \textnormal{dim}_\mathbb{C} (X) - \textnormal{deg}(T_k)$, we obtain
    \[
        (z\partial_z + \mathfrak{E}_{(\partial)} + \mu) _\circ \textbf{S}(t,z)(\alpha)
        =
        \left(
            \frac{\text{deg}(\alpha)}{2} - \frac{\textnormal{dim}_\mathbb{C}(X)}{2}
        \right)
        \textbf{S}(t,z)(\alpha)
    \]
\end{proof}

\begin{lemma}\label{hgw:lemma_fundsol_commutator_shift_zdz_2}
    Let $\alpha \in H^*(X; \mathbb{Q})$ be a section of $F$. We have
    \[
        \left( z\partial_z + \mathfrak{E}_{(\partial)} + \mu \right) \fndot e^{-\tau_2/z}
        =
        e^{-\tau_2/z} \fndot \left( z\partial_z + \mathfrak{E}_{(\partial)} + \mu - \frac{\rho}{z} \right)
    \]
\end{lemma}

\begin{proof}
    Let $T_a \in H^*(X; \mathbb{Q})$ be a section of $F$. The right hand side gives
    \[
        e^{-\tau_2/z} \fndot \left( z\partial_z + \mathfrak{E}_{(\partial)} + \mu - \frac{\rho}{z} \right) (T_a)
        =
        e^{-\tau_2/z} \left(\mu - \frac{\rho}{z} \right)(T_a)
    \]
    For the left hand side, $\left( z\partial_z + \mathfrak{E}_{(\partial)} + \mu \right) \fndot \left( e^{-\tau_2/z}T_a \right)$, we have
    \begin{align*}
        z \partial_z \cdot e^{-\tau_2/z} T_a
        &=
        \frac{\tau_2}{z} e^{-\tau_2/z} T_a
        \\
        \mathfrak{(E)}_\partial \cdot e^{-\tau_2/z} T_a
        &=
        e^{-\tau_2/z} \left(-\frac{\rho}{z}\right) (T_a)
        \\
        \mu \cdot e^{-\tau_2/z} T_a
        &=
        \sum_{k \geq 0} \frac{1}{k! z^k} \left( k + \frac{\textnormal{deg}(T_a) - \textnormal{dim}(X)}{2}\right) \tau_2^k \cup T_a
        =
        -z \partial_z( e^{-\tau_2/z} T_a ) + e^{-\tau_2/z} \cup \mu(T_a)
    \end{align*}
    Therefore,
    \[
        \left( z\partial_z + \mathfrak{E}_{(\partial)} + \mu \right) \fndot \left( e^{-\tau_2/z}T_a \right)
        =
        e^{\tau_2/z} \left( -\frac{\rho}{z} + \mu \right)(T_a)
    \]
\end{proof}

We finish by giving a lemma for the computation of $\nabla_{z \partial z} S^\textnormal{coh}(\tau,z)z^{-\mu}z^\rho$.

\begin{lemma}\label{HGW:lemma_fund_sol_new_operator_pde}
    Let $\alpha \in H^*(X)$. We have
    \[
        \left (
            z \frac{\partial}{\partial z} + \mu - \frac{\rho}{z}
        \right)
        (z^{-\mu}z^\rho \alpha)
        =
        0
    \]
\end{lemma}

\begin{proof}
    First, notice that we have
    \[
        \left[ \mu, \rho \right](T_a) = \frac{\text{deg}(T_a)+2}{2}\rho(T_a) - \rho \left( \frac{\text{deg}(T_a)}{2} T_a \right)=\rho (T_a)
    \]
    Thus $[\mu, \rho]=\rho$.
    
    We will now show that
    \[
        \frac{\rho}{z}z^{-\mu}=z^{-\mu}\rho
    \]
    We have
    \begin{align*}
        z^\mu \frac{\rho}{z} z^{-\mu} &= e^{\text{ad}(\mu \log z)} \frac{\rho}{z}
        =\sum_k \frac{1}{k!} (\text{ad}(\mu \log z))^k \frac{\rho}{z}\\
        &=\sum_k \frac{1}{k!} (\log z)^k \frac{1}{z} [\underbrace{\mu,[\mu,[\cdots,[\mu}_{\text{k times}},\rho]\cdots]]
        =\sum_k \frac{1}{k!} (\log z)^k \frac{1}{z} \rho
        =\rho
    \end{align*}
    This concludes the proof of the identity
    $
    \frac{\rho}{z}z^{-\mu}=z^{-\mu}\rho
    $.
    
    We have
    \[
    z \frac{\partial}{\partial z} (z^{-\mu}z^\rho \alpha)  = -\mu z^{-\mu}z^\rho + z^{-\mu}\rho z^\rho
    \]
    Thus
    \begin{align*}
        \left (z \frac{\partial}{\partial z} + \mu - \frac{\rho}{z} \right) (z^{-\mu}z^\rho \alpha) &=z^{-\mu} \rho z^\rho - \frac{\rho}{z}z^{-\mu}z^\rho
        =0
    \end{align*}
\end{proof}

This concludes the listing of lemmas required for the proof Theorem \ref{HGW:S_fund_sol}. 

\begin{proof}[Proof of Theorem \ref{HGW:S_fund_sol}]
    Let $T_a \in H^*(X)$ be a section of $F$. We begin by showing that
    \[
        \nabla_{z\partial t_i} S^\textnormal{coh}(\tau,z)(T_a) = 0
    \]
    Using Lemma \ref{HGW:lemma_fund_sol_compact_form}, we have
    \[
        \nabla_{z\partial t_i} S^\textnormal{coh}(\tau,z)(T_a)
        =
        z\partial_{t_i} S^\textnormal{coh}(\tau,z)(T_a)
        -
        T_i \bullet_\tau T_a + T_i \bullet_\tau \sum_{j=0}^N \coupl \frac{T_a}{z + \psi}, T_j \coupr_{0,2,\tau} T^j
    \]
    We compute every summand separately. We have
    \begin{align}
        z\partial_{t_i} S^\textnormal{coh}(\tau,z)(T_a)
        &=
        \sum_{k=0}^N \sum_{n \geq 0} \frac{(-1)^n}{z^n} \coupl \psi_1^n T_a, T_i, T_k \coupr_{0,2,\tau} T^k   \label{HGW:eqn_fund_sol_dti_computation1}
        \\
        T_i \bullet_\tau T_a 
        &=
        \sum_{k=0}^N \coupl T_i, T_a, T_k \coupr_{0,3,\tau} T^k \label{HGW:eqn_fund_sol_dti_computation2}
    \end{align}
    Next, we have
    \[
        T_i \bullet_\tau T_a + T_i \bullet_\tau \sum_{j=0}^N \coupl \frac{T_a}{z + \psi}, T_j \coupr_{0,2,\tau} T^j
        =
        \sum_{j,k=0}^N \sum_{n \geq 0} \frac{(-1)^n}{z^{n+1}} \coupl \psi_1^n T_a, T_j \coupr_{0,2,\tau} \coupl T_i, T^j, T_k \coupr_{0,3,\tau} T^k
    \]
    We apply Lemma \ref{HGW:lemma_fund_sol_TRR} to the right hand side. We obtain
    \begin{equation}\label{HGW:eqn_fund_sol_dti_computation3}
        T_i \bullet_\tau T_a + T_i \bullet_\tau \sum_{j=0}^N \coupl \frac{T_a}{z + \psi}, T_j \coupr_{0,2,\tau} T^j
        =
        \sum_{k=0}^N \sum_{n \geq 0} 
        \frac{(-1)^n}{z^{n+1}}
        \coupl 
            \psi_1^{n+1} T_a, T_i, T_k
        \coupr_{0,3,\tau} T^k
    \end{equation}
    Putting (\ref{HGW:eqn_fund_sol_dti_computation1}), (\ref{HGW:eqn_fund_sol_dti_computation2}) and (\ref{HGW:eqn_fund_sol_dti_computation3}) together, we obtain
    \[
        \nabla_{z\partial t_i} S^\textnormal{coh}(\tau,z)(T_a) = 0
    \]
    
    Next, we prove that
    \[
        \nabla_{z \partial_z } S^\textnormal{coh}(\tau,z)(T_a) = S^\textnormal{coh}(\tau,z) \left(\mu(\alpha)-\frac{1}{z}\rho(T_a) \right) 
    \]
    Using $\nabla_{z\partial t_i} S^\textnormal{coh}(\tau,z)(T_a) = 0$, we have
    \[
        \nabla_{z \partial_z} S^\textnormal{coh}(\tau,z)(T_a)
        =
        \left( z \partial_z + \mathfrak{E}_{(\partial)} + \mu \right) S^\textnormal{coh}(\tau,z)(T_a)
    \]
    Using Lemma \ref{hgw:lemma_fundsol_commutator_shift_zdz}, we have
    \[
        \nabla_{z \partial_z} S^\textnormal{coh}(\tau,z)(T_a)
        =
        \left(
            S^\textnormal{coh}(\tau,z) \fndot e^{\tau_2/z}
        \right)
        \fndot
        \left(
            z \partial_z + \mathfrak{E}_{(\partial)} + \mu
        \right)
        \fndot
        e^{-\tau_2/z}
    \]
    Using Lemma \ref{hgw:lemma_fundsol_commutator_shift_zdz_2}, we obtain
    \[
        \nabla_{z \partial_z } S^\textnormal{coh}(\tau,z)(T_a) = S^\textnormal{coh}(\tau,z) \left(\mu-\frac{\rho}{z} \right)(T_a) 
    \]
    This conclude the proof of the assertion \textit{(i)} of Theorem \ref{HGW:S_fund_sol}.
    
    Next, we show that the function $S^\textnormal{coh}(\tau,z)z^{-\mu}z^\rho$ is a fundamental solution of the quantum $\mathcal{D}$-module.
    We have to show that $\nabla_{z\partial t_i}S^\textnormal{coh}(\tau,z)z^{-\mu}z^\rho = 0$ and $\nabla_{z \partial z} S^\textnormal{coh}(\tau,z)z^{-\mu}z^\rho = 0$.
    The first identity is a consequence of $\nabla_{z\partial t_i}S^\textnormal{coh}=0$ which is contained in the assertion \textit{(i)} of Theorem \ref{HGW:S_fund_sol}.
    Using the Lemmas \ref{hgw:lemma_fundsol_commutator_shift_zdz} and \ref{hgw:lemma_fundsol_commutator_shift_zdz_2} again, we obtain
    \[
        \nabla_{z \partial z} S^\textnormal{coh}(\tau,z)z^{-\mu}z^\rho
        =
        S^\textnormal{coh}(\tau,z)
        \fndot
        \left(
            z\partial_z + \mu - \frac{\rho}{z}
        \right)
        \fndot
        z^{-\mu}z^\rho
    \]
    Using Lemma \ref{HGW:lemma_fund_sol_new_operator_pde}, we obtain
    \[
        \nabla_{z \partial z} S^\textnormal{coh}(\tau,z)z^{-\mu}z^\rho = 0
    \]
    
    Finally, we show that the function $S^\textnormal{coh}$ is an isometry. For $a,b \in \{0, \dots, N\}$, we want to show that
    \[
        \textbf{g}\left(
            S^\textnormal{coh}(\tau,z)(T_a),S^\textnormal{coh}(\tau,z)(T_b)
        \right)
        =
        g(T_a, T_b)
    \]
    Since $\nabla_{z\partial t_i}S^\textnormal{coh}=0$ and the metric $g$ is $\nabla$-flat by Proposition \ref{HGW:prop_QDM_flat}, we have for any $i \in \{ 0, \dots, N \}$
    \[
        \partial_{t_i}
        \textbf{g}\left(
            S^\textnormal{coh}(\tau,z)(T_a),S^\textnormal{coh}(\tau,z)(T_b)
        \right)
        =
        0
    \]
    So the expression $\textbf{g}\left(
            S^\textnormal{coh}(\tau,z)(T_a),S^\textnormal{coh}(\tau,z)(T_b)
    \right)$ is constant in $\tau$.
    At the large radius limit, we have
    \[
        \lim_{\tau' \to 0}
        \textbf{g}\left(
            S^\textnormal{coh}(\tau,z)(T_a),S^\textnormal{coh}(\tau,z)(T_b)
        \right)
        =
        \textbf{g}\left(
            e^{-\tau_2/z}T_a,e^{-\tau_2/z}T_b
        \right)
        =
        g(T_a,T_b)
    \]
\end{proof}

\subsubsection*{Big $J$-function}

\begin{prop}\label{HGW:prop_S_inverse}
    Let $T_a \in H^*(X;\mathbb{Q})$ be a section of $F$. The inverse of the fundamental solution $S^\textnormal{coh}$ is given by
    \[
        \left(S^\textnormal{coh}\right)^{-1}(\tau,z)(T_a) = \sum_{j=0}^N \textbf{g} \left( e^{\tau_2/z}T_j + \sum_{\substack{d \in H_2(X;\mathbb{Z})-\{0\} \\ l \geq 0 }} \sum_k \frac{1}{l!}
        e^{\tau_2(d)}
        \left \langle 
            \frac{e^{\tau_2/z} T_j}{-z+\psi}, T_k, \tau', \dots, \tau'
        \right \rangle^\textnormal{coh}_{0,l+2,d}
        T^k  , T_a \right) T^j
    \]
\end{prop}

\begin{proof}
    According to Theorem \ref{HGW:S_fund_sol}, \textit{(iii)}, the fundamental $S^\textnormal{coh}$ is an isometry for the metric $\textbf{g}$ on $F$. 
    This means that its inverse is given by its adjoint with respect to the metric $\textbf{g}$.
    This means that if $T_a \in H^*(X;\mathbb{Q})$ is a section of $F$, we have
    \[
        S^\textnormal{coh}(\tau,z)^{-1}(T_a)
        =
        \sum_{j=0}^N \textbf{g} \left( S^\textnormal{coh}(\tau,z)(T_j), T_a \right) T^j
    \]
    So we end up having
    \[
        \sum_{j=0}^N \textbf{g} \left( e^{\tau_2/z}T_j + \sum_{\substack{d \in H_2(X;\mathbb{Z})-\{0\} \\ l \geq 0 }} \sum_k \frac{1}{l!}
        e^{\tau_2(d)}
        \left \langle 
            \frac{e^{\tau_2/z} T_j}{-z+\psi}, T_k, \tau', \dots, \tau'
        \right \rangle^\textnormal{coh}_{0,l+2,d}
        T^k  , T_a \right) T^j
    \]
\end{proof}


To motivate the definition of Givental's $J$-function, recall also that the constant section $\mathds{1}$ played a special role in Proposition \ref{HGW:prop_qdm_cyclic}, as it was generating $QDM(X)$ as a cyclic $\mathcal{D}$-module.

\begin{defin}\label{HGW:def_J_fn}
    \textit{Givental's $J$-function} $J^\textnormal{coh}$ is given by the expression
    \[
        J^\textnormal{coh}(\tau,z) = S^\textnormal{coh}(\tau,z)^{-1} \mathds{1} \in QH(X) \otimes \mathbb{C}( \! (z) \! )
    \]
\end{defin}

\begin{prop}[\cite{CK_book}, Lemma 10.3.3]
    Givental's $J$-function $J^\textnormal{coh}$ is given by
    \[
        J^\textnormal{coh}(\tau,z) =
        e^{\tau_2/z}
        \left(
            \mathds{1} + 
            \sum_{\substack{d \neq 0 \\l \neq 0}} \sum_{i=0}^N \frac{1}{l! z} \left\langle \tau', \dots, \tau',
            \frac{T_i}{z-\psi}\right\rangle^\textnormal{coh}_{0,l+2,d} e^{\tau_2(d)} T^i
        \right)
    \]
\end{prop}

\begin{proof}
    Using Proposition \ref{HGW:prop_S_inverse}, we have
    \[
        J^\textnormal{coh}(\tau,z)
        =
        \left(S^\textnormal{coh}\right)^{-1}(\tau,z)(\mathds{1}) = \sum_{j=0}^N \textbf{g} \left( e^{\tau_2/z}T_j + \sum_{\substack{d \neq 0 \\ l \geq 0 }} \sum_k \frac{1}{l!}
        e^{\tau_2(d)}
        \left \langle 
            \frac{e^{\tau_2/z} T_j}{-z+\psi}, T_k, \tau', \dots, \tau'
        \right \rangle^\textnormal{coh}_{0,l+2,d}
        T^k  , \mathds{1} \right) T^j
    \]
    We expand the expression in the sum using the linearity of the metric $\textbf{g}$. First, notice that for any class $\alpha=\sum_k \alpha_k T_k \in H^*(X)$, we have
    \[
        \sum_{j=0}^N g(\alpha,T_j) 
        =
        \sum_{j,k} \alpha_k g_{kj} T^j 
        =
        \sum_k \alpha_k T_k
        =
        \alpha
    \]
    In our setting, using the definition of the metric $\textbf{g}$, we obtain
    \[
        \sum_{j=0}^N \textbf{g}(e^{\tau_2/z}T_j,\mathds{1}) T^j
        =
        \sum_{j=0}^N \textbf{g}(e^{\tau_2/z},T_j) T^j
        =
        e^{\tau_2/z}
    \]
    Next, we are going to compute
    \[
        \sum_{j=0}^N \textbf{g} \left( \sum_{(d,l) \neq 0} \sum_k \frac{1}{l!}
        e^{\tau_2(d)}
        \left \langle 
            \frac{e^{\tau_2/z} T_j}{-z+\psi}, T_k, \tau', \dots, \tau'
        \right \rangle^\textnormal{coh}_{0,l+2,d}
        T^k  , \mathds{1} \right) T^j
    \]
    Because $\mathds{1}$ is the unit in cohomology, if $\textnormal{deg}(T^k) \neq 2 \textnormal{dim}(X)$, we have
    \[
        g(T^k, \mathds{1}) = \int_X T^k = 0
    \]
    The degree will match the dimension if and only if $k=0$, therefore $g(T^k, \mathds{1}) = \delta_{k,0}$.
    Plugging that in our computation gives
    \[
        \sum_{j=0}^N \textbf{g} \left( \sum_{(d,l) \neq 0} \sum_k \frac{1}{l!}
        e^{\tau_2(d)}
        \left \langle 
            \frac{e^{\tau_2/z} T_j}{-z+\psi}, T_k, \tau'^{(l)}
        \right \rangle^\textnormal{coh}_{0,l+2,d}
        T^k  , \mathds{1} \right) T^j
        =
        \sum_{j=0}^N  \sum_{(d,l) \neq 0} \frac{1}{l!}
        e^{\tau_2(d)}
        \left \langle 
            \frac{e^{\tau_2/z} T_j}{-z+\psi}, \mathds{1}, \tau'^{(l)}
        \right \rangle^\textnormal{coh}_{0,l+2,d}
        T^j
    \]
    We are going to apply the String Equation (see \ref{HGW:prop_string_eqn}) to the right hand side. There is no issue since we have $d \neq 0$.
    We have
    \[
        \left \langle 
            \frac{e^{\tau_2/z} T_j}{-z+\psi}, \mathds{1}, \tau'^{(l)}
        \right \rangle^\textnormal{coh}_{0,l+2,d}
        =
        \sum_{n \geq 0} \frac{(-1)^n}{(-z)^{n+1}}
        \left \langle 
            \psi_1^n e^{\tau_2/z} T_j, \mathds{1}, \tau'^{(l)}
        \right \rangle^\textnormal{coh}_{0,l+2,d}
    \]
    \[
        =
        \sum_{n \geq 0} \frac{-1}{z^{n+1}}
        \left \langle 
            \psi_1^{n-1} e^{\tau_2/z} T_j, \tau'^{(l)}
        \right \rangle^\textnormal{coh}_{0,l+1,d}
        =
        \frac{-1}{z} 
        \left \langle 
            \frac{e^{\tau_2/z} T_j}{z+\psi} \tau'^{(l)}
        \right \rangle^\textnormal{coh}_{0,l+1,d}
    \]
    Doing a base change $T_j \mapsto e^{-\tau_2/z}T_j$ gives the desired formula.
\end{proof}

These definitions can be summed up in the diagram below.
\begin{center}
    \begin{tikzcd}
        (F,\text{d},\textbf{g})
            \arrow[rr,"\sim"',"S^\textnormal{coh}"]
            \arrow[dr]
        &
        &(F, \nabla,\textbf{g})
            \arrow[dl]
        \\      
        &U \times \mathbb{P}^1
            \arrow[ur, bend right=15,"\mathds{1}"']
            \arrow[ul, bend left=15,"J^\textnormal{coh}"]
            &
    \end{tikzcd}
\end{center}


\subsubsection*{Small $J$-function}

The constructions of the fundamental solution $S^\textnormal{coh}$ and Givental's $J$-function have their analogue in the small quantum $\mathcal{D}$-module by restricting to $\tau' = 0$.

\begin{defin}
    We define the formal function $s^\textnormal{coh}$ by
    \[
        s^\textnormal{coh}(\tau_2,z) = \left( S^\textnormal{coh}(\tau,z) \right)_{| \tau' = 0}
    \]
\end{defin}

\begin{prop}
    Denote by $\nabla$ the small Dubrovin connection and let $i \in \{1, \dots, r\}$.

    \textit{(i)} The formal function $s^\textnormal{coh}(\tau_2,z)$ satisfies
    \[
        \nabla_{\partial t_i} s^\textnormal{coh}(\tau_2,z) = 0
    \]
    
    \textit{(ii)} The formal function $s^\textnormal{coh}(\tau_2,z) z^{-\mu} z^\rho$ is a fundamental solution of the small quantum $\mathcal{D}$-module $(F_2, \nabla)$
\end{prop}

\begin{defin}
    Givental's small $J$-function $j^\textnormal{coh}(\tau_2,z)$ is given by
    \[
        j^\textnormal{coh}(\tau_2,z)
        =
        s^\textnormal{coh}(\tau_2,z)(\mathds{1})
    \]
\end{defin}

Note that $ (\mathbb{C}^*)^{N+1} $ acts on $\mathbb{P}^N$ by $(\lambda_0,\dots,\lambda_N) \cdot [z_0 : \cdots : z_N] = [\lambda_0 z_0 : \dots : \lambda_N z_N]$.
Using Atiyah--Bott fixed point localisation, it is possible \cite{Givental_EquivariantGW, Bertram_CF_Kim_MS_equivariant} to find an explicit formula for the small $J$-function of toric varieties.

\begin{prop}[\cite{Givental_EquivariantGW}]\label{HGW:ex_projspace_Jfn}
    In the case $X = \mathbb{P}^N$, 
    recall that $H^*(\mathbb{P}^N;\mathbb{Q}) \simeq \mathbb{Q}[H]\left/ (H^{N+1})\right.$, where $H=c_1 ( \mathcal{O}(1))$ is the hyperplane class.
    The small $J$-function is given by
%
    \[
        j^\textnormal{coh}(t_1,z) = e^{\frac{t_1 H}{z}} \sum_{d \geq 0} \frac{e^{t_1 d}}{\prod_{r=1}^d \left( H + rz\right)^{N+1}}
    \]
    The expression $\frac{1}{H-rz}$ should be understood as its power series expansions
    \[
        \frac{1}{H-rz}
        =
        \frac{-1}{rz} \sum_{m \geq 0} \left( \frac{H}{rz} \right)^m
        =
        \frac{-1}{rz} \sum_{m = 0}^N \left( \frac{H}{rz} \right)^m
    \]
    Furthermore, the small $J$-function satisfies the differential equation
    \[
        \left[ (z \partial_{t_1})^{N+1} - e^{t_1} \right] j^\textnormal{coh}(t_1,z) = 0
    \]
\end{prop}

\begin{remark}
    This differential equation corresponds to the relation in the small quantum cohomology of $\mathbb{P}^N$ found in Example \ref{HGW:ex_sqh_projectivespace}. Recall that we had seen
    \[
        H^{\circ_{t_1} (N+1)} = e^{t_1} \in SQH^*\left(\mathbb{P}^N\right)
    \]
    Since we have here $\nabla_{z \partial t_1} = z \partial_{t_1} + H \circ_{t_1}$, using that $s^\textnormal{coh}$ is a fundamental solution for the small quantum $\mathcal{D}$-module, we have
    \[
        \left( H^{\circ_{t_1} (N+1)} - e^{t_1} \right) \mathds{1} = 0
        \iff
        \left[ (z \partial_{t_1})^{N+1} - e^{t_1} \right] j^\textnormal{coh}(t_1,z) = 0
    \]
\end{remark}

\begin{remark}
    There is an alternative structure to the quantum $\mathcal{D}$-module (used e.g. in \cite{Giv_Ton_qk_HRR}) called \textit{Givental's Lagrangian cone}, denoted $\mathcal{L}^\textnormal{coh}_X$, see \cite{Givental_Frobenius_lagrangiancone}.
    The tangent spaces to this cone carry a $\mathcal{D}$-module structure that is identified with the quantum $\mathcal{D}$-module, see \cite{CCIT_Computing_TGW}, Appendix B or \cite{Iri_Mann_Mignon_QSerre}, Subsection 2.4.d.
\end{remark}

\chapter{$K$-theoretical Gromov--Witten invariants}\label{chapter:QK}

We give the elements to describe quantum $K$-theory, by essentially building $K$-theoretic analogues to the constructions of Chapter \ref{chapter:H_GW}.
In the first section, this will result in the construction of a quantum ring $(QK(X),\star_\tau)$ that is a deformation of the $K$-theory ring of a projective variety $X$. However, the properties satisfied by our constructions will be much more complicated than in the cohomological case.

The next section attempts to build a $K$-theoretical analogue of the quantum $\mathcal{D}$-module.
While we are able to construct differential operators acting on quantum $K$-theory, we will not try to define a flat connection analogous to Dubrovin's connection.

The third section adds additional structure to quantum $K$-theory.
Using Givental--Tonita's quantum Hirzebruch--Riemann--Roch theorem (see \cite{Giv_Ton_qk_HRR}, Section 9, Theorem), we will build $q$-difference operators acting on quantum $K$-theory. 

%
%
%

\section{$K$-theoretical Gromov--Witten invariants}

Our objective in this section is to define $K$-theoretic analogues to Gromov--Witten invariants and quantum cohomology.
The main references for this section are \cite{Lee_qk} and \cite{Giv_qk_wdvv}.

\subsection{Definitions}

We recall that just as we used the virtual fundamental class $\left[ \overline{\mathcal{M}}_{g,n}\left(X,d \right) \right]^\textnormal{vir}$, the moduli space of stable maps also an analogous virtual structure sheaf $\mathcal{O} ^ {\text{vir}} _ {\overline{\mathcal{M}}_{g,n}\left(X,d \right)}$ constructed by Y.P Lee in \cite{Lee_qk}.

\begin{defin}[\cite{Lee_qk}]\label{KGW:def_KGWI}
    Let $g,n \in \mathbb{Z}_{\geq 0}$, $d \in H_2(X;\mathbb{Z})$. Let $k_1, \dots, k_n \in \mathbb{Z}_{\geq 0}$ be some integers, and let $\phi_1, \dots \phi_n \in K(X)$.
    A \textit{$K$-theoretical Gromov--Witten invariant} is given by the Euler characteristic
    \[
        \left\langle
            \mathcal{L}_1^{k_1} \phi_1, \cdots, \mathcal{L}_n^{k_n} \phi_n
        \right\rangle_{g,n,\beta}^{K\textnormal{th}}
        =
        \chi \left(
            \overline{\mathcal{M}}_{g,n}\left(X,d \right);
            \mathcal{O} ^ {\text{vir}} _ {\overline{\mathcal{M}}}
            \bigotimes_{i=1}^n \mathcal{L}_i^{k_i} \text{ev}_i^*(\phi_i)
        \right)
        \in \mathbb{Z}
    \]
\end{defin}
We precise that in the left hand side, the notation $\mathcal{L}_i^{k_i} \phi_i$ should not be thought of as the tensor product of two sheaves, as the cotangent line bundle $\mathcal{L}_i \in K\left(\overline{\mathcal{M}}_{g,n}\left(X,d \right) \right)$ and the input $\phi_i \in K(X)$ are not sheaves on the same space.

Recall that a cohomological Gromov--Witten $\langle \psi_{k_1} \alpha_1 , \dots , \psi_{k_n} \alpha_n \rangle^H_{g,n,d}$ needed not be an integer. Here, since we are taking the Euler characteristic of a sheaf, a $K$-theoretical Gromov--Witten is necessarily an integer.

\subsection{Properties of $K$-theoretical Gromov--Witten invariants}

The $K$-theoretic Gromov--Witten invariants satisfy properties similar to the cohomological Gromov--Witten invariants (cf. \ref{subsection:HGW_props}).
However, we should note that the ring $K(X)$ does not come with a graduation, therefore we cannot define a Degree Axiom nor a Divisor Axiom, and the Point Mapping Axiom will be stated in a weaker form.
We suggest that the reader mainly interested in the algebraic structure of $K$-theoretical Gromov--Witten invariants should skip this section on a first lecture.
For simplicity, we will only state the axioms for the genus zero case.
For the proof of these statements, we refer to \cite{Lee_qk}, Subsections 4.3 and 4.4 (see also \cite{Mann_Robalo_DAG}, Theorem 5.4.2).

\begin{prop}[Fundamental class Axiom]
    This property can also be referred to as the String Equation. 
    Let $n,d$ such that $n \geq 4$ or $n>1, d \neq 0$. Let 
    $\pi_{n+1}: \overline{\mathcal{M}}_{0,n+1}(X,d) \to \overline{\mathcal{M}}_{0,n}(X,d)$
    be the universal curve and consider $q_1, \dots, q_n$ be some formal variables.
    We have
    \[
        \pi_* \left(
            \mathcal{O}^\textnormal{vir} \otimes
            \left( 
                \prod_{i=1}^n \frac{1}{1-q_i \mathcal{L}_i}
            \right)
        \right)
        =
        \left(
            1 + \sum_{i=1}^n \frac{q_i}{1 - q_i}
        \right)
        \left(
            \mathcal{O}^\textnormal{vir} \otimes
            \left(
                \prod_{i=1}^n \frac{1}{1-q_i \mathcal{L}_i}
            \right)
        \right)
        \in K \left( \overline{\mathcal{M}}_{0,n}(X,d) \right)
    \]
\end{prop}


\begin{prop}[Dilaton Axiom]
    Let $n,d$ such that $n \geq 4$ or $n>1, d \neq 0$. Let 
    $\pi_{n+1}: \overline{\mathcal{M}}_{0,n+1}(X,d) \to \overline{\mathcal{M}}_{0,n}(X,d)$
    be the universal curve and consider $q_1, \dots, q_n$ be some formal variables.
    We have
    \[
        \pi_\star
        \left(
            \mathcal{O}^\textnormal{vir} \otimes
            \left(
                \prod_{i=1}^{n-1}
            \right)
            \otimes \mathcal{L}_{n-1}
        \right)
        =
        \mathcal{O}^\textnormal{vir} \otimes
        \left[
            \left(
            \sum_{i=1}^{n-1} \frac{1}{1-q_i}
            \right)
            \prod_{i=1}^{n-1} \mathcal{L}_i^{-1}
            \prod_{i=1}^{n-1} \frac{1}{1-q_i \mathcal{L}_i}
        \right]
        \in K \left( \overline{\mathcal{M}}_{0,n}(X,d) \right)
    \]
\end{prop}

\begin{prop}[Point Mapping Axiom]
    Let $\overline{\mathcal{M}}_{g,n}$ be the moduli space of genus $g$ $n$-pointed stable curves, and $\pi : \mathcal{C} \to \overline{\mathcal{M}}_{g,n}$ be its universal curve. Consider the diagram
    \begin{center}
        \begin{tikzcd}
            \overline{\mathcal{M}}_{g,n}(X,0)
                \arrow[rr, equal]
            &
            &\overline{\mathcal{M}}_{g,n} \times X
                \arrow[ddl,"\textnormal{pr}_1"]
                \arrow[ddr,"\textnormal{ev}"]
            \\
            &\mathcal{C}
                \arrow[d,"\pi"]
            \\
            &\overline{\mathcal{M}}_{g,n}
            &
            & X
        \end{tikzcd}
    \end{center}
    We have
    \[
        \mathcal{O} ^ {\text{vir}} _ {\overline{\mathcal{M}}_{g,n}(X,0)}
        =
        \sum_{i \geq 0}
        (-1)^i
        \bigwedge \nolimits^i
        \left(
            pr_1^* R^1\pi_* \mathcal{O}_\mathcal{C}
            \otimes
            \textnormal{ev}^* T_X
        \right)^\vee
    \]
\end{prop}

    

    

%
%
%

\section{Quantum $K$-theory, differential and \lowercase{$q$}-difference operators}

\begin{assumption}
    Since we had assumed that $X$ is smooth, we will identify $K_\circ(X) \simeq K^\circ(X) =: K(X)$.
    Now, we also assume that the $K$-theory $K(X)$ admits a finite basis
    $\phi_0, \dots, \phi_N \in K(X)$ satisfying $\phi_0 = [\mathcal{O}_X]$ and for which there exists an integer $r \in \mathbb{Z}_{>0}$ such that $c_1(\phi_1), \dots, c_1(\phi_r)$ form an integral basis of $A^2(X)/(\textnormal{torsion})$ and $\textnormal{ch}(\phi_{r+1}), \dots, \textnormal{ch}(\phi_{N}) \in A^{\geq 4}(X)$
    
\end{assumption}

\subsection{Quantum $K$-theory}

In this subsection, the main reference is \cite{Lee_qk}.



\begin{notation}
    We keep the notation $(\phi_i)$ for the basis as in the assumption above. We will also write :
    \begin{align*}
        &   \mathds{1}
        &&  \textnormal{The class of the structure sheaf, unit in } K(X) : [\mathcal{O}_X]=\phi_0 ;\\
        &   g
        &&  \textnormal{Pairing on } K(X) \textnormal{ given by } g_{ij} = g(\phi_i,\phi_j) := \chi( \phi_i \otimes \phi_j) ;\\
        &   t_0, \dots, t_n
        &&  \textnormal{coordinates on } K(X) \textnormal{, associated to } \phi_0, \dots, \phi_n ;\\
        &   \tau
        &&  \textnormal{an arbitrary point in } K(X), \tau := \sum_{i=0}^n t_i \phi_i \in K(X) ;\\
        &    Q_1, \dots, Q_r
        &&  \textnormal{Novikov variables associated to } \phi_1, \dots, \phi_r ;\\
        &   Q^d
        &&  \textnormal{shortcut for } Q_1^{(c_1(\phi_1))(d)} \cdots Q_r^{(c_1(\phi_r))(d)}, \textnormal{ where } d \in H_2(X; \mathbb{Z}).
    \end{align*}    
    We recall that $(c_1(\phi_i))(d) = \int_d c_1(\phi_i)$.
    
    The Novikov ring $\mathbb{C}[\![Q]\!]$ is defined by the ring of formal series in $Q_1, \dots, Q_r$ :
     \[
         \mathbb{C}[\![Q]\!] := \left\{
             \sum_{d \in H_2(X;\mathbb{Z})} f_d Q^d \middle| f_d \in \mathbb{C}
         \right\}
     \]
     Again, if a definition depends on the formal variables $t_0, \dots, t_N$, we may say it depends on the symbol $\tau$ instead.
\end{notation}

\subsubsection*{Big quantum $K$-theory}

\begin{defin}[Genus zero potential]
    The \textit{genus zero ($K$-theoretical) Gromov--Witten potential} $\mathcal{F}$ is the generating series
    \[
        \mathcal{F}(\tau,Q) = \sum_{d,n \geq 0} \frac{1}{n!} \langle \tau, \dots, \tau \rangle^{K\textnormal{th}}_{0,n,d} Q^d
        \in
        \mathbb{Z}[ \![t_0, \dots, t_N] \! ] \otimes \mathbb{C}[\![Q]\!]
    \]
\end{defin}

\begin{remark}
    If we try from now on to reproduce the definition of the product as in quantum cohomology, the resulting product would not be associative.
    To fix this issue, Lee--Givental introduce a new metric, which we will call quantum metric.
    Once we replace the usual metric by the new metric, we use the formulas from cohomology.
\end{remark}


\begin{defin}
    The \textit{quantum metric} is the pairing $G_\tau$ defined on $K(X)$ by
    \[
        G_{ij} = G_\tau(\phi_i, \phi_j) = \partial_{t_0} \partial_{t_i} \partial_{t_j} \mathcal{F}
        \in \mathbb{Z}[ \![t_0, \dots, t_N] \! ] \otimes \mathbb{C}[\![Q]\!]
    \]
    We also denote by $(G^{ij})$ the inverse matrix of the Gram matrix $(G_{ij})$.
\end{defin}

This metric is said to be a quantisation of the usual metric $g$ because we have $\left(G_{ij}\right)_{|Q=0}=g_{ij}$ from Point Mapping Axiom.

\begin{remark}
    If we were to reproduce this definition for quantum cohomology, we would recover $G=g$ because of the Point Mapping Axiom (see Proposition \ref{HGW:prop_point_mapping_axiom}). 
\end{remark}

\begin{thm}[\cite{Giv_qk_wdvv},4, Theorem]\label{KGW:thm_potential_WDVV}
    The genus zero $K$-theoretical Gromov--Witten potential $\mathcal{F}$ satisfies the set of differential equations, called WDVV equations, indexed by $i,j,k,l \in \{0,\dots,N\}$:
    \[
        \sum_{a,b = 0}^N
        \frac{\partial^3 \mathcal{F}}{\partial t_i \partial t_j \partial t_a}
        G^{ab}
        \frac{\partial^3 \mathcal{F}}{\partial t_b \partial t_k \partial t_l}
        =
        \sum_{a,b = 0}^N
        \frac{\partial^3 \mathcal{F}}{\partial t_j \partial t_k \partial t_a}
        G^{ab}
        \frac{\partial^3 \mathcal{F}}{\partial t_b \partial t_i \partial t_l}
    \]
\end{thm}


\begin{defin}\label{KGW:def_qprod}
    The \textit{quantum product} $\star_\tau$ is a product on $K(X)[ \![\tau] \! ] \otimes \mathbb{C}[\![Q]\!]$ defined by
    \[
        G_\tau (\phi_i \star_\tau \phi_j, \phi_k) = \partial_{t_i} \partial_{t_j} \partial_{t_k} \mathcal{F}(\tau,Q)
    \]
    The ring $(K(X)[ \![\tau] \! ] \otimes \mathbb{C}[\![Q]\!],\star_\tau)$ formed by this product will be called \textit{quantum $K$-theory}, denoted by $QK(X)$.
\end{defin}
We have
\[
    \phi_i \star_\tau \phi_j
    =
    \sum_{\alpha,\beta=0}^N \left(
        \partial_{t_i} \partial_{t_j} \partial_{t_\alpha} \mathcal{F}
    \right)
    G^{\alpha \beta} \phi_\beta
\]

\begin{prop}\label{KGW:prop_qprod_commu_asso_etc}
    The quantum product $\star_\tau$ is commutative and associative. Its unit is the structure sheaf $\mathds{1}$. Furthermore, $(K(X)[ \![\tau] \! ] \otimes \mathbb{C}[\![Q]\!],G,\star_\tau)$ is a Frobenius algebra, i.e.
    \[
        G_\tau (\phi_i \star_\tau \phi_j, \phi_k) = G_\tau (\phi_i, \phi_j \star_\tau \phi_k)
    \]
    And we have
    \[
    \left( \phi_i \star_\tau \phi_j \right)_{|Q=0} = \phi_i \otimes \phi_j
    \]
\end{prop}

\begin{proof}
    \begin{enumerate}
        \item \textbf{Commutativity}. This is a consequence of $[\partial_{t_i},\partial_{t_j}]=0$.
        
        \item \textbf{Associativity}. Just like in cohomology, writing the expressions $\phi_i \star_\bullet \left( \phi_j \star_\bullet \phi_k \right)$ and $\left(\phi_i \star_\bullet \phi_j\right) \star_\bullet \phi_k$ will give each side of the $K$-theoretic WDVV equations of Theorem \ref{KGW:thm_potential_WDVV}.
        
        \item \textbf{Unit}. Plugging $i=0$ in the definition $G_\tau (\phi_i \star_\tau \phi_j, \phi_k) = G_\tau (\phi_i, \phi_j \star_\tau \phi_k)$ gives
        \[
            G_\tau (\phi_0 \star_\tau \phi_j, \phi_k)
            =
            \partial_{t_0} \partial_{t_j} \partial_{t_k} \mathcal{F}(\tau,Q)
            =
            G_\tau (\phi_j, \phi_k)
        \]
        
        \item \textbf{Frobenius algebra}. The compatibility relation between the quantum metric $G_\tau$ and the quantum product $\star_\tau$ comes from the definition of the quantum product and the symmetry of the pairing $G$.
        
        \item \textbf{Classical limit}. Using Point Mapping Axiom, we obtain $\left( \phi_i \star_\tau \phi_j \right)_{|Q=0} = \phi_i \otimes \phi_j$.
    \end{enumerate}
\end{proof}

\subsubsection*{Small quantum $K$-theory}

\begin{defin}
    The \textit{small quantum product} $\smallstar_Q$ is a product on $K(X) \otimes \mathbb{C}[\![Q]\!]$ defined by
    \[
        T_i \smallstar_Q T_j = \left( T_i \star_\tau T_j \right)_{\tau=0}
    \]
    We denote by $SQK(X)$ the ring $(K(X) \otimes \mathbb{C}[\![ Q ]\!], \smallstar_Q)$, called \textit{small quantum $K$-theory}. This ring comes with a pairing given by $G_{|t=0}$, called \textit{small quantum metric}.
\end{defin}

\begin{example}
    For $X=\mathbb{P}^N$, let $P=\mathcal{O}(1) \in K(\mathbb{P}^N)$ be the class of the anti-tautological bundle. We have
    \[
        K(\mathbb{P}^N) \simeq
        \mathbb{Z}[P,P^{-1}] \left/\left( (1-P^{-1})^{N+1} \right) \right.
    \]
    We choose therefore the basis given by $\phi_i = (1-P^{-1})^i$, which verifies $c_1 (\phi_1) = H \in H^2(X;\mathbb{Z})$.
    The small quantum $K$-theory is given by
    \[
        SQK(\mathbb{P}^N) \simeq
        \mathbb{C}[P,P^{-1}] \left/\left( (1-P^{-1})^{N+1} - Q \right) \right.
    \]
    For the pairings, let $\phi_i = \left(1-P^{-1}\right)^i$. We have
    \begin{align*}
        g_{ij} &= \left\{
            \begin{aligned}
                &1 && \textnormal{if } i+j \leq N \\
                &0 && \textnormal{otherwise}
            \end{aligned}
        \right.\\
        \left(G_{ij}\right)_{|t=0} &= g_{ij} + \frac{Q}{1-Q}
    \end{align*}
    The computation of the small quantum product can be found in \cite{Buch_Mihalcea_QK_grassmannians}, Section 5.
    The computation for small quantum metric relies on the small $J$-function; it can be found in \cite{Iri_Mil_Ton_qk}, Corollary 4.3.
\end{example}

\subsection{Differential operators on quantum $K$-theory}

\subsubsection*{Big differential module in quantum $K$-theory}

\begin{assumption}
    We assume that the potential $\mathcal{F}$ is convergent on some open set $U$.
\end{assumption}

\begin{defin}
    Let $q$ be a local coordinate on $\mathbb{P}^1$ at $1 \in \mathbb{A}^1$ and denote by $(t_i)$ the local coordinates on $U$ associated to the basis $(\phi_i)$.
    We define the trivial bundle $F^{K\textnormal{th}}$ by
    \[
        F^{K\textnormal{th}}=U \times \mathbb{P}^1 \times K(X) \to U \times \mathbb{P}^1
    \]
    The \textit{quantum connection} is the differential operator $\nabla$ acting on the sections of $F^{K\textnormal{th}}$ by
    \[
        \nabla_{(1-q)\partial_{t_i}} = (1-q) \partial_{t_i} + \phi_i \star_\tau
    \]
\end{defin}
As there is no definition for a derivative along the direction $q$, this is not exactly a connection. We will still use the language of connections to describe these operators.

\begin{remark}\label{KGW:rmk_on_q_and_z}
    When comparing the formulas for the operators $\nabla_{\partial t_i}$ in quantum cohomology and quantum $K$-theory, we can observe that the parameter $z$ in quantum cohomology is replaced by the expression $1-q$ in $K$-theory.
    It is possible to give a geometric meaning to the variables $q$ in quantum $K$-theory and $z$ in quantum cohomology.
    The variable $q$ should be understood as a generator of the $\mathbb{C}^*$-equivariant $K$-theory of a point $K_{\mathbb{C}^*}(pt)$.
    As for $z$, it is a generator of the $\mathbb{C}^*$-equivariant cohomology $H^*_{\mathbb{C}^*}(pt)$.  We have $z = -c_1(q)$.
    Writing $c_\textnormal{tot}$ for the total Chern class, these two variables are related through the formula $c_\textnormal{tot}(q)=1-z$.
    For more on this comparison, see \cite{Iri_Mil_Ton_qk}, Subsection 2.6.
\end{remark}

\begin{prop}[\cite{Giv_qk_wdvv}, Corollary 2]
    The quantum connection is flat, i.e. for $q \neq 1$, 
    \[
        \left[\nabla_{(1-q)\partial_{t_i}},\nabla_{(1-q)\partial_{t_j}}\right]=0
    \]
\end{prop}

\begin{proof}
    We have to compute the three commutators $[\partial_{t_i},\partial_{t_j}], \left[ \partial_{t_i}, \phi_j \star_\tau \right]$ and $\left[ \phi_i \star_\tau, \phi_j \star_\tau \right]$.
    
     We have $[\partial_{t_i},\partial_{t_j}] = 0$ by definition.
    Next, for all $k \in \{0, \dots, N \}$, we have 
     \[
        \partial_{t_i} \phi_j \star_\tau \phi_k
        =
        \sum_{\alpha,\beta \in \{0, \dots, N \}} \left( 
            \left(\partial_{t_i} \partial_{t_j} \partial_{t_k} \partial_{t_\alpha} \mathcal{F}(\tau,Q) \right) G^{\alpha \beta}
            +
            \left( \partial_{t_j} \partial_{t_k} \partial_{t_\alpha} \mathcal{F}(\tau,Q) \right) \partial_{t_i} G^{\alpha \beta}
        \right)  \phi_\beta
     \]
     The expression $\partial_{t_i} \partial_{t_j} \partial_{t_k} \partial_{t_\alpha} \mathcal{F}(\tau,Q) G^{nm}$ is invariant by permutation of the inputs $i,j,k$.
    Then, we have
    \[
        \left(\partial_{t_j} \partial_{t_k} \partial_{t_\alpha} \mathcal{F}(\tau,Q) \right) \left( \partial_{t_i} G^{\alpha \beta} \right)
        =
        \left( \partial_{t_j} \partial_{t_k} \partial_{t_\alpha} \mathcal{F}(\tau,Q) \right)
        \sum_{\alpha',\beta'}  
        \left( \partial_{t_i} G_{\alpha' \beta'} \right)
        G^{\alpha \alpha'} G^{\beta \beta'}
    \]
    Using String Equation, we have $\partial_{t_i} G_{\alpha' \beta'} = \partial_{t_i} \partial_{t_{\alpha'}} \partial_{t_{\beta'}}$. Finally, we obtain
    \[
        \left(\partial_{t_j} \partial_{t_k} \partial_{t_\alpha} \mathcal{F}(\tau,Q) \right) \left( \partial_{t_i} G^{\alpha \beta} \right)
        =
        \left( \partial_{t_j} \partial_{t_k} \partial_{t_\alpha} \mathcal{F}(\tau,Q) \right)
        \sum_{\alpha',\beta'}  
        \left( \partial_{t_i} \partial_{t_{\alpha'}} \partial_{t_{\beta'}} \mathcal{F}(\tau,Q) \right)
        G^{\alpha \alpha'} G^{\beta \beta'}
    \]
    The WDVV equations (see \ref{KGW:thm_potential_WDVV}) imply that the right hand side is invariant by permutation of the inputs $i,j,k$, so we get
    \[
        \left[ \partial_{t_i}, \phi_j \star_\tau \right] \phi_k = 0
    \]
    Lastly, using Proposition \ref{KGW:prop_qprod_commu_asso_etc}, the associativity of the quantum product of Porposi gives
    \[
        \left[ \phi_i \star_\tau, \phi_j \star_\tau \right] \phi_k = 0
    \]
\end{proof}

\begin{prop}[\cite{Lee_qk}, Proposition 12]
    Denote by $\bigtriangledown^{(\textnormal{LC})}$ the Levi--Civita of the metric $\textbf{G}_\tau$.
    We have
    \[
        2 \bigtriangledown^{(\textnormal{LC})}
        =
        \varphi^* \nabla_{|q=-1}
    \]
    In particular, the metric $G_\tau$ is flat.
\end{prop}

\begin{proof}
    We define the Christofell symbols $\Gamma_{ij}^k$ by $\bigtriangledown^{(\textnormal{LC})}_{\partial t_i} t_j = \sum_{k=0}^N \Gamma_{ij}^k t_k$.
    The definition of the Livi--Civita metric gives
    \[
        \Gamma_{ij}^k 
        =
        \sum_{l=0}^N G^{kl} \frac{1}{2}
        \left(
            \partial_{t_j} G_{il} +
            \partial_{t_i} G_{jl} -
            \partial_{t_l} G_{ij}
        \right)
    \]
    Notice that by String Equation, we have $\partial_{t_j}G_{il} = \partial_{t_j}\partial_{t_0}\partial_{t_i}\partial_{t_l} \mathcal{F} = \partial_{t_j}\partial_{t_i}\partial_{t_l} \mathcal{F}$. Plugging that in the computation of our Christoffel symbol gives
    \[
        \Gamma_{ij}^k 
        = \frac{1}{2} \sum_{l=0}^N G^{kl} \partial_{t_i} \partial_{t_j} \partial_{t_k} \mathcal{F}
    \]
    We have recovered the coefficients of the $K$-theoretical quantum product:
    \[
        \phi_i \star_\tau \phi_j = \sum_{k=0}^N 2 \Gamma_{ij}^k \phi_k
    \]
    Therefore, we obtain $2\bigtriangledown^{\textnormal{LC}} = \varphi^* \nabla_{|q=-1}$.
    Since the connections $\nabla$ are flat for all $q \neq 1$, the Livi--Civita $\bigtriangledown^{(\textnormal{LC})}$ is flat, i.e. the metric $\textbf{G}_\tau$ is flat.
\end{proof}

\begin{remark}
    The metric $g$ on $U$ is also flat, its Levi--Civita connection is the trivial connection $d$.
\end{remark}

\begin{defin}
    We define the sesquilinear pairings $\textbf{g}$ and $\textbf{G}_\tau$ are extended $F^{K\textnormal{th}}$ by setting, for any sections $\Phi_1(q), \Phi_2(q) \in QK(X)\otimes \mathbb{C}[\![q,q^{-1}]\!]$,
    \begin{align*}
        \textbf{g}\left(\Phi_1(q), \Phi_2(q)\right) = g\left(\Phi_1(q^{-1}), \Phi_2(q)\right)
        && \textbf{G}_\tau\left(\Phi_1(q), \Phi_2(q)\right) = G_\tau \left(\Phi_1(q^{-1}), \Phi_2(q)\right)
    \end{align*}
\end{defin}
The involution of $\mathbb{P}^1$ given by $q \mapsto q^{-1}$ corresponds to the involution $\iota : z \mapsto -z$ in quantum cohomology, following the relation $\textnormal{ch}(q)=e^{-z}$ of Remark \ref{KGW:rmk_on_q_and_z}.
This time, the fixed points of the involution are $q=1$ and $q=-1$.

\begin{prop}[\cite{Iri_Mil_Ton_qk}, Proposition 2.3]
    The endomorphism $S^{K\textnormal{th}} : (F^{K\textnormal{th}},\textbf{g}) \to (F^{K\textnormal{th}},\textbf{G})$ is an isometry, i.e. for all $i,j \in \{0, \dots, N\}$, we have
    \begin{equation}\label{KGW:equation_S_isometry}
        \textbf{G}_\tau\left( S^{K\textnormal{th}}(\phi_i), S^{K\textnormal{th}}(\phi_j) \right) = \textbf{g}(\phi_i, \phi_j)
    \end{equation}
    Moreover, we have
    \[
        T^{K\textnormal{th}}=\left(S^{K\textnormal{th}}\right)^{-1}
    \]
\end{prop}

\begin{proof}
    For all $i,j,k \in \{0, \dots, N\}$, 
\end{proof}

We want to compare the Levi--Civita connection of the metric $G_\tau$ on $U$ with the quantum connection.
Denote by $F^{K\textnormal{th}}_{|q=-1}$ the restriction of the trivial bundle $F^{K\textnormal{th}}$ to the hypersurface $U \times \{ q = -1\} \simeq U$.
We have an isomorphism $\varphi$ between the trivial bundles $TU$ and $F^{K\textnormal{th}}_{|q=-1}$ over $U$ given by (for their sections)
\[
    \functiondesc{\varphi}
    {TU}
    {F^{K\textnormal{th}}_{|q=-1}}
    {\partial_{t_i}}
    {\phi_i}
\]

We can build a fundamental for the quantum connection.

\begin{defin}[\cite{Giv_qk_wdvv}]
    For $i,j \in \{ 0, \dots, N\}$, define the formal function $S_{ij}$ by
    \[
        S_{ij}
        =
        g_{ij} +
        \sum_{\substack{n \geq 0 \\ d \in H_2(X;\mathbb{Z})}}
        \frac{1}{n!}
        \left\langle
            \phi_i, \tau, \dots, \tau, \frac{\phi_j}{1-qL}
        \right\rangle^{K\textnormal{th}}_{0,n+2,d} Q^d
    \]
\end{defin}

We will see below that the matrix $(S_{ij})$ defines a fundamental solution of the differential equations associated to the operators $\nabla_{\partial_{t_i}}$.

\begin{lemma}[\cite{Lee_qk}, Theorem 4]\label{KGW:lemma_Sfn_WDVV_plus}
    For all $i,j,k,l \in \{ 0, \dots, N \}$, we have
    \[
        \sum_{\alpha, \beta = 0}^N
        \left(
            \partial_{t_i} \partial_{t_j} \partial_{t_\alpha} \mathcal{F}
        \right)
        G^{\alpha \beta}
        \left(
            \partial_{t_k} S_{\beta l}
        \right)
        =
        \sum_{\alpha, \beta = 0}^N
        \left(
            \partial_{t_i} \partial_{t_k} \partial_{t_\alpha} \mathcal{F}
        \right)
        G^{\alpha \beta}
        \left(
            \partial_{t_j} S_{\beta l}
        \right)
    \]
\end{lemma}

\begin{defin}[\cite{Iri_Mil_Ton_qk}]
    We define the endomorphisms $S^{K\textnormal{th}},T^{K\textnormal{th}} \in \textnormal{End}(QK(X)) \otimes \mathbb{C}(\!(q)\!)$ such that
    \begin{align*}
        \textbf{G}_\tau(S^{K\textnormal{th}}(\phi_i), \phi_j) = S_{ij}
        && \textbf{g}(\phi_i, T^{K\textnormal{th}}(\phi_j)) = S_{ij}
    \end{align*}
\end{defin}

\begin{notation}
    Let $f(q)$ be an expression depending on the coordinate $q$ (e.g. $S_{ij}$). We will write $\overline{f(q)}:=f(q^{-1})$.
\end{notation}

The endomorphisms $S^{K\textnormal{th}}$ and $T^{K\textnormal{th}}$ have the explicit formulas below.
\begin{align*}
    S^{K\textnormal{th}}(\phi_i)
    &=
    \sum_{\alpha,\beta=0}^N \overline{S_{i \alpha}} G^{\alpha \beta} \phi_\beta
    \\
    T^{K\textnormal{th}}(\phi_i)
    &=
    \sum_{\alpha,\beta=0}^N S_{i \alpha} g^{\alpha \beta} \phi_\beta
\end{align*}

\begin{prop}[\cite{Iri_Mil_Ton_qk}, Proposition 2.3]
    The metric $\textbf{G}_\tau$ on $F$ is $\nabla$-parallel, i.e. given two sections $s_1, s_2$ of $F$, we have for all $i \in \{0, \dots, N \}$,
    \[
        \partial_{t_i} \textbf{G}_\tau (s_1,s_2)
        =
        \textbf{G}_\tau \left(
            \nabla_{\partial t_i} s_1, s_2
        \right)
        +
        \textbf{G}_\tau \left(
            s_1, \nabla_{\partial t_i} s_2
        \right)
    \]
\end{prop}

\begin{proof}
    For all $i,j,k \in \{0, \dots, N\}$, we have
    \[
        \textbf{G}_\tau \left(
            \nabla_{\partial t_k} \phi_i, \phi_j
        \right)
        +
        \textbf{G}_\tau \left(
            \phi_i, \nabla_{\partial t_k} \phi_j
        \right)
        =
        \left(
            \frac{1}{1-q} + \frac{1}{1-\overline{q}}
        \right)
        \textbf{G}_\tau (\phi_i, \phi_j \star_\tau \phi_k)
        =
        \partial_{t_i} \partial_{t_j} \partial_{t_k} \mathcal{F}
        =
        \partial_{t_k} \textbf{G}_\tau (\phi_i, \phi_j)
    \]
    Where the first equality uses the Frobenius algebra property, and the third uses String Equation.
\end{proof}

From these definitions we can instantly deduce the following adjunction property of $S^{K\textnormal{th}}$ and $T^{K\textnormal{th}}$ with respect to the metrics $\textbf{G}_\tau$ and $\textbf{g}$ :
\begin{equation}\label{KGW:eqn_fund_sol_adjunction}
    \textbf{G}_\tau(S^{K\textnormal{th}}(\phi_i), \phi_j)
    =
    \textbf{g}(\phi_i, T^{K\textnormal{th}}(\phi_j))
\end{equation}



\begin{thm}[\cite{Lee_qk}, Theorem 4]
    The endormorphisms $S^{K\textnormal{th}}, T^{K\textnormal{th}}$ are fundamental solutions to the set of differential equations indexed by $i\in \{ 0, \dots, N\}$ :
    \begin{align*}
        \nabla_{(1-q) \partial_{t_i}} \left.\right.  _\circ \left.\right. S^{K\textnormal{th}} = S^{K\textnormal{th}} \left.\right.  _\circ \left.\right. (1-q) \partial_{t_i}
        &&
        T^{K\textnormal{th}} \left.\right.  _\circ \left.\right. \nabla_{(1-q) \partial_{t_i}}  =  (1-q) \partial_{t_i} \left.\right.  _\circ \left.\right.  T^{K\textnormal{th}}
    \end{align*}
\end{thm}

%




\begin{remark}
    
    Notice that unit $\mathds{1}$ is not $\nabla$-flat, while it was the case in cohomology. This means that in quantum $K$-theory, we do not have a \textit{Frobenius manifold} with the same properties as in quantum cohomology.
\end{remark}

\begin{defin}\label{kgw:def_j_function}
    \textit{Givental's $K$-theoretical $J$-function} is given by the expression
    \begin{align*}
        J^{K\textnormal{th}}(\tau,q,Q) &= S^{K\textnormal{th}}(\tau,q,Q)^{-1} \mathds{1} \in QK(X) \otimes \mathbb{C}( \! (q) \! )  \\
        &= \mathds{1} + \sum_{(n,d) \geq 0} \sum_{i,j=0}^N
        \frac{1}{n!} \left\langle
            \mathds{1}, \tau, \dots, \tau, \frac{\phi_i}{1-qL}
        \right\rangle^{K\textnormal{th}}_{0,n+2,d} Q^d g^{ij} \phi_j 
    \end{align*}
    Where
    \[
        \left\langle
            \mathds{1}, \tau, \dots, \tau, \frac{\phi_i}{1-qL}
        \right\rangle^{K\textnormal{th}}_{0,n+2,d}
        =
        \sum_{m \geq 0} q^m
        \left\langle
            \mathds{1}, \tau, \dots, \tau, L_{n+2}^m \phi_i
        \right\rangle^{K\textnormal{th}}_{0,n+2,d}
    \]
    \textit{Givental's small $J$-function} is given by restricting the expression above to $\tau=0$.
\end{defin}

So far, our data fit in the diagram below.
\begin{center}
    \begin{tikzcd}
        (F,\text{d},\textbf{g})
            \arrow[rr,"S^{K\textnormal{th}}","\sim"']
            \arrow[dr]
        &
        &(F, \nabla, \textbf{G}_\tau)
            \arrow[dl]
            \\      
        &U \times \mathbb{P}^1
            \arrow[ur, bend right=15,"\mathds{1}"']
            \arrow[ul, bend left=15,"J^{K\textnormal{th}}"]
    \end{tikzcd}
\end{center}

\begin{example}
    for $X=\mathbb{P}^N$, let $P=\mathcal{O}(1) \in K(\mathbb{P}^N)$.
    The $K$-theoretical small $J$-function is given by
    \[
        j^{K\textnormal{th}}(q,Q) = 
        \mathds{1} + \sum_{d \geq 0} \sum_{i=0}^{N-1}
        \left\langle
            \mathds{1}, \frac{(1-P^{-1})^i}{1-q \mathcal{L}}
        \right\rangle^{K\textnormal{th}}_{0,n+2,d} Q^d
        \varphi_i
    \]
    Where $\varphi_i$ is the metric-dual of $(1-P^{-1})^i$ with respect to $g$ :
    \begin{align*}
        \varphi_0 &= (1-P^{-1})^N \\
        \varphi_i &= P^{-(N-i)}(1-P^{-1}) && \textnormal{if } i \neq 0
    \end{align*}
    A computation gives
    \[
        j^{K\textnormal{th}}(q,Q) = \sum_{d \geq 0} \frac{Q^d}{\prod_{r=1}^d \left( 1 - q^r P^{-1} \right)^{N+1}}
    \]
    Where
    \[
        \frac{1}{1-q^r P^{-1}}
        = \frac{1}{(1-q^r) +q^r (1-P^{-1})}
        = \frac{1}{1-q^r} \sum_{m = 0}^{N} \left( \frac{q^r}{1-q^r} (1-P^{-1}) \right)^m
        \in K\left( \mathbb{P}^N \right) \otimes \mathbb{C}(\!(q)\!)
    \]
    The computation of $j^{K\textnormal{th}}$ uses fixed point localisation in equivariant $K$-theory, see \cite{Giv_Lee_qk,Giv_PermEquiv}.
\end{example}

\subsection{Building $q$-shift operators on quantum $K$-theory}\label{kgw:subsection_q_shift_operators_in_QK}

In this subsection we will follow \cite{Iri_Mil_Ton_qk}, Subsection 2.5.

    

    

    


The $q$-difference structure on quantum $K$-theory was first found by A. Givental and Y. P. Lee for flag manifolds in \cite{Giv_Lee_qk}, and they were able to identify it to a difference Toda lattice.
For a general target $X$, it is obtained by A. Givental and V. Tonita in \cite{Giv_Ton_qk_HRR}.

This $q$-difference structure should play the role of the Divisor Axiom in $K$-theory. In \cite{Giv_Ton_qk_HRR}, Givental--Tonita show that because of the Divisor Axiom, one can find some differential operators acting on quantum cohomology. Then, using a Riemann--Roch theorem, they send these operators to quantum $K$-theory and realise they become $q$-difference operators. More precisely, they show that the tangent spaces to Givental's Lagrangian cone in quantum $K$-theory is preserved by $q$-difference operators denoted by $P_j^{-1} \qdeop{Q_j}$ for $j\in \{1,\dots,r\}$ (see the two definitions below for notations).

\begin{defin}
    Let $j \in \{ 1 , \dots, r \}$. We denote by $\qdeop{Q_j}$ the $q$-difference operator that acts on functions $f=f(Q_1, \dots, Q_r)$ by 
    \[
    \left(\qdeop{Q_i} f\right)(Q_1, \dots, Q_r) = f(Q_1, \dots, Q_{j-1}, q Q_j, Q_{j+1}, \dots, Q_r)
    \]
\end{defin}

\begin{defin}[\cite{Iri_Mil_Ton_qk}]
    Let $j \in \{ 1 , \dots, r \}$. Denote by $P_j^{-1} \in \textnormal{End}(K(X))$ the map $\psi \mapsto \phi_j \otimes \psi$. The $q$-shift operator $\mathcal{A}_j$ is given by the expression
    \begin{align*}
        \mathcal{A}_j
        &= 
        S^{K\textnormal{th}}(q,t,Q) \left.\right. _\circ \left.\right. P_j^{-1} \qdeop{Q_j} \left.\right. _\circ \left.\right. \left( S^{K\textnormal{th}} \right)^{-1}(q,t,Q)
        \\
        &=
        S^{K\textnormal{th}}(q,t,Q) \left.\right.  _\circ \left.\right. \left( P_j^{-1} \qdeop{Q_j} \left(S^{K\textnormal{th}} \right)^{-1}\right) (q,t,Q) \left.\right.  _\circ \left.\right. \qdeop{Q_j}
    \end{align*}
    This expression defines an automorphism\footnote{This is a consequence of Givental--Tonita's quantum Riemann--Roch theorem, see \cite{Giv_Ton_qk_HRR}, Section 9, Theorem. This theorem is used in the definition of $\mathcal{A}_j$ to justify using in the right hand sides the endomorphism $S^{K\textnormal{th}}(q,t,Q)$ instead of $S^{K\textnormal{th}}(q,t,qQ)$. See also the diagram below Proposition \ref{KGW:prop_compatibility_metric_qde}.}
    $\mathcal{A}_j$ of $QK(X) \otimes \mathbb{C}[q,q^{-1}]$.

\end{defin}

Our main motivation for introducing that $q$-shift operator is that it preserves flat sections.

\begin{thm}[\cite{Iri_Mil_Ton_qk}, Equation 2, Proposition 2.6]
    The endomorphisms $S$ and $T$ are fundamental solutions to the set of $q$-difference equations indexed by $j \in \{ 1 , \dots, r \}$
    \begin{align}\label{KGW:equation_qde_fund_sol}
        &   S^{K\textnormal{th}} \left.\right.  _\circ \left.\right. P_j^{-1}q^{Q_j \partial_{Q_j}} = \mathcal{A}_j \left.\right.  _\circ \left.\right. S^{K\textnormal{th}}
        &&  P_j^{-1}q^{Q_j \partial_{Q_j}} \left.\right.  _\circ \left.\right. T^{K\textnormal{th}} = T^{K\textnormal{th}} \left.\right.  _\circ \left.\right. \mathcal{A}_j
    \end{align}
    Furthermore, let $i \in \{ 0, \dots, N \}$. The $q$-shift operator $\mathcal{A}_i$ and the derivative $\nabla^{q}_{\partial_{t_i}}$ satisfy the compatibility identity
    \begin{equation}\label{KGW:equation_d_qd_operators_commute}
        \left[ \mathcal{A}_j, \nabla_{\partial_{t_i}} \right] = 0
    \end{equation}
\end{thm}


\begin{prop}[\cite{Iri_Mil_Ton_qk}, Proposition 2.6]\label{KGW:prop_compatibility_metric_qde}
    The compatibility of the metric $G_\tau$ and the $q$-shift operator is given by the identity
    \[
        \qdeop{Q_j} \textbf{G}_\tau(\phi,\psi) = \textbf{G}_\tau(\mathcal{A}_j \phi, \mathcal{A}_j^{-1} \psi)
    \]
\end{prop}

We can add the $q$-shift operators to our previous diagram.
\begin{center}
    \begin{tikzcd}
        (F,\text{d},\textbf{g})
            \arrow[rr,"S^{K\textnormal{th}}","\sim"']
            \arrow[dr]
            \arrow[loop left,"P_j^{-1} q^{Q_j \partial_{Q_j}}"]
        &
        &(F, \nabla, \textbf{G}_\tau)
            \arrow[dl]
            \arrow[loop right, "\mathcal{A}_j"]
            \\      
        &U \times \mathbb{P}^1
            \arrow[ur, bend right=15,"\mathds{1}"']
            \arrow[ul, bend left=15,"J^{K\textnormal{th}}"]
    \end{tikzcd}
\end{center}


\begin{example}
    for $X=\mathbb{P}^N$, let $P=\mathcal{O}(1) \in K(\mathbb{P}^N)$.
    The $K$-theoretical small $J$-function was given by
    \[
        j^{K\textnormal{th}}(q,Q) = \sum_{d \geq 0} \frac{Q^d}{\prod_{r=1}^d \left( 1 - q^r P^{-1} \right)^{N+1}}
    \]
    It satisfies the $q$-difference equation
    \[
        \left[ \left(1 - P^{-1} \qdeop{Q}\right)^{N+1} - Q \right] j^{K\textnormal{th}}(q,Q) = 0 
    \]
    Which should be compared to the relation
    \[
        \left(1 - P^{-1} \right)^{N+1} - Q = 0 \in QK(\mathbb{P}^N)
    \]
\end{example}

To mirror Example \ref{HGW:ex_projspace_Jfn}, it can be convenient to modify the endomorphisms $S,T$ so that in the $q$-difference equations, the operator $P_j^{\pm 1} \qdeop{Q_j}$ is replaced by $\qdeop{Q_j}$. In \cite{Iri_Mil_Ton_qk}, this is done by introducing a shift by $f(q,Q_j)=\phi_j^{\pm \log(Q_j)/\log(q)}$, which satisfies $\qdeop{Q_j} f(q,Q_j) = \phi_j^{\pm 1} f(q,Q_j)$. In this thesis, we will use a different function which has the same purpose: they are both solutions of the $q$-difference equation.

\begin{defin}[\cite{Iri_Mil_Ton_qk}]\label{KGW:def_S_tilde}
    We consider the new endomorphisms $\widetilde{S^{K\textnormal{th}}}, \widetilde{T^{K\textnormal{th}}}$ defined by
    \begin{align*}
        &\widetilde{S^{K\textnormal{th}}} = S^{K\textnormal{th}} \left.\right. _\circ \left.\right. \prod_{j=1}^r \phi_j^{\ell_q(Q_j)}
        &&\widetilde{T^{K\textnormal{th}}} = \prod_{j=1}^r \phi_j^{-\ell_q(Q_j)} \left.\right. _\circ \left.\right. T^{K\textnormal{th}}
    \end{align*}
    Where
    \[
    \phi_j^{-\ell_q(Q_j)}
    = \sum_{k \geq 0} (-1)^k \binom{\ell_q(Q_j)}{k} \left( 1 - \phi_j^{-1} \right)^k
    = \sum_{k \geq 0} (-1)^k \left( \frac{1}{k!} \prod_{r=0}^{k-1}(\ell_q(Q)-r) \right) \left( 1 - \phi_j^{-1} \right)^k
    \]
    And
    \begin{align*}
        \ell_q(Q_j) &= \frac{-Q_j \theta_q'(Q_j)}{\theta_q(Q_j)} \\
        \theta_q(Q_j) &= \prod_{r \geq 0} (1-q^{r+1})(1+q^r Q_j)\left( 1 + \frac{q^{r+1}}{Q_j} \right)
    \end{align*}
    We also do the same for the $J$-function. We set
    \[
        \widetilde{J^{K\textnormal{th}}} = \prod_{j=1}^r \phi_j^{-\ell_q(Q_j)} \left.\right. _\circ \left.\right. J^{K\textnormal{th}}
    \]
\end{defin}
We refer to Subsection \ref{qde:subsection_qde_defs} for more details on this special function. One should understand the product $\prod_{j=1}^r \phi_j^{-\ell_q(Q_j)}$ as the $K$-theoretical analogue of term $e^{-\tau_2/z}$ in the formula of $S^H,J^H$ for cohomological Gromov--Witten theory).
This comparison will be explained in the last chapter.
Meanwhile, we invite the reader to compare the formulas for the small $J$-functions of $\mathbb{P}^2$ given in Examples \ref{qkqde:ex_JK_P2_again} and \ref{qkqde:ex_JH_P2}.

\begin{coro}
    The endomorphism $\widetilde{S}, \widetilde{T}$ are fundamental solutions of the set of $q$-difference equations indexed by $j \in \{ 1 , \dots, r \}$
    \begin{align*}
        &   \widetilde{S^{K\textnormal{th}}} \left.\right.  _\circ \left.\right. q^{Q_j \partial_{Q_j}} = \mathcal{A}_j \left.\right.  _\circ \left.\right. \widetilde{S^{K\textnormal{th}}}
        && q^{Q_j \partial_{Q_j}} \left.\right.  _\circ \left.\right. \widetilde{T^{K\textnormal{th}}} = \widetilde{T^{K\textnormal{th}}} \left.\right.  _\circ \left.\right. \mathcal{A}_j
    \end{align*}
\end{coro}

These results fit in the following diagram
\begin{equation}\label{KGW:diagram_QK}
        \begin{tikzcd}
            (F,\text{d},\textbf{g})
                \arrow[rr,"\widetilde{S^{K\textnormal{th}}}","\sim"']
                \arrow[dr]
                \arrow[loop left,"q^{Q_j \partial_{Q_j}}"]
            &
            &(F, \nabla, \textbf{G}_\tau)
                \arrow[dl]
                \arrow[loop right, "\mathcal{A}_j"]
            \\      
            &U \times \mathbb{P}^1
                \arrow[ur, bend right=15,"\mathds{1}"']
                \arrow[ul, bend left=15,"\widetilde{J^{K\textnormal{th}}}"]
        \end{tikzcd}
\end{equation}



%
%
%















\chapter{Regular singular $q$-difference equations}\label{chapter:regsingqde}

In this chapter, we begin by giving a brief overview of (analytic) $q$-difference equations, which the reader might not be familiar with.
We want to give the results necessary to understand the main theorem which relates the $q$-difference module in quantum $K$-theory with the differential module in quantum cohomology.
This chapter is organised as follows:  
\begin{itemize}
    \item In Section \ref{qde:section_examples}, we work some examples to motivate the general theory.
    \item In Section \ref{qde:section_qde_survey}, we introduce the main definitions for $q$-difference equations. Then, we focus on the class of regular singular $q$-difference equations. We explain how they are solve and their confluence phenomenon.
    \item In Section \ref{qde:section_monodromy}, we discuss the analogue of monodromy of regular singular $q$-difference equations. This section is not necessary for the understanding of the main theorem.
\end{itemize}


%
%
%

\section{First two examples}\label{qde:section_examples}

In this section we will work on two $q$-difference equations.
The aim of the subsection is to introduce the basics of the analytical theory of $q$-difference equations: the space of functions in which we look for solutions, special functions needed to build solutions, and the analogue of the characters $\left( Q \mapsto Q^\mu \right)$.

\subsection{Finding solutions which are $q$-characters}\label{qde:subsection_ex_q_characters}

\begin{defin}\label{qde:def_q_difference_linear_sys}
Let $\mathcal{M}(\mathbb{C})$ be the field of meromorphic functions on $\mathbb{C}$.
Fix $q\in \mathbb{C}, |q|<1$ and $n \in \mathbb{Z}_{>0}$. Let $q^{Q \partial_Q}$ be the $q$-difference operator acting on functions $f: \mathbb{C} \to \mathbb{C}$ by $\left(q^{Q \partial_Q} f\right)(Q) = f(q Q)$. A \textit{linear $q$-difference system} is a functional equation
\begin{equation*}
    q^{Q \partial_Q} X_q(Q) = A_q(Q) X_q(Q)
\end{equation*}
where $X_q$ is a column vector of $n$ complex functions of input $Q$, and $A \in \text{M}_n (\mathcal{M}(\mathbb{C}))$.

The rank of this $q$-difference system is defined to be the rank of the matrix $A_q$.
\end{defin}

\begin{remark}
    
    Let us discuss briefly our choice to take $q\in \mathbb{C}, |q|<1$.
    Suppose that $q$ is a $r^\textnormal{th}$ root of unity, for some $r \in \mathbb{Z}_{\geq 0}$.
    Then, we have $\left( \qdeop{Q} \right)^{r} = \textnormal{Id}$.
    Therefore, having $A_q^r = \textnormal{I}_n$ is a necessary condition for the $q$-difference equation $q^{Q \partial_Q} X_q(Q) = A_q(Q) X_q(Q)$ to have a non trivial solution.
    To avoid having to deal with this condition, we set $|q| \neq 1$ and choose the inner side of the unit circle.
    Our results will still hold on the outer side after replacing $q$ by $q^{-1}$.
    We could also consider the case $|q|=1$, $q$ not root of unity, however it is more technical. We refer to \cite{DiVizio_q1}
\end{remark}

Let $\lambda_q \in \mathbb{C}^*$ be some complex number, which may depend of the parameter $q$.
In this subsection, our goal is to understand solutions of the rank 1 $q$-difference equation
\begin{equation}\label{qde:eqn_q_character}
    \qdeop{Q} f_q(Q) = \lambda_q f_q(Q)
\end{equation}

\begin{defin}
    A solution of the $q$-difference equation $\qdeop{Q} f_q(Q) = \lambda_q f_q(Q)$ is called a $q$\textit{-character}.
\end{defin}
\begin{remark}\label{qde:remark_qchar_are_qdeformation_of_chars}
    To see the necessity for $\lambda_q$ to depend on $q$, and to motivate the choice of name in the definition above, we will see in Proposition \ref{qde:prop_specfns_limits} that if $\lim_{q \to 1} \frac{\lambda_q - 1}{q-1} = \mu \in \mathbb{C}^*$, then
    \[
        \lim_{q \to 1} e_{q, \lambda_q}(Q) = Q^\mu
    \]
    This statement is actually not true in general and needs to be refined in most situations.
    However, in the next subsection, we will see an example where this kind of limit can be directly computed.
\end{remark}

\subsubsection{Solutions in $\mathbb{C}[\![Q]\!][Q^{-1}]$}

\begin{prop}\label{qde:prop_qchars_with_no_essential_sing}
    The $q$-difference equation (\ref{qde:eqn_q_character}) only has non trivial solutions in $\mathbb{C}[\![Q]\!][Q^{-1}]$ if $\lambda_q = q^k$ for some integer $k \in \mathbb{Z}$.
    If this condition is satisfied, then the $\mathbb{C}$-vector space of solutions is spanned by the function
    \[
        f_q(Q) = Q^k
    \]
\end{prop}

\begin{proof}
    Assume that $f_q(Q) = \sum_{d \geq k} f_d(q) Q^d$ for some $k \in \mathbb{Z}$. If $f_q$ is a solution of the $q$-difference equation (\ref{qde:eqn_q_character}), then for all $d \geq k$, $q^d f_d = \lambda_q f_d$. This is only possible if $f_q=0$ or if $\lambda_q = q^{k_0}$ for some $k_0 \in \mathbb{Z}$, and then $f_q(Q) = Q^{k_0}$ up to a constant.
\end{proof}

To get non trivial solutions for every $\lambda_q \neq 0$, we need to look for solutions in a bigger space: we are going to allow our solutions to have an essential singularity at $Q=0$

\begin{notation}
    We denote by $\mathcal{M} \left( \mathbb{C}^* \right)$ the field of meromorphic function on $\mathbb{C}^*$.
    The germ of such functions at $0$ is denoted by $\mathcal{M}(\mathbb{C}^*,0)=\mathbb{C}\{Q,Q^{-1}\}$, the space of convergent Laurent series defined on a punctured disk at $0$.
\end{notation}

\subsubsection{Solutions in $\mathcal{M} \left( \mathbb{C}^* \right)$}

We will build solutions using a special function which we introduce below (note that it is not a $q$-character).

\begin{defin}[\cite{Mum_tata_1}]\label{qde:def_theta}
    \textit{Jacobi's theta function} $\theta_q$ is the complex function defined by
    \[
        \theta_q(Q) = \sum_{d \in \mathbb{Z}} q^\frac{d(d-1)}{2} Q^d
    \]
\end{defin}
Since $|q|<1$, this defines a convergent Laurent series
$\theta_q(Q) \in \mathbb{C}\{Q,Q^{-1}\}$.
The computation of the convergence rays give that the function $\theta_q$ is defined for any $Q \neq 0$.

\begin{prop}[\cite{Mum_tata_1}]\label{qde:prop_theta_qde}
    Jacobi's theta function $\theta_q$ is a solution of the rank 1 $q$-difference equation
    \[
        \qdeop{Q} \theta_q(Q) = \frac{1}{Q} \theta_q(Q)
    \]
\end{prop}

\begin{proof}
    We compute each side of this equality individually.
    \begin{align*}
        \frac{1}{Q} \theta_q(Q)
        &=
        \sum_{d \in \mathbb{Z}}
        q^\frac{d(d-1)}{2} Q^{d-1}
        =
        \sum_{d \in \mathbb{Z}}
        q^\frac{(d+1)d}{2} Q^d
        \\
        \qdeop{Q} \theta_q(Q) 
        &=
        \sum_{d \in \mathbb{Z}}
        q^\frac{d(d-1)}{2} (qQ)^d
        =
        \sum_{d \in \mathbb{Z}}
        q^\frac{(d+1)d}{2} Q^d
    \end{align*}
\end{proof}

\begin{defin}[\cite{Sauloy_qde_regsing}]\label{qde:def_e_qc_qchar}
    We define the function $e_{q,\lambda_q} \in \mathcal{M} \left( \mathbb{C}^* \right)$ by
    \begin{equation*}\label{qde:eqn_qchar}
        e_{q,\lambda_q}(Q) = \frac{\theta_q(Q)}{\theta_q(\lambda_q Q)} \in \mathcal{M} \left( \mathbb{C}^* \right)
    \end{equation*}
\end{defin}


\begin{lemma}\label{qde:lemma_qdeop_is_ring_automorphism}
    Denote by $\textnormal{Fun}(\mathbb{C})$ the ring of complex functions.
    The $q$-difference operator $\qdeop{Q} : \textnormal{Fun}(\mathbb{C})\to \textnormal{Fun}(\mathbb{C})$ is a ring automorphism.
\end{lemma}

\begin{prop}[\cite{Sauloy_qde_regsing}]\label{qde:prop_eqc_is_a_qchar}
    The function $e_{q,\lambda_q}$ is a solution of the $q$-difference equation
    $
    \qdeop{Q} f_q(Q) = \lambda_q f_q(Q)
    $.
\end{prop}

\begin{proof}
    Using Proposition \ref{qde:prop_theta_qde} and Lemma \ref{qde:lemma_qdeop_is_ring_automorphism}, we have
    \[
        \qdeop{Q} e_{q,\lambda_q}(Q)
        =
        \frac{\lambda_q Q \theta_q(Q)}{Q \theta_q(\lambda_q Q)}
        =
        \lambda_q e_{q,\lambda_q}(Q)
    \]
\end{proof}

We have exhibited a non trivial solution $e_{q,\lambda_q}$ of the $q$-difference equation (\ref{qde:eqn_q_character}).
To find the others, notice that we can obtain another solution by multiplying $e_{q,\lambda_q}$ with a solution of the $q$-difference equation $\qdeop{Q}f_q(Q) = f_q(Q)$.
Let us describe the solution space of $\qdeop{Q}f_q(Q) = f_q(Q)$.

\subsubsection{The case $\lambda_q$=1}

\begin{defin}
    A function $f$ which is a solution of the $q$-difference equation $\qdeop{Q} f_q(Q) = f_q(Q)$ will be called a $q$-\textit{constant}.
\end{defin}

\begin{remark}
    The Remark \ref{qde:remark_qchar_are_qdeformation_of_chars} in this case ($\mu = 0$) says that the limit when $q$ tends to 1 of $q$-constants will give us constant functions.
\end{remark}

\begin{prop}\label{qde:prop_qconstants_are_merofns_on_ellipticcurve}
    The solution space of the $q$-difference equation $\qdeop{Q} f_q(Q) = f_q(Q)$ is given by the space of meromorphic functions on the elliptic curve $\mathbb{C}^* / q^\mathbb{Z}$.
\end{prop}

\begin{proof}
    We look for functions $f_q$ that are meromorphic on $\mathbb{C}^*$ that satisfy for any $k\in \mathbb{Z}, f_q(q^kQ)=f_q(Q)$.
    We can identify the solutions to this $q$-difference equation with functions on the quotient $\mathbb{E}_q:=\mathbb{C}^*/q^\mathbb{Z}$, the action of $q^\mathbb{Z}$ being given by the multiplication $q^k \cdot z = q^kz$.
    
    The space $\mathbb{C}^* / q^\mathbb{Z}$ has the structure of an elliptic curve. Take $\tau \in \mathbb{H}$ such that $e^{2i \pi \tau} = q$.
    The elliptic curve $\mathbb{C}^* / q^\mathbb{Z}$ is related to the usual elliptic curve $\mathbb{C}/\left(\mathbb{Z}\cdot1 \oplus \mathbb{Z} \cdot \tau\right)$ by
    \begin{center}
        \begin{tikzcd}
            \mathbb{C}
                \arrow[r,"\exp"]
                \arrow[d]
            &
            \mathbb{C}^*
                \arrow[d]
            \\
            \frac{\mathbb{C}}{\mathbb{Z} \cdot (1, \tau)}
                \arrow[r,"\widetilde{\exp}","\sim"']
            &
            \frac{\mathbb{C}^*}{q^\mathbb{Z}}=:\mathbb{E}_q
        \end{tikzcd}
    \end{center}
    Therefore, the solutions to the $q$-difference equation $\qdeop{Q}f(Q)=f(Q)$ which are meromorphic on $\mathbb{C}^*$ is given by the space of meromorphic functions $\mathcal{M}(\mathbb{E}_q)$.
\end{proof}

\begin{remark}
    This space of functions strictly contains the space of constant functions. To give an example of a non constant function which is a $q$-constant, take $\lambda,\mu \in \mathbb{C}^* - q^\mathbb{Z}, \lambda/\mu \notin q^\mathbb{Z}$, and consider
    \[
        \frac{e_{q,\lambda_q}(Q)e_{q,\mu}(Q)}{e_{q,\lambda_q \mu}(Q)}
    \]    
\end{remark}

Using Lemma \ref{qde:lemma_qdeop_is_ring_automorphism}, we obtain the corollary below.
\begin{coro}\label{qde:coro_sol_space_is_M_Eq_vector_space}
    Let $\qdeop{Q} f_q(Q) = A_q(Q) f_q(Q)$ be a $q$-difference equation.
    Then, the set of its solutions in $\mathcal{M} \left( \mathbb{C}^* \right)$ has the structure of a $\mathcal{M}\left(\mathbb{E}_q \right)$-vector space.
\end{coro}

\subsubsection{Conclusion}

We return to the general case $\lambda_q \neq 0$.

\begin{prop}
    The solution space of the $q$-difference equation (\ref{qde:eqn_q_character}) : $\qdeop{Q} f_q(Q) = \lambda_q f_q(Q)$ is the $\mathcal{M}(\mathbb{E}_q)$-vector space of dimension 1 spanned by $e_{q,\lambda_q}$.
\end{prop}

\begin{proof}
    According to Proposition \ref{qde:prop_eqc_is_a_qchar}, the function $e_{q,\lambda_q}$ is a $q$-character.
    Using Proposition \ref{qde:prop_qconstants_are_merofns_on_ellipticcurve}, the solution space has the structure of a $\mathcal{M}(\mathbb{E}_q)$-vector space.
    What remains is to show that the dimension of this vector space is given by the rank of the $q$-difference equation which we admit for now, cf Proposition \ref{qde:prop_sol_space_had_dim_equal_rank}.
\end{proof}

\begin{remark}
    Later, we will explain how confluence relates the $q$-characters $e_{q, \lambda_q}$ to the characters $\left( Q \mapsto Q^\mu \right)$ of differential equations; see Proposition \ref{qde:prop_specfns_limits}.
\end{remark}

\subsection{An example to introduce confluence}\label{qde:subsection_ex_for_confluence}

The goal of this subsection is to introduce the confluence of $q$-difference equations on an example where the limits are not too technical.
This time, we consider the rank 1 $q$-difference equation
\begin{equation}\label{qde:eqn_ex_confluence}
    \qdeop{Q} f_q(Q) = (1-Q) f_q(Q)
\end{equation}

Let us build a solution in $\mathcal{M} \left( \mathbb{C}^* \right)$.

\begin{defin}
    The \textit{$q$-Pochhammer symbol} is the complex function defined by, for $d \in \mathbb{Z}_{\geq 0}$
    \begin{align*}
        (Q;q)_0 &= 1 \\
        (Q;q)_d &= \prod_{r=0}^{d-1} (1-q^r Q) \\
        (Q;q)_\infty &= \prod_{r \geq 0} (1-q^r Q)
    \end{align*}
\end{defin}

\begin{prop}
    The function $f_q$ defined by
    \[
        f_q(Q) = \frac{1}{(Q;q)_\infty}
    \]
    is a solution of the $q$-difference equation (\ref{qde:eqn_ex_confluence})
\end{prop}

\begin{proof}
    We have
    \[
        \qdeop{Q} (Q;q)_\infty
        =
        \prod_{r \geq 0} (1-q^{r+1}Q;q)
        =
        \frac{1}{1-Q} (Q;q)_\infty
    \]
\end{proof}
We remark that the function $f_q$ has poles at complex numbers of the form $Q=q^k$ for $k \in \mathbb{Z}_{\leq 0}$. 

The idea for confluence of $q$-difference lies in the following computation.
Introduce another $q$-difference $\delta_q$ by
\[
    \delta_q = \frac{\qdeop{Q} - \textnormal{Id}}{q-1}
\]
Then, for any $k \in \mathbb{Z}$, applying the $q$-difference operator $\delta_q$ to the monomial $Q^k$ gives
\[
    \delta_q Q^k 
    =
    \frac{q^k Q^k - Q^k}{q-1}
    =
    \left(
        1 + q + \cdots + q^{k-1}
    \right)
    Q^k
\]
Taking the limit when $q$ tends to $1$ gives
\[
    \lim_{q \to 1} \delta_q Q^k 
    =
    k Q^k
    =
    Q \partial_Q Q^k
\]

\begin{defin}
    We say the $q$-difference operator $\delta_q$ has the \textit{formal limit} $Q \partial_Q$ when $q$ tends to 1 because of the above computation: for any monomial $Q^k$, we have
        \[
        \lim_{q \to 1} \delta_q Q^k 
        =
        Q \partial_Q Q^k
    \]
\end{defin}

Trying to apply this principle to our example, we rewrite the $q$-difference equation (\ref{qde:eqn_ex_confluence}) as
\[
    \frac{\qdeop{Q}-Id}{q-1} f_q(Q) = \frac{Q}{1-q} f_q(Q)
\]
Notice that the coefficient in front of $f_q(Q)$ in the right hand side does not have a limit when $q$ tend to one.
this leads us to introduce instead the $q$-difference equation
\begin{equation}\label{qde:eqn_ex_confluence_confluent_version}
    \qdeop{Q} f_q(Q) = (1 - (1-q)Q) f_q(Q) 
\end{equation}
This time, this $q$-difference equation can be rewritten as
\[
    \frac{\qdeop{Q}-Id}{q-1} f_q(Q) =  Q f_q(Q)
\]
Writing formally $\widetilde{f}(Q) = \lim_{q \to 1} f_q(Q)$, we obtain the differential equation below as a formal limit of the $q$-difference (\ref{qde:eqn_ex_confluence_confluent_version}) when $q$ tends to 1.  
\[
    Q \partial_Q \widetilde{f}(Q) = Q \widetilde{f}(Q)
\]
Furthermore, the $q$-difference (\ref{qde:eqn_ex_confluence_confluent_version}) has solutions spanned over $\mathcal{M} \left( \mathbb{E}_q \right)$ by
\[
    g_q(Q) = f_q\left( (1-q)Q \right) = \frac{1}{((1-q)Q;q)_\infty}
\]
Let us try to compute $\lim_{q \to 1} g_q(Q)$ for a given $Q$ near zero.

\begin{prop}\label{qde:prop_confluence_for_the_example}
    \begin{enumerate}[label=(\roman*)]
        \item For $|Q| < 1$, the function $g_q$ has the Taylor expansion at $Q=0$
        \[
            g_q(Q) = \sum_{d \geq 0} \frac{(1-q)^d}{(q;q)_d} Q^d 
        \]
        
        \item For $|Q|<1$, we have the following pointwise convergence
        \[
            \lim_{q \to 1} g_q(Q) = e^Q
        \]
    \end{enumerate}
\end{prop}

\begin{proof}
        \textit{(i)} We are going to show instead that
        \[
            f_q(Q) = \sum_{d \geq 0} \frac{1}{(q;q)_d} Q^d 
        \]
        We show that the Taylor series of the right hand side $h_q(Q) := \sum_{d \geq 0} \frac{1}{(q;q)_d} Q^d $ is a solution of the $q$-difference equation (\ref{qde:eqn_ex_confluence}), which was the equation $\qdeop{Q} f_q(Q) = (1-Q) f_q(Q)$.
        Indeed, we have
        \begin{align*}
            \qdeop{Q} \left( \sum_{d \geq 0} \frac{1}{(q;q)_d} Q^d  \right)
            &=
            \sum_{d \geq 0} \frac{q^d}{(q;q)_d} Q^d
            \\
            (1-Q) \left( \sum_{d \geq 0} \frac{1}{(q;q)_d} Q^d  \right)
            &=
            1 + \sum_{d \geq 1} \left( \frac{1}{(q;q)_d} - \frac{1}{(q;q)_{d-1}} \right) Q^d
            =
            \sum_{d \geq 0} \frac{1-(1-q^d)}{(q;q)_d}Q^d
        \end{align*}
        Therefore, the functions $f_q$ and $h_q$ are solutions of the same rank 1 $q$-difference equation, so they are related by a $q$-constant by Proposition \ref{qde:prop_sol_space_had_dim_equal_rank}.
        Note that $f_q \in \mathbb{C}\{Q\}$ as it is a convergent limit of inverses of finite $q$-Pochhammer symbols.
        We also have $f_q(0)=h_q(0)=1$.
        Thus, by Proposition \ref{qde:prop_qchars_with_no_essential_sing}, the $q$-constant relating the solutions $f_q$ and $h_q$ is necessarily the constant function $1$.
        
        \bigskip \textit{(ii)}
        We use \textit{(i)} to compute the limit of the function $g_q$. We have
        \[
            \lim_{q \to 1} \frac{(1-q)^d}{(q;q)_d} = \frac{1}{d!}
        \]
        From which we can deduce the wanted result.
\end{proof}

Notice that $\lim_{q \to 1} g_q = \exp$ is indeed a solution of the formal limit differential equation $\partial_Q \widetilde{f}(Q) = \widetilde{f}$. This observation will be formalised in the Theorem \ref{qde:thm_confluence_of_sol_with_initial_cond}.

\section{Survey of the theory of regular singular \lowercase{$q$}-difference equations}\label{qde:section_qde_survey}

The aim of this section is to give definitions related to the study of $q$-difference equations.
After giving basic definitions, we review the theory of regular singular $q$-difference equations. Similarly to differential equations, solutions to a regular singular $q$-difference will exhibit polynomial growth properties at the singularity.
We will begin by showing how to find a basis of solutions for such $q$-difference equation.
Then, we are interested in describing how a $q$-difference equation is a $q$-unfolding of a differential equation, which is a phenomenon referred as \textit{confluence}.
This phenomenon was first described in the easy case of Subsection \ref{qde:subsection_ex_for_confluence}.
The main result will be that if such a $q$-difference equation admits a formal limit when $q \to 1$, then it admits a fundamental solution whose limit is a fundamental solution of the limit differential equation.
The main references for this section are \cite{Sauloy_qde_regsing} and \cite{Har_Sau_Sin_book}.

\subsection{Fundamental solution, $q$-gauge transform and $q$-pullback}\label{qde:subsection_qde_defs}

In this subsection, we define some basic notions in the theory of $q$-difference equations.

\begin{remark}
    We will actually restrict ourselves to finding solutions that are germs at $Q=0$ or $Q=\infty$ for the following reason : if a function $f_q$ is a solution of a $q$-difference equation $\qdeop{Q}f_q(Q) = a_q(Q) f_q(Q)$ and has a singularity at some $Q_0 \neq 0, \infty$, then $f_q$ has a singularity at any complex number $Q_0 q^k$.
\end{remark}

From now on we will be working all the time locally at $Q=0$. Our results will also hold for $Q = \infty$.

\begin{defin}
    Let $(E_q) : q^{Q \partial_Q} X_q(Q)=A_q(Q) X_q(Q)$ be a $q$-difference system, with $A_q \in \textnormal{M}_n \left( \mathcal{M}(\mathbb{C}) \right)$.
    We define the \textit{solution space} of this $q$-difference equation by
    \[
        \textnormal{Sol}\left( E_q \right)
        =
        \left\{
            X_q \in \left( \mathbb{C}\{Q,Q^{-1}\} \right)^n
            \, \middle| \,
            q^{Q \partial_Q} X_q(Q)=A_q(Q) X_q(Q)
        \right\}
    \]
\end{defin}
By Corollary \ref{qde:coro_sol_space_is_M_Eq_vector_space}, this set has the structure of a $\mathcal{M} \left( \mathbb{E}_q \right)$-vector space.

\begin{prop}[\cite{Har_Sau_Sin_book}, Theorem 2.3.1 p.118]\label{qde:prop_sol_space_had_dim_equal_rank}
    Let $(E_q) : q^{Q \partial_Q} X_q(Q)=A_q(Q) X_q(Q)$ be a $q$-difference system.
    Then, we have
    \[
        \textnormal{dim}_{\mathcal{M} \left( \mathbb{E}_q \right)}
        \left(
             \textnormal{Sol}\left( E_q \right)
        \right)
        \leq
        \textnormal{rank}(A_q)
    \]
\end{prop}

\begin{defin}
    Let $q^{Q \partial_Q} X_q(Q)=A_q(Q) X_q(Q)$ be a $q$-difference system.
    A \textit{fundamental solution} of this system is a matrix
    $\mathcal{X}_q \in \text{GL}_n\left(\mathbb{C}\left\{Q,Q^{-1}\right\}\right)$
    such that $q^{Q \partial_Q} \mathcal{X}_q(Q) = A_q(Q) \mathcal{X}_q(Q)$.
\end{defin}

Let us define gauge transforms and pullbacks for $q$-difference equations.
We will only use pullbacks by isomorphisms so we will give an easier definition.

\begin{remark}
    In the differential case, let $\nabla = \textnormal{d} + A$ be a connection on a bundle $F \to X$ (we have $A \in \textnormal{M}_n(\Omega_U)$ where $U$ is a trivialising open for $F$).
    Given a base change $P \in \Gamma(U, \textnormal{GL}_n(\mathcal{O}_X))$, the connection matrix $A$ in the new base becomes
    \[
        P \cdot [A]
        :=
        P^{-1} A P + P^{-1} \textnormal{d}P
    \]
    The matrix $P \cdot [A]$ is called the \textit{gauge transform of} $A$ \textit{by} $P$.
\end{remark}

\begin{defin}
    Let $q^{Q \partial_Q} X_q(Q) = A_q(Q) X_q(Q)$ be a $q$-difference system.
    Consider a matrix $P_q \in \text{GL}_n\left( \mathbb{C}\left\{Q,Q^{-1}\right\} \right)$.
    The \textit{gauge transform} of the matrix $A_q$ by the \textit{gauge transformation} $P_q$ is the matrix
    \[
        P_q \cdot [A_q]
        :=
        \left( q^{Q \partial_Q} P_q \right) A_q P_q^{-1}
    \]
    A second $q$-difference system $q^{Q \partial_Q} X_q(Q) = B_q(Q) X_q(Q)$ is said to be \textit{equivalent by gauge transform} to the first one if there exists a matrix $P_q \in \text{GL}_n\left( \mathbb{C}\left\{Q,Q^{-1}\right\} \right)$ such that
    \[
        B_q = P_q \cdot [A_q]
    \]
\end{defin}

\begin{remark}
    We can retrieve the formula for the $q$-gauge transform by computation.
    Consider a $q$-difference system $q^{Q \partial_Q} X_q(Q) = A_q(Q) X_q(Q)$ and a matrix $P_q \in \text{GL}_n\left( \mathbb{C}\left\{Q,Q^{-1}\right\} \right)$.
    Denote $Y_q = P_q X_q$.
    Then, we have
    \[
        \qdeop{Q} Y_q
        =
        \qdeop{Q} \left( P_q X_q \right)
        =
        \left(\qdeop{Q} P_q \right) \left(\qdeop{Q} X_q \right)
        =
        \left(\qdeop{Q} P_q \right) A_q X_q
        =
        \left[
            \left(\qdeop{Q} P_q \right) A_q P_q^{-1}
        \right]
        Y_q
    \]
\end{remark}

\begin{defin}
    Let $(E_q) : q^{Q \partial_Q} X_q(Q)=A_q(Q) X_q(Q)$ be a $q$-difference system and let $f : \mathbb{C} \to \mathbb{C}$ be an isomorphism.
    The $q$-pullback $(f^* E_q)$ of $(E_q)$ by $f$ is the $q$-difference system given by
    \[
        (f^* E_q) \, : \, \qdeop{Q} X_q(Q) = A_q(f^{-1}(Q)) X_q(Q)
    \]
\end{defin}

\begin{defin}
    A system $q^{Q \partial_Q} X_q(Q) = A_q(Q) X_q(Q)$ is regular if $A_q(0)$ is diagonal and if its eigenvalues are of the form $q^k$ for $k \in \mathbb{Z}_{\geq 0}$
\end{defin}

\begin{example}[Proposition \ref{qde:prop_qchars_with_no_essential_sing}]
    Let $k \in \mathbb{Z}_{\geq 0}$ and consider the $q$-difference equation
    \[
        \qdeop{Q} f_q(Q) = q^k f_q(Q)
    \]
    Then, the solutions of this $q$-difference equation are spanned by the function $\left( Q \mapsto Q^k \right)$.
\end{example}

\begin{defin}\label{qde:def_regular_singular_qde}
    A system $q^{Q \partial_Q} X_q(Q) = A_q(Q) X_q(Q)$ is said to be \textit{regular singular at $Q=0$} if there exists a $q$-gauge transform $P_q \in \text{GL}_n\left( \mathbb{C}\left\{Q,Q^{-1}\right\} \right)$ such that the matrix $\left(P_q \cdot \left[ A_q \right]\right)(0)$ is well-defined and invertible: $P_q \cdot \left[ A_q \right](0) \in \text{GL}_n(\mathbb{C})$.
\end{defin}

\begin{remark}
    In \cite{Har_Sau_Sin_book}, such a $q$-difference system is called \textit{fuchsian} at 0. A $q$-difference system is \textit{strictly fuchsian} at 0 if its associated matrix is already invertible without the need of a $q$-gauge transform.
\end{remark}

\begin{example}[\cite{Iri_Mil_Ton_qk, Giv_Lee_qk}]
    Let us considerate the $q$-difference system obtained from the $q$-difference equation of the small $K$-theoretical $J$-function of $\mathbb{P}^2$. The $q$-difference system given by
    \[
        \qdeop{Q} X_q(Q)
        =
        \begin{pmatrix}
            0   &&  1   &&  0   \\
            0   &&  0   &&  1   \\
            1-Q &&  -3  &&  3   \\
        \end{pmatrix}
        X_q(Q)
    \]
    is regular singular at $Q=0$.
\end{example}

Let us give a criteria for when a $q$-difference equation is regular singular at $Q=0$. This critera is the analogue of Fuchs' condition for differential equation (\cite{Sabbah_book}, Théorème 4.3)

\begin{prop}[\cite{Har_Sau_Sin_book}]\label{qde:prop_when_a_qde_equation_is_regular_singular}
    Let $P\left( \qdeop{Q} \right) = \sum_k^n a_k(q,Q) \left(\qdeop{Q}\right)^k$ be a $q$-difference operator.
    \begin{enumerate}[label=(\roman*)]
        \item
        The $q$-difference equation $P\left( \qdeop{Q} \right) f_q(Q) = 0$ can be vectorised to a $q$-difference system $q^{Q \partial_Q} X_q(Q) = A_q(Q) X_q(Q)$ where $A_q$ is companion matrix of the operator $P$. The resulting $q$-difference system is
        \[
            \qdeop{Q}
            \begin{pmatrix}
                f_q(Q)                                  \\
                \qdeop{Q} f_q(Q)                        \\
                \vdots                                  \\
                \left( \qdeop{Q} \right)^{n-1} f_q(Q)   \\
            \end{pmatrix}
            =
            \begin{pmatrix}
                0               &&  1               &&  0       &&   \cdots     &&  0   \\
                \vdots          &&  0               &&  \ddots  &&              &&  0   \\
                0               &&  \cdots          &&  \cdots  &&   \cdots     &&  1   \\
                -\frac{a_0}{a_n}&& -\frac{a_1}{a_n} &&  \cdots  &&  \cdots      &&  -\frac{a_{n-1}}{a_n}
            \end{pmatrix}
            \begin{pmatrix}
                f_q(Q)                                  \\
                \qdeop{Q} f_q(Q)                        \\
                \vdots                                  \\
                \left( \qdeop{Q} \right)^{n-1} f_q(Q)   \\
            \end{pmatrix}
        \]
    
        \item
        We denote by $val_0 (a_k))$ the $Q$-adic valuation of the polynomial $a_k$, i.e. the lowest integer $\alpha \in \mathbb{Z} \cup \{+\infty\}$ such that $\left(Q^\alpha a_k(Q) \right)_{|Q=0} \neq 0 $.
        The $q$-difference system associated to the $q$-difference equation $P\left( \qdeop{Q} \right) f_q(Q) = 0$ is regular singular if and only if $val_0(a(0))-val_0(a(n))=0$, and for every $k \in \{ 1, \dots, r-1\}, val_0(a(k)))-val_0(a(n)) \geq 0$
    \end{enumerate}
\end{prop}

\subsection{Special functions to solve regular singular $q$-difference equations}

We define some specials functions required to construct fundamental solutions for regular singular $q$-difference systems.
We recall that to build a fundamental solution in the case of a regular singular differential system, we need to use the functions $(Q \mapsto Q^\mu)$ and $(Q \mapsto \log(Q))$.
We are going to give $q$-analogues of these functions.

Recall that in Definition \ref{qde:def_theta}, we defined Jacobi's theta function by
    \[
        \theta_q(Q) = \sum_{d \in \mathbb{Z}} q^\frac{d(d-1)}{2} Q^d
    \]
This function satisfies by Proposition \ref{qde:prop_theta_qde} the $q$-difference equation
    $
        \qdeop{Q} \theta_q(Q) = \frac{1}{Q} \theta_q(Q)
    $.

We also recall the $q$-characters below.
\begin{defin}[cf \ref{qde:def_e_qc_qchar}]
    Let $\lambda_q \in \mathbb{C}^*$. The $q$-character associated to $\lambda_q$ is the function $e_{q,\lambda_q} \in \mathcal{M} \left( \mathbb{C}^* \right)$ defined by
    \begin{equation*}
        e_{q,\lambda_q}(Q) = \frac{\theta_q(Q)}{\theta_q(\lambda_q Q)} \in \mathcal{M} \left( \mathbb{C}^* \right)
    \end{equation*}
\end{defin}
By Proposition \ref{qde:prop_eqc_is_a_qchar}, the function $e_{q, \lambda_q}$ satisfies the $q$-difference equation
$
    \qdeop{Q} e_{q, \lambda_q}(Q) = \lambda_q e_{q, \lambda_q}(Q)
$.

\begin{defin}\label{qde:def_q_logarithm}
    The $q$-logarithm is the function $\ell_q \in \mathcal{M}(\mathbb{C}^*)$ defined by
    \[
        \ell_q(Q) = \frac{-Q \theta_q'(Q)}{\theta_q(Q)}
    \]
\end{defin}

\begin{prop}\label{qde:prop_ell_q_is_a_q_log}
    The function $\ell_q$ is a solution of the $q$-difference equation
    \[
        \qdeop{Q} \ell_q(Q) = \ell_q(Q) + 1
    \]
\end{prop}

\begin{proof}
    We obtain the wanted identity after deriving the $q$-difference equation satisfied by Jacobi's theta function $\theta_q$.
    By Proposition \ref{qde:prop_theta_qde}, we have
    \[
        q \theta_q'(qQ) = \frac{1}{Q} \theta_q'(Q) - \frac{1}{Q^2} \theta_q(Q)
    \]
    Which gives
    \[
        \qdeop{Q} \left(
            \frac{-Q \theta_q'(Q)}{\theta_q(Q)}
        \right)
        =
        \frac{-qQ \left(\frac{1}{qQ}\theta_q'(Q) -\frac{1}{qQ^2}\theta_q(Q) \right)}{\frac{1}{Q} \theta_q(Q)}
        = \ell_q(Q) + 1
    \]
\end{proof}

\begin{remark}
    Note that this $q$-difference equation is not a linear $q$-difference system as in Definition \ref{qde:def_q_difference_linear_sys}.
    To understand $\ell_q$ as a solution of a $q$-difference system, we should instead consider the $q$-difference system
    \[
        \qdeop{Q}
        \begin{pmatrix}
            f_q(Q)  \\
            g_q(Q)
        \end{pmatrix}
        =
        \begin{pmatrix}
            1   &&  0   \\
            1   &&  1   \\
        \end{pmatrix}
        \begin{pmatrix}
            f_q(Q)  \\
            g_q(Q)
        \end{pmatrix}
    \]
    Notice that $f_q$ is a $q$-constant and $g_q$ satisfies $\qdeop{Q}g_q = g_q + f_q$. Therefore, solutions of this $q$-difference system are given by
    $
    \begin{pmatrix}
        1  \\
        0
    \end{pmatrix}
    $
    and
    $
    \begin{pmatrix}
        1  \\
        \ell_q(Q)
    \end{pmatrix}
    $.
\end{remark}

\subsection{Fundamental solution of a regular singular $q$-difference equation}


The strategy to construct a fundamental solution of a regular singular $q$-difference equation is the same as for differential equations.
As a reminder, we give a sketch of this strategy for differential equations.

\subsubsection{Fundamental solution of a regular singular differential equation}

\begin{defin}
    Let $\partial_Q X(Q) = A(Q) X(Q)$ be a differential system, where $A \in \text{M}_n(\mathbb{C}(Q))$.
    A fundamental solution of this differential system is a matrix $\mathcal{X} \in \textnormal{GL}_n \left(\mathbb{C}\{Q\}\left[Q^{-1}\right]\right)$ such that
    \[
        \partial_Q \mathcal{X} (Q) = A(Q) \mathcal{X} (Q)
    \]
\end{defin}

\begin{defin}
    A complex function $f : \mathbb{C} \to \mathbb{C}$ has \textit{polynomial} (or \textit{moderate}) \textit{growth} at $Q=0$ if there exists a neighbourhood $U$ of $0$ and two constants $d \in \mathbb{Z}_{\geq 0}, C \in \mathbb{R}_{>0}$ such that for any $Q \in U$, we have
    \[
        |f(Q)| \leq C |Q|^{-n}
    \]
\end{defin}

\begin{defin}
    Consider a differential system $\partial_Q X(Q) = A(Q) X(Q)$.
    This system is called \textit{regular singular at} $Q=0$ if it admits a fundamental solution which has a pole at $Q=0$, with polynomial growth at this pole on small triangular sector.
\end{defin}

\begin{prop}[\cite{Sabbah_book}, Subsections II.2.1 and II.2.6]
    For a regular singular system, there exists a gauge transform
    $F \in \textnormal{GL}_n\left(\mathbb{C}\{Q\}\left[Q^{-1}\right]\right)$, such that the connection matrix becomes
    \[
    \partial_Q X(Q) = \frac{1}{Q} B(Q) X(Q)
    \]
    where the matrix $B \in \textnormal{M}_n(\mathbb{C}[Q])$ has holomorphic coefficients.
\end{prop}

\begin{remark}\label{qde:remark_counterexample_form_of_regular_singular_differential}
    Consider a differential system $\partial_Q X(Q) = A(Q) X(Q)$.
    The order of the pole at $Q=0$ of the matrix $A$ does not immediately determine the nature of the complex number $0$ as a singularity.
    Indeed, following \cite{Sabbah_book}, Exercise II.4.4, consider the gauge transform 
    \[
        P(Q) =
        \begin{pmatrix}
            Q^{-1} &   0   \\
            0   &   Q^{1}
        \end{pmatrix}
    \]
    The differential system $\partial_Q X(Q) =
    \begin{pmatrix}
        0   &   0   \\
        \frac{1}{Q}   &   0
    \end{pmatrix} X(Q)
    $
    is regular singular, as we see the logarithm as a solution, but its gauge transform by the matrix $P$ is the system given by
    $
    \partial_Q X(Q) =
    \begin{pmatrix}
        Q^{-1}  &   0   \\
        Q^{-3}  &   Q^{-1}
    \end{pmatrix}
    $
    ,which appears to be of order 3.
\end{remark}

To construct a fundamental solution of a regular singular differential system, we are going to consider first the case where the matrix $B(Q)$ is constant in $Q$.
In a second time, we will explain how this gives a fundamental solution for any regular singular differential system.

\begin{lemma}[\cite{Sabbah_book}, Chapter II, 2.6.]\label{qde:lemma_fund_sol_for_almost_constant_differential}
    Consider a regular singular differential system $\partial_Q X(Q) = \frac{1}{Q} B X(Q)$, where the matrix $B$ is constant.
    Take the Jordan-Chevalley decomposition $B=D+N$ where $D$ is semi-simple, $N$ is nilpotent and $D,N$ commute. Write the diagonalization $D=P^{-1} \textnormal{diag}(\lambda_i) P$.
    We define
    \[
    Q^D:=P^{-1} \textnormal{diag}(Q^{\lambda_i}) P
    \]
    and
    \[
    Q^N:=\sum_{k=0}^d \frac{1}{k!}(N \log(Q))^d
    \]
    
    Then, the matrix $Q^D Q^N$ is a fundamental solution of the differential system $\partial_Q X(Q) = \frac{1}{Q} B X(Q)$.
\end{lemma}

\begin{lemma}[\cite{Sabbah_book}, Chapter II, Propositions 2.11 to 2.13]\label{qde:lemma_gauge_for_differential}
    Consider a regular singular differential system $\partial_Q X(Q) = \frac{1}{Q} B(Q) X(Q)$.
    Assume that the difference between any two distinct eigenvalues of the matrix $B(0)$ is not an integer.
    Then, there exists a gauge transform $P\in \textnormal{GL}_n \left(\mathbb{C}\{Q\}\left[Q^{-1}\right]\right)$ which is recursively computable such that
    \[
        P \cdot \left[B(Q)\right] = B(0)
    \]
\end{lemma}

\begin{defin}
    A a regular singular differential system $\partial_Q X(Q) = \frac{1}{Q} B(Q) X(Q)$ is said to be \textit{non resonant} if the difference between any two distinct eigenvalues of the matrix $B(0)$ is not an integer.
\end{defin}

\begin{thm}[\cite{Sabbah_book}, Chapter 2, Theorem 2.8]\label{qde:thm_fund_sol_for_differential}
    Consider a regular singular differential system $\partial_Q X(Q) = \frac{1}{Q} B(Q) X(Q)$.
    Assume that this system is also non resonant.
    Then, there exists a fundamental solution $\mathcal{X} \in \textnormal{GL}_n \left(\mathbb{C}\{Q\}\left[Q^{-1}\right]\right)$ of the regular singular differential system. 
\end{thm}

\begin{proof}
    We apply Lemma \ref{qde:lemma_gauge_for_differential} to our system and obtain a gauge transform $P\in \textnormal{GL}_n \left(\mathbb{C}\{Q\}\left[Q^{-1}\right]\right)$.
    
    Then, we can apply Lemma \ref{qde:lemma_fund_sol_for_almost_constant_differential} to the differential system $\partial_Q X(Q) = \frac{1}{Q} B(0) X(Q)$ and obtain matrices $Q^D, Q^U$, so that the product $Q^D Q^U$ is a fundamental solution of $\partial_Q X(Q) = \frac{1}{Q} B(0) X(Q)$.
    
    Finally, the matrix $P Q^D Q^U =: \mathcal{X}$ is a fundamental solution of the starting 1differential system $\partial_Q X(Q) = \frac{1}{Q} B(Q) X(Q)$.
\end{proof}

\subsubsection{Existence of a fundamental solution for regular singular $q$-difference equations}

\begin{defin}
    We denote by $q^\mathbb{Z}$ (resp. $q^\mathbb{R}$) the \textit{discrete (resp. continuous) }$q$\textit{-spiral} 
    \begin{align*}
        q^\mathbb{Z} := \left\{ q^k \, \middle| \, k \in \mathbb{Z} \right\}  \subset \mathbb{C}
        &&
        q^\mathbb{R} := \left\{ q^k \, \middle| \, k \in \mathbb{R} \right\} \subset \mathbb{C}
    \end{align*}
    For a complex number $\lambda_q \in \mathbb{C}$, we will also use the notation
    \[
        \lambda_q q^\mathbb{Z} := \left\{ \lambda_q q^k \, \middle| \, k \in \mathbb{Z} \right\} \subset \mathbb{C}
    \]
\end{defin}

\begin{defin}
    Consider a regular singular $q$-difference system $q^{Q \partial_Q} X_q(Q) = A_q(Q) X_q(Q)$ and denote by $\left(\lambda_i\right)$ the eigenvalues of the matrix $A_q(0)$.
    This $q$-difference system is said to be \textit{non ($q$-)resonant} if for every $i \neq j$, we have
    $\frac{\lambda_i}{\lambda_j} \notin q^\mathbb{Z}$.
\end{defin}

\begin{thm}[\cite{Sauloy_qde_regsing},1.1.4]\label{qde:thm_qde_frob}
    Let $q^{Q \partial_Q} X_q(Q) = A_q(Q) X_q(Q)$ be a regular singular $q$-difference system.
    Assume that this $q$-difference system is non resonant.
    Then, there exists a fundamental solution of $\mathcal{X}_q \in \text{GL}_n\left(\mathbb{C}\left\{Q,Q^{-1}\right\}\right)$ of this $q$-difference equation expressed with functions $e_{q,\lambda_q}$ and $\ell_q$ of Definitions \ref{qde:def_e_qc_qchar} and \ref{qde:def_q_logarithm}.
\end{thm}

\begin{proof}[Sketch of proof]
    We start by recursively building a gauge transform $F_q \in \text{GL}_n\left( \mathbb{C}\left\{Q,Q^{-1}\right\} \right)$ sending the matrix $A_q(Q)$ to the constant matrix $A_q(0)$.
    See \cite{Sauloy_qde_regsing}, Subsections 1.1.1 and 1.1.3, or \cite{Har_Sau_Sin_book}, Theorems 3.2.2 and 3.2.3 pp.127-129 for the construction.
    
    Then, we find a fundamental solution for the $q$-difference system of constant matrix $\qdeop{Q} X_q = A_q(0) X_q$.
    Take the Jordan-Chevalley decomposition of $A(0)=DU$ where $D$ is semi-simple, $U$ is unipotent and $D,U$ commute.
    Take a basis change $P$ to diagonalise $D=P^{-1 }\text{diag}(\lambda_i) P$. We define
    \[
        e_{q,D}:=P^{-1} \text{diag}(e_{q,\lambda_i}(Q)) P
    \]
    Write $U=e^N$, where $N$ is nilpotent. We define
    \[
        e_{q,U}:=\sum_{d=0}^\infty \frac{1}{k!}(\ell_q(Q) N)^d
    \]
    Then one can check that the product $F_q e_{q,D} e_{q,U} =: \mathcal{X}_q(Q)$ is a fundamental solution of the $q$-difference system $q^{Q \partial_Q} X_q(Q) = A_q(Q) X_q(Q)$.
\end{proof}

Recall that in the case of differential systems, solutions of a regular singular $q$-difference system at $Q=0$ had polynomial growth at that singularity.
Solutions of a regular singular $q$-difference system will also have polynomial growth at $Q=0$.
Before giving a statement, we need to identify the poles of the special functions $e_{q,\lambda_q}$ and $\ell_q$.
We recall the notation for the $q$-Pochhammer symbol: $(Q;q)_\infty := \prod_{r \geq 0} (1-q^r Q)$.

\begin{prop}[Jacobi's triple product identity, \cite{Mum_tata_1}]\label{qde:prop_theta_jacobi}
    For $Q \in \mathbb{C}^*$, the following identity holds
    \[
        \theta_q(Q) = (q;q)_\infty (-Q;q)_\infty \left( \frac{-q}{Q} ; q \right)_\infty
    \]
\end{prop}

\begin{coro}[\cite{Sauloy_qde_regsing}]
    \begin{enumerate}[label=(\roman*)]
        \item Jacobi's theta function $\theta_q$ has an essential singularity at $0, \infty$. Its zeroes are all single and at every complex number $-q^k$ for $k \in \mathbb{Z}$.
        
        \item Let $\lambda_q \in \mathbb{C}$. The poles of the $q$-character $e_{q, \lambda_q}$ are all simple and are given by the set $-\lambda_q^{-1} q^\mathbb{Z}$.
        
        \item The poles of the $q$-logarithm $\ell_q$ are all simple and are given by the set $- q^\mathbb{Z}$.
    \end{enumerate}
\end{coro}

At last, we can give a statement on the growth of the solutions of a regular singular $q$-difference system.
These solutions will have polynomial growth along the continuous $q$-spirals $\nu q^\mathbb{R}$.
However we have to make sure the $q$-spiral we choose does not contain a pole of the solution.

\begin{prop}[\cite{Har_Sau_Sin_book}, Theorem 3.1.7 p.127]
    Let $\qdeop{Q} X_q(Q) = A_q(Q) X_q(Q)$ be a regular singular $q$-difference system of rank $n$ which is not resonant.
    For $i \in \{ 1, \dots, n\}$, denote by $X_{(i)}$ the $i^\textnormal{th}$ column of the fundamental solution given by Theorem \ref{qde:thm_qde_frob}.
    Choose $\nu \in \mathbb{C}, |\nu|=1$ such that the function below is well defined:
    \[
        \functiondesc
        {f_{i,\nu}}
        {(\mathbb{R},-\infty)}
        {(\mathbb{C}^*,0)}
        {t}
        {X_{(i)}(\nu q^t)}
    \]
    Then, the function $f_{i,\nu}$ has polynomial growth at $t=-\infty$.
\end{prop}

\subsection{Confluence of a regular singular $q$-difference equation}\label{qde:subsection_confluence_reg_sing}

In this last subsection we explain the confluence phenomenon for the class of regular singular $q$-difference systems.
The main idea is the following: take a regular singular $q$-difference system and its fundamental solution $\mathcal{X}_q$ given by Theorem \ref{qde:thm_qde_frob}.
If the $q$-difference system admits a formal limit when $q$ tends to $1$ that is a regular singular differential system, then the limit $\lim_{q \to 1} \mathcal{X}_q$ coincides with the fundamental solution of the differential system given by Theorem \ref{qde:thm_fund_sol_for_differential}.
The precise statement will be found in Definition \ref{qde:def_confluentsys} and Corollary \ref{qde:coro_confluence_for_fund_sols}.

\begin{remark}
    We recall that in Subsection \ref{qde:subsection_ex_for_confluence}, we got a similar result in Proposition \ref{qde:prop_confluence_for_the_example}.
    We had (the $q$-pullback of) a $q$-difference equation whose solutions were spanned by $g_q(Q) = ((1-q)Q;q)_\infty^{-1}$.
    Furthermore, the function $g_q$ was pointwise convergent when $q \to 1$ to a solution of the formal limit of the $q$-difference equation for $Q \leq 1$.
    
    In general, these limits will be harder to obtain.
    One issue that arises is quickly is the behaviour of the poles of our special function $e_{q, \lambda_q}, \ell_q$ when we consider $q$ as a variable.
    One way to control the poles of these functions is to fix $q_0 \in \mathbb{C}^*, |q_0|<1$, let $q(t)=q_0^t, t \in (0,1]$ and compute limits when $t \to 0$.
    By doing so, the discrete $q$-spirals $q(t)^\mathbb{Z}$ are for all $t>0$ subsets of the continuous $q$-spiral $q_0^\mathbb{R}$.
    Notice also that the set $\mathbb{C}^* - q(0)^\mathbb{R}$ is simply connected, so we can define a logarithm on it.
\end{remark}

\begin{prop}[\cite{Sauloy_qde_regsing}, Subsections 3.1.3 and 3.1.4]\label{qde:prop_specfns_limits}
    Let $q_0 \in \mathbb{C}^*, |q_0|<1$, let $q(t)=q_0^t, t \in (0,1]$, and consider $\lambda_{q(t)}, \mu \in \mathbb{C}^*$ such that $\frac{\lambda_{q(t)}-1}{q-1} \to \mu$. Then we have the asymptotics:
    \begin{enumerate}[label=(\roman*)]
        \item Denote by $\log$ the logarithm on $\mathbb{C}^*- (-1)q_0^\mathbb{R}$, such that $\log(1) = 0$.
        We have the uniform convergence on any compact of $\mathbb{C}^*- (-1)q_0^\mathbb{R}$
        \[
            \lim_{t \to 0} (q(t)-1) \ell_{q(t)}(-Q) = \log(Q)
        \]
        
        \item We have the uniform convergence on any compact of $\mathbb{C}^*- \lambda_{q_0} q_0^\mathbb{R}$
        \[
            \lim_{t \to 0} e_{q(t),\lambda_{q(t)} }(-Q) = Q^\mu
        \]
    \end{enumerate}
\end{prop}

\subsubsection{Confluence of the $q$-difference equation}

\begin{defin}
    A regular singular $q$-difference system $q^{Q \partial_Q} X = A_q(Q) X$ is said to be \textit{non resonant} if two different eigenvalues $\lambda_i \neq \lambda_j$ of the matrix $A(0)$ satisfy $\lambda_i \lambda_j^{-1} \notin q^\mathbb{Z}$
\end{defin}

\begin{defin}[\cite{Sauloy_qde_regsing} Section 3.2]\label{qde:def_confluentsys}
    Let $q_0 \in \mathbb{C}, |q_0|<1$, and set $q(t)=q_0^t$, for $t \in (0,1]$.
    A regular singular, non $q$-resonant $q$-difference system $q^{Q \partial_Q} X = A_q(Q) X$ is said to be \textit{confluent} if it satisfies the four conditions below.
    Set $B_q(Q) = \frac{A_q(Q) - \text{Id}}{q-1}$, whose coefficients have poles $Q_1(q), \dots, Q_k(q)$ in the input $Q$. We require that
    \begin{enumerate}
        \item The $q$-spirals satisfy $\bigcap_{i=1}^k Q_i(q_0) q_0^\mathbb{R} = \varnothing$.

        \item There exists a matrix $\Tilde{B} \in \text{GL}_n \left(\mathbb{C}(Q)\right)$ such that
        \[
            \lim_{t \to 0} B_{q(t)} = \Tilde{B}
        \]
        uniformly in $Q$ on any compact of $\mathbb{C}^* - \bigcup_{i=0}^k Q_i q_0^\mathbb{R}$, where $Q_0=1$.
        
        \item This limit defines a regular singular, non resonant differential system
        \begin{equation*}
            Q \partial_Q \widetilde{X} = \widetilde{B} X 
        \end{equation*}
        with distinct singularities $\widetilde{Q}_i = \lim_{t \to 0} Q_i(q)$.
        
        \item Take all Jordan decompositions $B_{q(t)}(0)={P_{q(t)}}^{-1} J_{q(t)} P_{q(t)}$ as well as $\widetilde{B}(0)=\widetilde{P}^{-1} \widetilde{J} \widetilde{P}$.
        We ask that \[ \lim_{t \to 0} P_{q(t)} = \widetilde{P} \]
    \end{enumerate}
\end{defin}

Let us discuss these hypotheses.
The condition 1. says that two different poles $Q_i, Q_j$ lie on different continuous $q$-spirals.
This implies that the limit differential system will have simple singularities given by $\lim_{t \to 0} Q_i(q)$.

\begin{example}[About condition 2.]\label{qde:rmk_confluence_sys}
    Given a regular singular $q$-difference system $q^{Q \partial_Q} X = A_q(Q) X$, such system needs not to satisfy the existence of $\lim_{t \to 0} \frac{A_q - Id}{q-1}$.
    For example, take the $q$-difference equation $\qdeop{Q} f_q(Q) = (1-Q) f_q(Q)$ of Subsection \ref{qde:subsection_ex_for_confluence}.
    We can easily see that this $q$-difference equation is regular singular at $Q=0$ while the matrix $B_{q(t)} = \frac{-Q}{q(t)-1}$ does not admit a limit when $t \to 0$, so the $q$-difference equation $\qdeop{Q} f_q(Q) = (1-Q) f_q(Q)$ fails condition 2.
\end{example}


\begin{example}[About condition 3.]
    The $q$-difference equation
    \[
        \qdeop{Q} f_q(Q) = \left( 1 + \frac{q-1}{Q + (q-1)} \right) f_q(Q)
    \]
    is singular regular at $Q=0$ and admits a formal limit which is
    \[
        Q\partial_Q \widetilde{f}(Q) = \frac{1}{Q} f(Q)
    \]
    whose solutions are spanned by $e^{-\frac{1}{x}}$, which has exponential growth at $x=0$.
    This system is therefore singular irregular.
\end{example}

For a counter-example where condition 4. fails, we refer to \cite{Sauloy_qde_regsing}, Subsection 3.3.2, Subsubsection "Remarque préliminaire".

\subsubsection{Confluence of the solutions}

\begin{thm}[\cite{Sauloy_qde_regsing}, Section 3.3]\label{qde:thm_confluence_of_sol_with_initial_cond}
    Let $q_0 \in \mathbb{C}, |q_0|<1$, and set $q(t)=q_0^t$, for $t\in[0,1]$.
    Consider a regular singular confluent $q$-difference system $q^{Q \partial_Q} X_q(Q) = A_q(Q) X_q(Q)$, whose limit system is $Q \partial_Q X(Q) = \widetilde{B}(Q) X(Q)$.
    
    We assume that there exists a vector $X_0 \in \mathbb{C}^n - 0$, independent of $q$, such that $A_{q(t)} X_0 = X_0$ for all $t \in (0,1]$.
    We also assume that we have a solution $X_q$ of the $q$-difference system satisfying the initial condition $X_q(0)=X_0$.
    
    Let $\widetilde{X}$ be the unique solution of $Q \partial_Q X(Q) = \widetilde{B}(Q) X(Q)$ satisfying the initial condition $\widetilde{X}(0)=X_0$.
    We have
    \[
        \lim_{t \to 0} X_{q(t)}(Q) = \widetilde{X}(Q)
    \]
    uniformly in $Q$ on any compact of $\mathbb{C}^* - \bigcup_{i=0}^k Q_i q_0^\mathbb{R}$.
\end{thm}

Applying this theorem to the fundamental solutions given by Theorems \ref{qde:thm_qde_frob} and \ref{qde:thm_fund_sol_for_differential} gives the corollary below.

\begin{coro}[\cite{Sauloy_qde_regsing}, Subsections 3.2.4 and 3.4]\label{qde:coro_confluence_for_fund_sols}
Let $q^{Q \partial_Q} X_q(Q) = A_q(Q) X_q(Q)$ be a confluent regular singular $q$-difference system.
Denote by $\mathcal{X}_q$ the fundamental solution of the $q$-difference given by Theorem \ref{qde:thm_qde_frob}.
The limit differential system $Q \partial_Q \widetilde{X} = \widetilde{B} X$ also has a fundamental solution constructed by Theorem \ref{qde:thm_fund_sol_for_differential}, which we denote by $\widetilde{\mathcal{X}}$.
Then, we have
    \[
        \lim_{t \to 0} \mathcal{X}_{q(t)} = \widetilde{\mathcal{X}}
    \]
\end{coro}

\begin{remark}
    This corollary applies only to the fundamental solutions of Theorems \ref{qde:thm_qde_frob} and \ref{qde:thm_fund_sol_for_differential}.
    For example, consider the $q$-difference equation $\qdeop{Q} f_q(Q) = f_q(Q)$, which has formal limit $\partial_Q \widetilde{f} = 0$.
    The function $\mathcal{X}_q(Q) = \frac{1}{q-1}$ is a solution of this $q$-difference equation.
    However, it does not admit a limit when $t \to 0$.
    This function failed the requirements on the initial condition of Theorem \ref{qde:thm_confluence_of_sol_with_initial_cond}.
\end{remark}

\begin{defin}\label{qde:def_confluent_solution}
    Let $\qdeop{Q} X = A_q X$ be a confluent $q$-difference system. A fundamental solution $\mathcal{X}_q$ is \textit{confluent} if $\lim_{q \to 1} \mathcal{X}_q$ is a fundamental solution of the formal limit of $\qdeop{Q} X = A_q X$.
\end{defin}

\subsubsection{When confluence fails}

We close this subsection on what can be done if confluences fails because a limit is not well defined.

\begin{remark}\label{qde:remark_when_system_is_not_confluent}
    Recall that in Subsection \ref{qde:subsection_ex_for_confluence}, the example $\qdeop{Q} f_q(Q) = (1-Q) f_q(Q)$ was not confluent (see Remark \ref{qde:rmk_confluence_sys}).
    Then, we studied the confluence of another $q$-difference equation, $\qdeop{Q} f_q(Q) = (1-(1-q)Q) f_q(Q)$.
    The second equation is the $q$-pullback of the first one by the isomorphism
    \[
    \left| \,
        \begin{tikzcd}[
                cramped,row sep = 0cm, every text node part/.style={align=center}, node distance=0.1cm
            ]
            \mathbb{C}
                \arrow[r]
            &
            \mathbb{C}
            \\
            Q
                \arrow[r, mapsto]
            &
            (1-q)^{-1}Q
        \end{tikzcd}
    \right.
    \]
    
    Notice that this is the only $q$-pullback by an isomorphism of the form $\left( Q \mapsto (1-q)^k Q \right)$ for $k \in \mathbb{Z}$ such that, after $q$-pullback, the system has a limit which is both defined and non trivial. 
\end{remark}
When working on concrete example, a $q$-pullback of the same form $\left( Q \mapsto (1-q)^k Q \right)$ is always used when a $q$-difference system fails the condition 2. for confluence.
However, we could not find a statement which guarantees the existence of a "good" $q$-pullback.

\begin{remark}\label{qde:remark_when_solution_is_not_confluent}
    The $q$-difference equation $\qdeop{Q} f_q(Q) = f_q(Q)$ is a regular singular confluent $q$-difference system of formal limit $\partial_Q \widetilde{f}(Q)=0$.
    The function $f_q(Q) = \frac{1}{q-1}$ is a solution of this $q$-difference equation, however it is not confluent.
    However, let us consider the change of fundamental solution replacing $f_q$ by $(1-q)f_q$.
    The rescaled solution $(1-q) f_q = 1$ is now confluent.
    
    Notice that this is the only transformation of the form $\left( f_q \mapsto (1-q)^k f_q \right)$ for $k \in \mathbb{Z}$ such that the limit when $q \to 1$ of the rescaled solution is both defined and non trivial.
    
    A similar situation can also be found in Proposition \ref{qde:prop_specfns_limits} where the special function $\ell_q$ was rescaled to $(q-1) \ell_q$ to obtain a non trivial limit.
\end{remark}

Let us give an informal summary of this section.
\begin{summary}
    Let $\qdeop{Q}X = AX$ be a regular singular $q$-difference system (Definition \ref{qde:def_regular_singular_qde}).
    \begin{itemize}
        \item This $q$-difference system is confluent if, up to some technicalities, its formal limit when $q \to 1$ defines a regular singular differential system (cf. Definition \ref{qde:def_confluentsys}).
        If that limit is not defined, we may try to find a $q$-pullback after which it exists (cf. Remark \ref{qde:remark_when_system_is_not_confluent}).
        In common situations, the $q$-pullback that is used is of the form $\left( Q \mapsto (1-q)^k Q \right)$, with $k \in \mathbb{Z}$.
        \item If this $q$-difference system is confluent, it has a fundamental solution $\mathcal{X}_q$ given by Theorem \ref{qde:thm_qde_frob}.
        Also, the limit when $q \to 1$ of this system defines a singular regular differential system, which has a fundamental solution $\widetilde{\mathcal{X}}$ constructed by Theorem \ref{qde:thm_fund_sol_for_differential}.
        Then (cf. Corollary \ref{qde:coro_confluence_for_fund_sols}),
        \[
            \lim_{q \to 1} \mathcal{X}_q = \widetilde{X}
        \]
        \item A non trivial solution to a confluent $q$-difference system is confluent if it admits a non trivial limit when $q \to 1$ (cf. Definition \ref{qde:def_confluent_solution}).
        If that limit is not defined, we may try to find a change of fundamental solution after which it exists (cf. Remark \ref{qde:remark_when_solution_is_not_confluent}).
        In common situations, this change of fundamental solution changes the non confluent fundamental solution $\mathcal{X}_q$ into the confluent fundamental solution $\mathcal{X}_q B_q$, where $B_q \in \textnormal{GL}_n(\mathcal{M} \left( \mathbb{E}_q \right))$ is a $q$-constant matrix.
    \end{itemize}
\end{summary}


%
%
%

\section{Monodromy of regular singular \lowercase{$q$}-difference equations}\label{qde:section_monodromy}


In this subsection we discuss the $q$-analogue of monodromy for regular singular $q$-difference systems.
Since the contents of this subsection does not apply to the $q$-difference of Givental's small $K$-theoretic $J$-function of projective planes (see Example \ref{qde:ex_qde_jfunction_is_not_fuchsian}), the reader interested in the comparison theorem can skip this section.
However, this subsection will be useful for the computation of the $q$-stokes matrices.

\subsection{Birkhoff's connection matrix}

\begin{defin}
    A $q$-difference system $q^{Q \partial_Q} X_q(Q) = A_q(Q) X_q(Q)$ is \textit{fuchsian} if it is regular singular at $Q=0$ and $Q=\infty$.
\end{defin}

\begin{example}\label{qde:ex_qde_jfunction_is_not_fuchsian}
    The $q$-difference equation $(1-\qdeop{Q})^2 f_q(Q) = Q f_q(Q)$ is regular singular at $Q=0$ but is not fuchsian.
    Define $W=Q^{-1}$. This $q$-difference equation becomes
    \[
        \left(
            1-\left(\qdeop{W}\right)^{-1}
        \right)^2
        g_q(W)
        =
        \frac{1}{W} g_q(W)
    \]
    Which can be rewritten as $q^2W \left( \qdeop{W} - 1 \right)^2 - (\qdeop{W})^2 g_q(W) = 0$. This $q$-difference equation is not regular singular at $W=0$ by Proposition \ref{qde:prop_when_a_qde_equation_is_regular_singular}.
\end{example}

\begin{defin}
    Let $q^{Q \partial_Q} X_q(Q) = A_q(Q) X_q(Q)$ be a fuchsian $q$-difference system.
    This $q$-difference systems admits a fundamental solution $\mathcal{X}_0(Q)$ at $Q=0$ and a second one $\mathcal{X}_\infty(1/Q)$ at $Q=\infty$. 
    \textit{Birkhoff's connection matrix} (or $q$\textit{-monodromy}) $P_q$ is the ratio 
    $P_q(Q) = \mathcal{X}_0(Q) \left( \mathcal{X}_\infty(1/Q) \right)^{-1} \in \textnormal{GL}_n \left( \mathcal{M} \left( \mathbb{E}_q \right) \right)$.
\end{defin}

Since Birkhoff's connection matrix is the ratio of two fundamental solutions, it is invariant by gauge transform.
The data of Birkhoff's connection matrix classifies fuchsian $q$-difference systems up to gauge transform. For a precise statement, see \cite{Har_Sau_Sin_book}, Theorem 3.4.9 p.134.

\begin{thm}[\cite{Sauloy_qde_regsing}, Section 4.3]
    Let $q^{Q \partial_Q} X_q(Q) = A_q(Q) X_q(Q)$ denote a fuchsian confluent $q$-difference system such that the matrix $B_q(Q) = \frac{A_q(Q) - Id}{q-1}$ has poles $Q_1(q), \dots, Q_k(q)$.
    Let $P_q \in \textnormal{GL}_n \left( \mathcal{M} \left( \mathbb{E}_q \right) \right)$ be the Birkhoff's connection matrix of this system.
    \begin{enumerate}[label=(\roman*)]
        \item Let $Q_0=1$. The limit $\widetilde{P}(Q)=\lim_{t \to 0} P_{q(t)}(Q)$ is well defined on any compact of $\mathbb{C}^* - \bigcup_{i=0}^k q^\mathbb{R}$, and is locally constant.
        \item Define $\widetilde{Q_i} := \lim_{t \to 0} Q_i(q(t))$ and let $\widetilde{P_i}$ be the value of $\widetilde{P}(Q)$ on the connected component whose boundary is given by the $q$-spirals $Q_i(q_0) q_0^\mathbb{R}$ and $Q_{i+1}(q_0) q_0^\mathbb{R}$.
        Then the monodromy $M_i$ of the limit differential system at $\widetilde{Q_i}$ is given by the ratio $\widetilde{P_i}^{-1} \widetilde{P}_{i-1}$.
    \end{enumerate}
\end{thm}

\subsection{Computing $q$-monodromy and its confluence on an example}\label{qde:subsection_ex_monodromy}

In this subsection we give the computations for a given example.
This subsection has no purpose in quantum $K$-theory and can be skipped.
However, we mention that this subsection contains two propositions that can be helpful to compute the confluence of solutions, which are Propositions \ref{qde:prop_asymptotic_qpoch} and \ref{qde:prop_asymptotic_thetadeformed}.

Consider the rank 1 $q$-difference equation
\begin{equation}\label{qde:eqn_ex_monodromy}
    \qdeop{Q} f_q(Q) = \left( 1 + \frac{(q-1)Q}{(Q-1)(Q-i)(Q+1)}\right) f_q(Q)
\end{equation}
We chose this $q$-difference equation because it admits the formal limit
\[
    \partial_Q \widetilde{f}(Q) = \frac{1}{(Q-1)(Q-i)(Q+1)} \widetilde{f}(Q)
\]

\textbf{Solution at $Q=0$.}
We check that the coefficient in front of $f_q$ in the right hand side is not an integer power of $q$ and well defined when $Q=0$, so the $q$-difference equation is regular singular at $Q=0$.

\begin{prop}\label{qde:prop_sol_rank_1}
    Let $\lambda_q \in \mathbb{C}^*, N>0$, and $\left(\alpha_i\right), \left(\beta_i\right) \in \left(\mathbb{C}^*\right)^N$. Consider the $q$-difference equation
    \[
    \qdeop{Q} f_q(Q) = \lambda_q \frac{\prod_i (1-\alpha_i Q)}{\prod_i (1-\beta_i Q)} f_q(Q)
    \]
    A solution to such a $q$-difference equation is given by
    \[
    f_q(Q) = e_{q, \lambda_q}(Q) \frac{\prod_i (\beta_i Q; q)_\infty}{\prod_i (\alpha_i Q;q)_\infty}
    \]
\end{prop}

\begin{proof}
    Because of the compatibility of the $q$-difference operator $\qdeop{Q}$ with multiplication,  this is a consequence of
    \[
        \qdeop{Q} \frac{1}{(\alpha Q;q)_\infty} = \frac{1 - \alpha Q}{(\alpha Q;q)_\infty}
    \]
\end{proof}

\begin{coro}\label{qde:coro_ex_confluence_sol_0}
    Let $\alpha_1(q), \alpha_2(q), \alpha_3(q) \in \mathbb{C}^*$ be the inverses of the roots of the polynomial $P(Q) = (q-1)Q + (Q-1)(Q-i)(Q+1)$.
    A solution of the $q$-difference (\ref{qde:eqn_ex_monodromy}) is given by
    \[
        f_q(Q) = \frac{(Q, q)_\infty (-iQ, q)_\infty (-Q, q)_\infty}{(\alpha_1(q) Q, q)_\infty (\alpha_2(q) Q, q)_\infty (\alpha_3(q) Q, q)_\infty}
    \]
\end{coro}

While the expression of the roots of the polynomial $P(Q) = (q-1)Q + (Q-1)(Q-i)(Q+1)$ might be hard to obtain (especially if $P$ were to be degree $>4$), we are only interested in computing their Taylor polynomial in $q-1$ of degree 1.

\textbf{Solution at $Q=\infty$ and connection formula.}
We set $W=\frac{1}{Q}$. Notice that $\qdeop{Q} = \left(\qdeop{W}\right)^{-1}$. Setting $g_q(W)=f_q\left(W^{-1}\right)$, the $q$-difference equation (\ref{qde:eqn_ex_monodromy}) becomes
\begin{equation}\label{qde:eqn_ex_monodromy_infty}
    \qdeop{W} g_q(W)
    =
    \frac{
        (1-qW)(1-iqW)(1-(-1)qW
    }{
        (1-\alpha_1(q)qW)(1-\alpha_2(q)qW)(1-\alpha_3(q)qW)
    }
    g_q(W)
\end{equation}

\begin{prop}
    A solution at $W=0$ of the $q$-difference equation (\ref{qde:eqn_ex_monodromy_infty}) is given by the function
    \[
        g_q(W) = \frac{
            (q \alpha_1(q)^{-1} W, q)_\infty (q \alpha_2(q)^{-1} W, q)_\infty (q \alpha_3(q)^{-1} W, q)_\infty
        }{
            (qW, q)_\infty (iqW, q)_\infty (-qW, q)_\infty
        }
    \]
\end{prop}

\begin{proof}
    Same argument as in the Proposition \ref{qde:coro_ex_confluence_sol_0}
\end{proof}

\begin{coro}
    Birkhoff's connection matrix of the $q$-difference equation (\ref{qde:eqn_ex_monodromy}) is given by
    \[
        P_q(Q) = \frac{
            \theta_q(-Q) \theta_q(iQ) \theta_q(Q)
        }{
            \theta_q \left(-\alpha_1(q)Q\right) \theta_q \left((-\alpha_2(q)Q\right) \theta_q \left((-\alpha_3(q)Q \right)
        }
    \]
\end{coro}

\begin{proof}
    Consequence of Jacobi's triple product identity, see Proposition \ref{qde:prop_theta_jacobi}.
\end{proof}

\textbf{Confluence.}
We had seen that the $q$-difference equation (\ref{qde:eqn_ex_monodromy}) has the formal limit
\[
    \partial_Q \widetilde{f}(Q) = \frac{1}{(Q-1)(Q-i)(Q+1)} \widetilde{f}(Q)
\]
The solutions to this differential equation are spanned by the function
\begin{equation}\label{qde:eqn_ex_monodromy_confluencesol}
    \widetilde{f}(Q) =
    (Q-1)^{\frac{1}{4}(1+i)}
    (Q-i)^{-\frac{1}{2}}
    (Q+1)^{\frac{1}{4}(1-i)}
\end{equation}

For confluence of solutions, we will need the Taylor polynomial in $q-1$ of degree 1 for every root of $P(Q) = (q-1)Q + (Q-1)(Q-i)(Q+1)$. We have three roots $Q_1,Q_i,Q_{-1}$ which satisfy
\begin{align*}
    Q_1 &= 1 + \frac{-1}{2(1-i)} (q-1) + o(q-1)     \\
    Q_i &= i + \frac{i}{2} (q-1) + o(q-1)           \\
    Q_{-1} &= -1 + \frac{1}{2(1+i)} (q-1) + o(q-1)
\end{align*}
Therefore,
\begin{align*}
    \alpha_1 &= 1 + \frac{1+i}{4}(q-1) + o(q-1) = 1 \left(1 + \frac{1+i}{4}(q-1) + o(q-1) \right)   \\
    \alpha_2 &= -i + \frac{i}{2}(q-1) + o(q-1) = -i \left(1 - \frac{1}{2}(q-1) + o(q-1) \right)     \\
    \alpha_3 &= -1 - \frac{1-i}{4}(q-1) + o(q-1) =-1\left(1 + \frac{1-i}{4}(q-1) + o(q-1) \right)
\end{align*}

\begin{prop}[\cite{Sauloy_qde_regsing}]\label{qde:prop_asymptotic_qpoch}
    Let $\Omega = \mathbb{C}-q_0^{\mathbb{R}}, Q_0 \in \Omega$. Let $Q_1(q), Q_2(q) \in \Omega$. Assume there exists complex numbers $\alpha_0, \alpha_1 \in \mathbb{C}$ so that $Q_i(q(t)) = Q_0 q_0^{\alpha_i t + o(t)}$.
    
    Then, on $\Omega$, we have the uniform convergence when $t \to 0$
    \[
        \lim_{t \to 0}
        \frac{(Q_1(q(t);q(t))_\infty}{(Q_2(q(t));q(t))_\infty}
        =
        (1-Q_0)^{\alpha_2 - \alpha_1}
    \]
    
\end{prop}

\begin{proof}
    See \cite{Sauloy_qde_regsing}, Subsection 3.1.7, Corollaire 3.
\end{proof}

\begin{coro}
    Let $f_q$ be the solution of the $q$-difference equation (\ref{qde:eqn_ex_monodromy}) given by Proposition \ref{qde:coro_ex_confluence_sol_0}, and let $\widetilde{f}$ be the solution (\ref{qde:eqn_ex_monodromy_confluencesol}) of the limit differential equation.
    Set $\Omega = \mathbb{C} - (q^\mathbb{R} \cup iq^\mathbb{R} \cup -q^\mathbb{R})$. Then, on any compact of $\Omega$,
    \[
    \lim_{t \to 0} f_{q^t} = \widetilde{f}
    \]
    \qed
\end{coro}

\begin{prop}[\cite{Sauloy_qde_regsing}]\label{qde:prop_asymptotic_thetadeformed}
    Let $\Omega = \mathbb{C}-q_0^{\mathbb{R}}, Q_0 \in \Omega$. Let $Q_1(q), Q_2(q) \in \Omega$. Assume there exists complex numbers $\alpha_0, \alpha_1 \in \mathbb{C}$ so that $Q_i(q(t)) = Q_0 q_0^{\alpha_i t + o(t)}$.
    
    Then, on $\Omega$, we have the uniform convergence when $t \to 0$
    \[
    \frac{\theta_{q(t)}(Q_1(q(t))}{\theta_{q(t)}(Q_2(q(t))}
    =
    Q_0^{\alpha_2-\alpha_1}
    \]
\end{prop}

\begin{proof}
    See \cite{Sauloy_qde_regsing}, Subsection 3.1.7, Corollaire 1.
\end{proof}

\begin{coro}
    Let $P_q$ be Birkhoff's connection "matrix" of the $q$-difference equation (\ref{qde:eqn_ex_monodromy}). Set $\Omega = \mathbb{C} - (q^\mathbb{R} \cup iq^\mathbb{R} \cup -q^\mathbb{R})$. Then, on any compact of $\Omega$, we have the uniform convergence
    \[
    \lim_{t \to 0} P_{q^t}(Q)
    =
    (-Q)^\frac{1+i}{4}
    (-iQ)^{- \frac{1}{2}}
    Q^\frac{1-i}{4}
    \]
    \qed
\end{coro}
We recall that that the logarithm used to define the characters $Q \mapsto Q^\alpha$ is the principal determination of the logarithm on the simply connected subset $\mathbb{C}^* - (-1)q^\mathbb{R}$.

We give below a picture of the three connected components of $\Omega$.
If we denote by $\widetilde{P_i}^\textnormal{up}, \widetilde{P_i}^\textnormal{down}$ the two connected components bordering the singularity $i$, 
we remark that the computation $\left( \widetilde{P_i}^\textnormal{down} \right)^{-1} \widetilde{P_i}^\textnormal{up}$ should amount to the action of a loop around the singularity $i$ on the solution of the differential equation.

\begin{center} 
    \tikzmath{
          integer \k, \niterA, \niterB, \niterC;
          real \ini, \fin, \niter;
          \ini = -12;
          \fin = 4.4;
          \niter = 200;
          \niterA = 200;
          \niterB = 185;
          \niterC = 170;
          function angle(\k){
            \t = \ini+\k*(\fin-\ini)/\niter;
            return 45*\t;
          };
          function module(\k){
            \t = \ini+\k*(\fin-\ini)/\niter;
            return sqrt(2)^\t;
          };
        }
        
        \begin{tikzpicture}[scale=2.5]
          \tikzset{%
           evaluate={%
              for \k in {1,2,...,\niter}{%
                \a{\k} = angle(\k);
                \m{\k} = module(\k);
              };
            }
          }
        
          \clip (-3,-1.5) rectangle ++(6,4);
          \fill[OY, opacity=.8] (-3,-1.5) rectangle ++(6,4);
        
          \draw[thick, fill=RO, fill opacity=.7] (0,0)
          \foreach \k in {1,2,...,\niterA}{%
            -- +(\a{\k}:\m{\k})
          }
          \foreach \k in {\niterB,...,1}{%
            -- +(\a{\k}+90:\m{\k})
          } -- cycle;
        
          \draw[thick, fill=DarkR, fill opacity=.7] (0,0)
          \foreach \k in {1,2,...,\niterB}{%
            -- +(\a{\k}+90:\m{\k})
          }
          \foreach \k in {\niterC,...,1}{%
            -- +(\a{\k}+180:\m{\k})
          } -- cycle;
        
          \draw[GB] (0,0) circle (1);
          
          \draw[fill=white] (1,0) circle (1pt) node[right=2pt] {$1$};
          \draw[fill=white] (-1,0) circle (1pt) node[left=2pt] {$-1$};
          \draw[fill=white] (0,1) circle (1pt) node[above=2pt] {$i$};
    
    \end{tikzpicture}
    
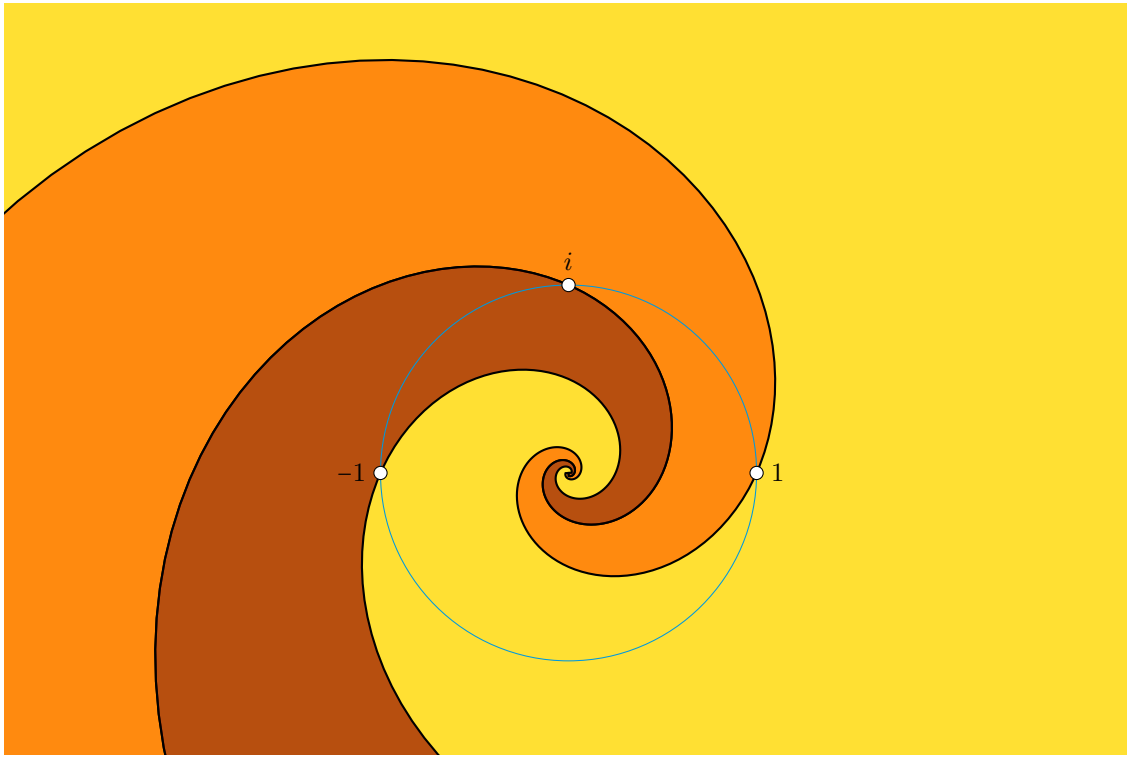
    \textit{The three connected components of $\Omega$ with the singularities of the differential system}
\end{center}

%
%
%

\chapter{Confluence of the \lowercase{$q$}-difference structure for $SQK(\mathbb{P}^N)$}\label{chapter:QKqDE}

In quantum cohomology, the $I$-function appears as a solution to a (differential) hypergeometric system.
In quantum $K$-theory, we see instead the apparition of $q$-hypergeometric series (see e.g. \cite{Giv_PermEquiv}, V).
Therefore, it may be interesting to study the $q$-difference module structures in quantum $K$-theory.
It turns out that the $q$-difference module in quantum $K$-theory can be related to the differential module in quantum cohomology using the confluence of $q$-difference equations. The goal of this chapter is explain this relation.

%
%
%

\section{Definitions revisited from the point of view of $q$-difference equations}

In this section we revisit the definitions related to Givental's $K$-theoretical small $J$-function. We focus on the $q$-difference structure given by this function, and try to apply the theory of the previous chapter.

\subsection{Solving the $q$-difference equation for the $J$-function}

The goal of this subsection is to try to find the solutions of the $q$-difference equation of the $J$-function.
The main result of this subsection is Proposition \ref{qkqde:prop_JK_is_a_fund_sol}, which states that Givental's $J$-function is a fundamental solution of its $q$-difference equation.

We begin by recalling the definition of the two Givental's small $K$-theoretical $J$-function, then we establish their $q$-difference equations.

\begin{defin}[\cite{Giv_Lee_qk}]
    Let $X = \mathbb{P}^N$ and let $P=\mathcal{O}(1) \in K\left( \mathbb{P}^N \right)$ be the anti-tautological bundle. \textit{Givental's small $J$-function} is the power series
    \[
    J^{K\textnormal{th}}(q,Q) = \sum_{d \geq 0} \frac{Q^d}{\left(qP^{-1};q\right)_d^{N+1}}
    \in
    K \left( \mathbb{P}^N \right) \otimes \mathbb{C}(q) [\![Q]\!]
    \]
    We recall that the $q$-Pochhammer symbol $(qP^{-1};q)_d$ is defined by $\prod_{r=1}^d \left( 1-q^r P^{-1} \right)$
\end{defin}

Recall that the $K$-theory of $\mathbb{P}^N$ is the ring (\cite{Karoubi_book_Ktheory}, Theorem 2.5 p.190)
\[
    K\left( \mathbb{P}^N \right) = \mathbb{Z}[P,P^{-1}] \left/\left( \left(1-P^{-1} \right)^{N+1} \right) \right.
\]
Thus we can look for a basis of $K\left( \mathbb{P}^N \right)$ in which we will decompose the value $J^{K\textnormal{th}}(q,Q)$.
We choose the basis $\mathds{1}, 1-P^{-1}, ,\dots (1-P^{-1})^N$.

\begin{remark}
    The basis $\left( \left( 1 -P^{-1} \right)^i \right)_i$ of $ K\left( \mathbb{P}^N \right)$ has two advantages.
    \begin{enumerate}[label=(\roman*)]
        \item The ring generator $1-P^{-1}$ is nilpotent which will be useful for definitions and computations.
        
        \item We have the ring automorphism $ \gamma : K\left( \mathbb{P}^N \right) \otimes \mathbb{Q} \to H^*\left( \mathbb{P}^N; \mathbb{Q}\right)$ defined by
        $\gamma(1-P^{-1})=H$, where $H \in H^2\left( \mathbb{P}^N; \mathbb{Q}\right)$ is the hyperplane class.
        \end{enumerate}
\end{remark}

\begin{example}\label{qkqde:ex_computation_qpoch}
    Let us give the decomposition of the value $J^{K\textnormal{th}}(q,Q)$ when $X=\mathbb{P}^2$.
     We have
     \begin{align*}
        (qP^{-1};q)_d
        &:=
        \prod_{r=1}^d \left( 1-q^r P^{-1} \right)
        =
        \prod_{r=1}^d \left( (1-q^r) + q^r (1-P^{-1}) \right) \\
        &=\prod_{r=1}^d (1-q^r)
        +
        (1-P^{-1})\prod_{r=1}^d (1-q^r) \sum_{j=1}^d \frac{q^j}{1-q^j}
        +
        (1-P^{-1})^2 \prod_{r=1}^d (1-q^r)
        \sum_{0 \leq i < j\leq d} \frac{q^{i+j}}{(1-q^i)(1-q^j)}
     \end{align*}
     Therefore, in the power series $J^{K\textnormal{th}}(q,Q)$, the coefficient attached to the to $Q^d$ is
     \[
        (qP^{-1};q)_d^{-3}
        =
        \frac{1}{(q;q)_d^3} \left(
            \mathds{1}
            +
            (1-P^{-1}) \left[
                -3 \sum_{j=1}^d \frac{q^j}{1-q^j}
            \right]
            +
            (1-P^{-1})^2 \left[
                6 \left(\sum_{j=1}^d \frac{q^j}{1-q^j}\right)^2
                -3 \sum_{0 \leq i < j\leq d} \frac{q^{i+j}}{(1-q^i)(1-q^j)}
            \right]
        \right)
     \]
     
     Notice that in the sum following $(1-P^{-1})^i$, for $i \in \{0,1,2\}$, each term is a fraction of the form below
     \[
        \frac{q^{j_1 + \cdots + j_i}}{(1-q^{j_1}) \cdots (1-q^{j_i})}
     \]
    We point out that this denominator consists of exactly $i$ products, which will be important for the confluence of the $J$-function.
\end{example}


\begin{prop}[\cite{Giv_Lee_qk}]\label{qkqde:prop_JK_sol_qde_pre}
    The small $J$-function is a solution of the $q$-difference equation with $K$-theoretical coefficients
    \[
        \left[ \left( 1 - P^{-1} \qdeop{Q} \right)^{N+1} - Q \right] J^{K\textnormal{th}}(q,Q) = 0
    \]
    Where $P^{-1}$ is the operator acting on $K\left( \mathbb{P}^N \right)$ by tensor product with the class $P^{-1} \in K\left( \mathbb{P}^N \right)$.
\end{prop}

\begin{proof}
    For $d>0$, notice that
    \[
        \left( 1 - P^{-1} \qdeop{Q} \right)^{N+1}
        \frac{Q^d}{\left(qP^{-1};q\right)_d^{N+1}}
        =
        \frac{(1-q^dP^{-1})^{N+1} Q^d}{\left(qP^{-1};q\right)_d^{N+1}}
        =
        \frac{Q^d}{\left(qP^{-1};q\right)_{d-1}^{N+1}}
    \]
    Therefore,
    \begin{align*}
        \left[ \left( 1 - P^{-1} \qdeop{Q} \right)^{N+1} - Q \right] J^{K\textnormal{th}}(q,Q) &=
        \sum_{d \geq 0}
        \frac{(1-q^dP^{-1})^{N+1} Q^d}{\left(qP^{-1};q\right)_d^{N+1}}
        -
        \sum_{d \geq 1}
        \frac{Q^d}{\left(qP^{-1};q\right)_{d-1}^{N+1}}\\
        &=\left(1-P^{-1}\right)^{N+1}
    \end{align*}
    We conclude using the relation $\left(1-P^{-1}\right)^{N+1}=0$ in $K\left( \mathbb{P}^N \right)$.
\end{proof}

Notice that the $q$-difference equation of Proposition \ref{qkqde:prop_JK_sol_qde_pre} is not a $q$-difference equation in the sense of Definition \ref{qde:def_q_difference_linear_sys}: since we make use of the operator $P^{-1}$, it is a $q$-difference equation with $K$-theoretical coefficients.
Therefore, it is convenient to introduce another $J$-function which will be a solution of the same $q$-difference equation without any $K$-theoretical operator.
We will refer to this function as the $J$-function throughout this chapter.

\begin{defin}[\cite{Iri_Mil_Ton_qk}]\label{qkqde:def_J_fn}
    Let $X = \mathbb{P}^N$ and let $P=\mathcal{O}(1) \in K\left( \mathbb{P}^N \right)$ be the anti-tautological bundle. Givental's (small, modified) $J$-function is the function
    \begin{equation}\label{qkqde:eqn_J_def}
        \widetilde{J^{K\textnormal{th}}}(q,Q) 
        :=
        P^{-\ell_q(Q)} J^{K\textnormal{th}}(q,Q)
        =
        P^{-\ell_q(Q)}
        \sum_{d \geq 0} \frac{Q^d}{\left(qP^{-1};q\right)_d^{N+1}}
    \end{equation}
    Where
    \[
        P^{-\ell_q(Q)}
        =
        \left( 1 - \left(1-P^{-1}\right) \right)^{\ell_q(Q)}
        =
        \sum_{k \geq 0} (-1)^k \binom{\ell_q(Q)}{k} \left( 1 - P^{-1} \right)^k
    \]
    And
    \[
        \binom{\ell_q(Q)}{k} = \frac{1}{k!} \prod_{r=0}^{k-1} (\ell_q(Q)-r)
    \]
\end{defin}

\begin{prop}\label{qkqde:prop_P_power_ellq_qde}
    The $K$-theoretical function $\left( Q \mapsto P^{-\ell_q(Q)} \right)$ satisfies the ordinary linear $q$-difference equation with $K$-theoretical coefficients
    \[
        \qdeop{Q} P^{-\ell_q(Q)} = P^{-1} P^{-\ell_q(Q)}
    \]
\end{prop}

\begin{proof}
    The $q$-logarithm $\ell_q$ satisfies by Proposition \ref{qde:prop_ell_q_is_a_q_log} the $q$-difference equation $\qdeop{Q} \ell_q(Q) = \ell_q(Q) +1$.
    Therefore, we have
    \[
        \qdeop{Q} P^{-\ell_q(Q)} =P^{-\ell_q(Q) - 1} = P^{-1}P^{-\ell_q(Q)}
    \]
\end{proof}
\begin{prop}[\cite{Iri_Mil_Ton_qk}]
    The $J$-function $\widetilde{J^{K\textnormal{th}}}(q,Q)$ is a solution of the $q$-difference equation
    \begin{equation}\label{qkqde:eqn_JK}
        \left[ \left( 1 - \qdeop{Q} \right)^{N+1} - Q \right]  \widetilde{J^{K\textnormal{th}}}(q,Q) = 0
    \end{equation}
    This $q$-difference equation is regular singular at $Q=0$ and irregular singular at $Q=\infty$ and of rank $N+1$.
\end{prop}

\begin{proof}
    Using Proposition \ref{qkqde:prop_P_power_ellq_qde}, we get
    \[
        \qdeop{Q} \widetilde{J^{K\textnormal{th}}}(q,Q) 
        =
        \left(
            \qdeop{Q} P^{-\ell_q(Q)}
        \right)
        \left(
            \qdeop{Q} J^{K\textnormal{th}}(q,Q)
        \right)
        =
        P^{-1} P^{-\ell_q(Q)} \left( \qdeop{Q} J^{K\textnormal{th}}(q,Q) \right)
    \]
    Therefore, using Proposition \ref{qkqde:prop_JK_sol_qde_pre}, we have
    \[
        \left[
            \left( 1 - \qdeop{Q} \right)^{N+1} - Q
        \right]
        \widetilde{J^{K\textnormal{th}}}(q,Q)
        =
        P^{-\ell_q(Q)}
        \left[
            \left( 1 - P^{-1}\qdeop{Q} \right)^{N+1} - Q
        \right]
        J^{K\textnormal{th}}(q,Q)
        =0
    \]
    Applying Proposition \ref{qde:prop_when_a_qde_equation_is_regular_singular} to this $q$-difference equations gives that it is regular singular at $Q=0$.
    Now, to study the singularity at $Q=\infty$, let us set $W=Q^{-1}$.
    Then, the $q$-difference equation (\ref{qkqde:eqn_JK}) becomes
    \[
        \left[
        q^{N+1} W \left(\qdeop{W} - 1 \right)^{N+1} - \left( \qdeop{W} \right)^{N+1}
        \right]
        f(W)
        =
        0
    \]
    This $q$-difference equation is not regular singular at $W=0$. More precisely, its Newton polygon consists of one slope of value $\frac{1}{N+1}$.
\end{proof}

Now, we want to exhibit a fundamental solution of the $q$-difference equation (\ref{qkqde:eqn_JK_eq}).
We are going to show that the $J$-function $\widetilde{J^{K\textnormal{th}}}$ provides this fundamental solution.
We will require an intermediary result on the $q$-logarithms.

\begin{lemma}[\cite{Adams_qde}]\label{qkqde:lemma_independence_of_powers_of_ellq}
    Let $N \in \mathbb{Z}_{\geq 0}$.
    The family consisting of the functions $\ell_q(Q)^i \in \mathcal{M}\left( \mathbb{C}^* \right)$ for $i \in \{0, \dots, N\}$ is linearly independent over the field $\mathcal{M}\left( \mathbb{E}_q \right)$
\end{lemma}

\begin{proof}
%
%
    Let us assume that the function $\ell_q$ is algebraic over the field $\mathcal{M} \left( \mathbb{E}_q \right)$.
    Then, it admits a unique minimal polynomial, denoted by $P(X) = X^d+a_{d-1}X^{d-1}+\cdots+a_0 \in \mathcal{M} \left( \mathbb{E}_q \right)[X]$, which satisfies $P \left(\ell_q(Q) \right) =0$.
    Applying the $q$-difference operator $\qdeop{Q}$ to this identity, we obtain
    \[
        P\left(\ell_q(Q)+1\right) = \ell_q(Q)^d + \cdots + (a_0 + \cdots + a_{d-1}+1) = 0
    \]
    Which means we have a new unitary polynomial $Q \in \mathcal{M} \left( \mathbb{E}_q \right)[X]$ satisfying $P\left(\ell_q(Q)+1\right) = Q(\ell_q(Q))=0$ and $P \neq Q$.
    By contradiction, the function $\ell_q$ is not algebraic.
\end{proof}

\begin{remark}
    Consider the $q$-difference system
    \[
    \qdeop{Q}
    X_q(Q)
    =
    \begin{pmatrix}
        1               &   0               & 0         & \cdots    & 0         \\
        \binom{1}{0}    &   1               &  0        &           & \vdots    \\
        \binom{2}{0}    &   \binom{2}{1}    & 1         & \ddots    & \vdots    \\
        \vdots          &   \ddots          & \ddots    & \ddots    & \vdots    \\
        \binom{N}{0}    &   \binom{N}{1}    & \cdots    & \binom{N}{N-1} & 1
    \end{pmatrix}   
    X_q(Q)
    \]
    This $q$-difference system admits the fundamental solution below:
    \[
        \mathcal{X}_q(Q)
        =
        \begin{pmatrix}
            1               &   0               & \cdots    & \cdots    & 0         \\
            \ell_q(Q)       &   1               &  \ddots   &           & \vdots    \\
            \ell_q(Q)^2     &   \ell_q(Q)       & \ddots    &           & \vdots    \\
            \vdots          &   \ddots          & \ddots    & \ddots    & \vdots    \\
            \ell_q(Q)^N     &   \ell_q(Q)^{N-1} & \cdots    & \ell_q(Q) & 1
        \end{pmatrix}   
    \]
\end{remark}

\begin{prop}\label{qkqde:prop_JK_is_a_fund_sol}
    Consider the $q$-difference equation (\ref{qkqde:eqn_JK}) : $\left[ \left( 1 - \qdeop{Q} \right)^{N+1} - Q \right] f_q(Q) = 0$ and take the decomposition
    \[
        \widetilde{J^{K\textnormal{th}}}(q,Q)
        =
        \sum_{i=0}^N \widetilde{J_i}(q,Q) \left(1-P^{-1} \right)^i
        \in
        K \left( \mathbb{P}^N \right) \otimes \mathbb{C}(q) [\![Q]\!]
    \]
    Then, the functions $\widetilde{J_0}, \dots, \widetilde{J_N}$ form a fundamental solution of this $q$-difference equation.
\end{prop}

\begin{remark}
    Obtaining a fundamental solution this way is similar to the Frobenius method for the resolution of regular singular $q$-difference equations which have exponents of multiplicity strictly greater than 1.
    For a reference in the differential setting, see \cite{HLF_frobenius}.
\end{remark}

\begin{proof}
    Since the $q$-difference equation (\ref{qkqde:eqn_JK}) does not involve $K$-theoretical coefficients, each of the functions $\widetilde{J_0}, \dots, \widetilde{J_N}$ is a solution of the $q$-difference equation.
    Since the $q$-difference equation has rank $N+1$, it remains to check if our solutions are linearly independent over the field $\mathcal{M}\left( \mathbb{E}_q \right)$.
    We recall the definition
    \[
        \widetilde{J^{K\textnormal{th}}}(q,Q) 
        =
        P^{-\ell_q(Q)}
        \sum_{d \geq 0} \frac{Q^d}{\left(qP^{-1};q\right)_d^{N+1}}
    \]
    The decomposition of the function $Q \mapsto P^{-\ell_q(Q)}$ in our basis is given by
    \[
        P^{-\ell_q(Q)}
        =
        \sum_{k \geq 0} (-1)^k \binom{\ell_q(Q)}{k} \left( 1 - P^{-1} \right)^k
    \]
    Let $J_i$ be the coefficient in front of $\left( 1-P^{-1} \right)^i$ in $J^{K\textnormal{th}}(q,Q)$.
    Then, $J_i$ does not involve $q$-logarithms. More precisely, a computation gives
    \begin{equation}\label{qkqde:eqn_decomposition_JK_taylor_series}
        J_i(q,Q) = \sum_{d \geq 0} \frac{Q^d}{(q;q)_d^{N+1}}
        \left(
            \sum_{k=0}^N \sum_{\substack{
            0 \leq j_1,\dots,j_N \leq N \\  j_1 + \cdots + j_N = k \\ j_1 + 2 j_2 + \cdots + N j_N = i
            }}
            (-1)^k \frac{(N+k)!}{N! j_1! \cdots j_N!}
            \prod_{l=1}^N \left(
                \sum_{1 \leq m_1 < \cdots < m_l \leq d} \frac{q^{m_1+\cdots+m_l}}{(1-q^{m_1}) \cdots (1-q^{m_l})}
            \right)^{j_l}
        \right)
    \end{equation}
    Therefore, the decomposition of $\widetilde{J^{K\textnormal{th}}}(q,Q)$ is given by the functions
    \[
    \widetilde{J_i}(q,Q) = \sum_{\substack{
            a + b = i  \\  0 \leq a, b \leq N
        }}
        (-1)^a \binom{\ell_q(Q)}{a}  J_b(q,Q) 
    \]
    This expression involves integer powers of the $q$-logarithm $\ell_q$ up to the order $i$.
    Therefore, by applying Proposition \ref{qkqde:lemma_independence_of_powers_of_ellq}, we obtain that the functions $\widetilde{J_0}, \dots, \widetilde{J_N}$ are linearly independent over the field $\mathcal{M}\left( \mathbb{E}_q \right)$.
\end{proof}

For the remainder of this subsection, we try to naively solve the $q$-difference equation (\ref{qkqde:eqn_JK}) satisfied by the $J$-function by hand, in the case of $X = \mathbb{P}^2$.
This will allow us to recover the functions $\widetilde{J_i}$ given by the decomposition of the $J$-function.
We will try to find two solutions $\widetilde{J_0}, \widetilde{J_1} \in \mathcal{M}(\mathbb{C}^*,0)$ of the form
\begin{align*}
    &\widetilde{J_0}(q,Q) = \sum_{d \in \mathbb{Z}} f_d(q)Q^d
    &&\widetilde{J_1}(q,Q) = \sum_{d \in \mathbb{Z}} g_d(q)Q^d + \ell_q(Q) \sum_{d \in \mathbb{Z}} g_d(q)Q^d
\end{align*}
Using the previous Proposition \ref{qkqde:lemma_independence_of_powers_of_ellq}, we will get the that solutions $\widetilde{J_0}$ and $\widetilde{J_1}$ are linearly independent over the field $\mathcal{M}(\mathbb{E}_q)$ as long as the sequences $f,g,d$ are non zero.
To have a basis of solutions, we could also look for third solution $\widetilde{J_2}$ involving the special function $\ell_q^2$.

\begin{prop}[A first solution for $N=2$]
    The function $\widetilde{J_0}$ below is a solution of the $q$-difference equation (\ref{qkqde:eqn_JK}).
    \[
        \widetilde{J_0}(q,Q)
        =
        \sum_{d \geq 0} \frac{Q^d}{(q;q)_d^3}
    \]
\end{prop}

Before giving a proof, we remark that the solution $\widetilde{J_0}$ coincides with the coefficient in front of $\mathds{1}$ in the decomposition of $\widetilde{J^{K\textnormal{th}}_{X=\mathbb{P}^2}}$ (see Example \ref{qkqde:ex_computation_qpoch}).

\begin{proof}
    Let us explain how the Taylor series $\widetilde{J_0}$ is found.
    Let us assume that $\widetilde{J_0}(q,Q) = \sum_{d \in \mathbb{Z}} f_d(q) Q^d$ is a solution of the $q$-difference (\ref{qkqde:eqn_JK}), where $f_d(q) \in \mathbb{C}$ are some unknown coefficients to determine.
    By assuming that $\widetilde{J_0}$ is a solution of the $q$-difference equation, we have
    \[
        \left( 1 - \qdeop{Q} \right)^3 \cdot \sum_{d \in \mathbb{Z}} f_d(q) Q^d = Q \sum_{d \in \mathbb{Z}} f_d(q) Q^d
    \]
    Now, fix an integer $d \in \mathbb{Z}$. By comparing the coefficients in front of $Q^d$, we obtain the identity
    \[
        \left( 1-q^d \right)^3 f_d(q) = f_{d-1}(q)
    \]
    Which allows us to determine recursively the coefficients $f_d(q)$.
    Setting $d=0$ implies that all coefficients $f_d(q)$ for $d<0$ are zero.
    The other coefficients are given by
    \[
        f_d(q) = \frac{f_0(q)}{\prod_{r=1}^d (1-q^r)^3} = \frac{f_0(q)}{(q;q)_d^3}
    \]
    Setting $f_0(q) = 1$, we get the wanted solution.
\end{proof}

Our next proposition will explain how to construct another solution of the $q$-difference equation (\ref{qkqde:eqn_JK}).
This solution will be a priori of the form $\widetilde{J_1}(q,Q) = g_q(Q) + h_q(Q) \ell_q(Q)$, where $g_q, h_q \in \mathbb{C}( \! (q) \!)$.

\begin{prop}[Second solution for $N=2$]\label{qkqde:prop_second_solution_for_N_equal_2}
    The function $\widetilde{J_1}$ below is a solution of the $q$-difference equation (\ref{qkqde:eqn_JK})
    \[
        \widetilde{J_1}(q,Q) = \left( \sum_{d \geq 0} \frac{Q^d}{(q;q)_d^3} \right)
        \left(
            \ell_q(Q) + \sum_{k=1}^d \frac{-3q^k}{1-q^k}
        \right)
    \]
    Setting $h_0=1$ and $a_0=0$, we obtain the wanted solution.
\end{prop}

\begin{proof}
    Let us assume that we have a solution of (\ref{qkqde:eqn_JK}) of the form 
    \[
        \widetilde{J_1}(q,Q) = \sum_{d \in \mathbb{Z}} g_d(q) Q^d + \ell_q(Q) \sum_{d \in \mathbb{Z}} h_d(q) Q^d
    \]
    Where $g_d(q), h_d(q) \in \mathbb{C}$ are to be determined.
    We begin by computing $\left( \qdeop{Q} \right)^k \widetilde{J_1}(q,Q)$ for $k \in \mathbb{Z}_{\geq 0}$ before computing $\left( 1-\qdeop{Q} \right)^3 \widetilde{J_1}(q,Q)$.
    We have, for any $k \in \mathbb{Z}_{\geq 0}$,
    \[
        \left( \qdeop{Q} \right)^k \widetilde{J_1}(q,Q) = \sum_{d \geq 0} \left(g_{d} + k h_{d} \right) q^{kd} Q^d + \ell_q(Q) \sum_{d \geq 0} h_{d} q^{kd} Q^d 
    \]
    Thus, we have
    \[
        \left( 1 - \qdeop{Q} \right)^3 \widetilde{J_1}(q,Q)
        =
        \sum_{d \geq 0} \left(
            \sum_{k=0}^3 \binom{3}{k} (-q^d)^k (g_{d} + k h_{d})
        \right) Q^d
        +
        \ell_q(Q) \sum_{d \geq 0} \left(
            \sum_{k=0}^3 \binom{3}{k} (-q^d)^k h_{d}
        \right) Q^d
    \]
    Now, in the identity $\left( 1 - \qdeop{Q} \right)^3 \widetilde{J_1}(q,Q) =  Q f(q,Q)$, because of Proposition \ref{qkqde:lemma_independence_of_powers_of_ellq}, we can identify the coefficients in front of $Q^d$ for both series to get the recursion relations for $d \in \mathbb{Z}$
    \[
    \left\{
        \begin{aligned}
            &   \left(1-q^d\right)^3 h_{d} = h_{d-1} \\
            &   \left(1-q^d\right)^3 g_d + \sum_{k=0}^3 \binom{3}{k} (-q^d)^k k h_{d} = g_{d-1}
        \end{aligned}
    \right.
    \]
    The recursion relation for $h_d$ gives (see the previous proof) that for all $d \in \mathbb{Z}$,
    \[
        h_d(q) = \frac{h_0(q)}{(q;q)_d^3}
    \]
    In particular, $h_d(q)=0$ for $d<0$.
    Using the second line, we also obtain
    \[
        g_{-1}(q) = h_0(q) \sum_{k=0}^3 \binom{3}{k} (-1)^k k = 0
    \]
    so $g_d(q)=0$ for $d<0$.
    Now, let us make a change of variable and set $g_d= h_d a_d$ for some $a_d \in \mathbb{C}$.
    The second line becomes
    \[
        \left( 1-q^d\right)^3 h_d a_d + \sum_{k=0}^3 \binom{3}{k} (-q^d)^k k h_{d} = h_{d-1} a_{d-1}
    \]
    Using $\left(1-q^d\right)^3 h_{d} = h_{d-1}$ in the right hand side, we can factor by $\left(1-q^d \right)^3 h_d$ for $d>0$ and obtain the relation
    \[
        a_d + \frac{1}{\left( 1-q^d\right)^3} \sum_{k=0}^3 \binom{3}{k} (-q^d)^k k = a_{d-1}
    \]
    Developing and reorganising, we obtain for $d>0$
    \[
        a_{d} - a_{d-1}
        =
        \frac{1}{\left( 1-q^d \right)^3}
        \left(
            -3q^d + 6q^{2d} - 3q^{3d}
        \right)
        =
        \frac{-3q^d}{1-q^d}
    \]
\end{proof}

\begin{remark}
    Notice that the function $\widetilde{J_1}$ of Proposition \ref{qkqde:prop_second_solution_for_N_equal_2} corresponds to the coefficient in front of $(1-P^{-1})$ in the development of $\widetilde{J^{K\textnormal{th}}}(q,Q)$.
\end{remark}

A third and last independent solution can be found using the same method as in Proposition \ref{qkqde:prop_second_solution_for_N_equal_2}.
This solution will be matched with the coefficient in front of $(1-P^{-1})^2$ in the development of $\widetilde{J^{K\textnormal{th}}}(q,Q)$, using Example \ref{qkqde:ex_computation_qpoch} and the formula of the function $P^{-\ell_q(Q)}$ in Definition \ref{qkqde:def_J_fn}.
Since the computations are getting quite technical, we will stop there.

\subsection{About the special function used in the $J$-function}

In this transitory subsection we discuss the role played by the function $Q \mapsto P^{-\ell_q(Q)}$ in Givental's $K$-theoretical $J$-function $\widetilde{J^{K\textnormal{th}}}$.

We recall that the function $Q \mapsto P^{-\ell_q(Q)}$ is a solution of the $q$-difference equation with $K$-theoretical coefficients $\qdeop{(Q)} f_q(Q) = P^{-1} f_q(Q)$.
This function allowed us to replace the original $J$-function, which was a solution of a $q$-difference equation with $K$-theoretical coefficients, by a changed solution which is a solution of a $q$-difference equation with complex coefficients, in the sense of Definition \ref{qde:def_q_difference_linear_sys}.
To do so, we need to pick a $q$-character, solution of $\qdeop{(Q)} f_q(Q) = P^{-1} f_q(Q)$. In this subsection, we explain that our choice has a simple decomposition in the $K$-theoretical basis $\left( \left(1 - P^{-1} \right)^i \right)$, unlike the usual $q$-characters from the literature.

\begin{remark}
    Let us discuss why we chose the function $P^{-\ell_q(Q)}$ as a solution of the $q$-difference equation
    \[
        \qdeop{Q} f_q(Q) = P^{-1} f_q(Q)
    \]
    Notice that this $q$-difference equation is close to the one satisfied by the $q$-characters of Definition \ref{qde:def_e_qc_qchar}.
    We will explain in this remark that our choice is motivated by a computational reason.
    
    In a first time, we return to the usual $q$-characters of Subsection \ref{qde:subsection_ex_q_characters}.
    Let $\lambda_q \in \mathbb{C}^* - q^\mathbb{R}$. 
    We consider the $q$-difference equation $\qdeop{Q} f_q(Q) = \lambda_q f_q(Q)$. We can find two solutions
    \begin{align*}
        &
        e_{q,\lambda_q}(Q) = \frac{\theta_q(Q)}{\theta_q(\lambda_q Q)}
        &&
        \lambda_q^{\ell_q(Q)} = \sum_{ k \geq 0 } \frac{1}{k!} \left( \log(\lambda_q) \ell_q(Q) \right)^k
    \end{align*}
    These two solutions are in general different, and the second one depends on the choice of a logarithm (which is well defined here since $\lambda_q$ takes values in a simply connected space).
    Let us compare the confluence of these solutions.
    Assume that $\lambda_q = q^\mu$. We had already seen in Proposition \ref{qde:prop_specfns_limits} that for $Q \in \mathbb{C}^* - (-1)q^\mathbb{R}$, we had the uniform convergence
    \[
        \lim_{t \to 0} e_{q^t,\lambda_{q^t}}(Q) = Q^\mu
    \]
    For the second solution, we have
    \[
        \lambda_{q^t}^{\ell_{q^t}(Q)}
        = e^{\mu \log(q^t) \ell_q(Q)} 
        \sim_{t \to 0}
        e^{\mu (q^t-1) \ell_{q^t}(Q)}
    \]
    Using Proposition \ref{qde:prop_specfns_limits}, we obtain
    \[
        \lim_{t \to 0} \lambda_{q^t}^{\ell_{q^t}(Q)}
        =
        e^{\mu \log(Q)}
        =
        Q^\mu
    \]
    Thus both solutions $e_{q,\lambda_q}(Q)$ and $\lambda_q^{\ell_q(Q)}$ are confluent without the need of a change of fundamental solution, and have the same limit.
    This means that choosing one or the other has no influence on confluence.
    
    In a second time, we consider instead the $q$-difference equation with $K$-theoretical coefficients
    \[
        \qdeop{Q} f_q(Q) = P^{-1} f_q(Q)
    \]
    We want to build the analogues of the two solutions above. Since $1-P^{-1} \in K\left( \mathbb{P}^N \right)$ is nilpotent, we can define the following functions
    \begin{align*}
        e_{q,P^{-1}}(Q) &:= \frac{\theta_q(Q)}{\theta_q(P^{-1} Q)}
        = \frac{(Q;q)_\infty \left( \frac{q}{Q};q \right)_\infty}
        {(P^{-1} Q;q)_\infty \left(\frac{q}{P^{-1} Q};q\right)_\infty}
        \\
        P^{-\ell_q(Q)}
        &=
        \left(1-(1-P^{-1})\right)^{\ell_q(Q)}
        :=
        \sum_{k \geq 0} (-1)^k \binom{\ell_q(Q)}{k} \left( 1 - P^{-1} \right)^k
    \end{align*}
    Notice that it is harder a obtain a decomposition of the first function in our basis of $K\left( \mathbb{P}^N \right)$.
\end{remark}

\subsection{Recall: Givental's cohomological $J$-function}


Let us recall the definition of the cohomological small $J$-function.
In this section we will use a slightly different form of the cohomological $J$-function compared to Definition \ref{HGW:def_J_fn}.
\begin{defin}
    The small cohomological $J$-function is given by the expression
    \begin{equation}
        \widetilde{J^\textnormal{coh}}(z,Q) = Q^{\frac{H}{z}} \sum_{d \geq 0} \frac{Q^d}{\prod_{r=1}^d \left( H + rz\right)^{N+1}}
        \in
        H\left( \mathbb{P}^N \right) \otimes \mathbb{C}[z,z^{-1}][\![Q]\!]
    \end{equation}
\end{defin}

\begin{prop}
    This function is a solution of the differential equation
    \begin{equation}\label{qkqde:eqn_pde_JH}
        \left[(zQ \partial_Q)^{N+1} - Q\right]
        \widetilde{J^\textnormal{coh}}(z,Q) = 0
    \end{equation}
\end{prop}

The strategy we use to prove this result is similar to the proof of Proposition \ref{qkqde:prop_JK_sol_qde_pre}, which is the $K$-theoretical analogue of this proposition.
\begin{proof}
    First, notice that
    \[
        zQ \partial_Q \left(
            Q^{\frac{H}{z}}
        \right)
        =
        H Q^\frac{H}{z}
    \]
    Thus, we have
    \[
        (zQ \partial_Q)^{N+1} Q^{\frac{H}{z}} \frac{Q^d}{\prod_{r=1}^d \left( H + rz\right)^{N+1}}
        =
        (H + dz)^{N+1} Q^{\frac{H}{z}} \frac{Q^d}{\prod_{r=1}^d \left( H + rz\right)^{N+1}}
        =
        Q^{\frac{H}{z}} \frac{Q^d}{\prod_{r=1}^{d-1} \left( H + rz\right)^{N+1}}
    \]
    Therefore, we get
    \[
        \left[(zQ \partial_Q)^{N+1} - Q\right]
        \widetilde{J^\textnormal{coh}}(z,Q)
        =
        H^{N+1} Q^\frac{H}{z}
    \]
    We conclude using the relation $H^{N+1}=0$ in $H^* \left( \mathbb{P}^N \right)$.
\end{proof}

\begin{remark}[Comparison with Definition \ref{HGW:def_J_fn}]
    In Definition \ref{HGW:def_J_fn}, or more precisely in Proposition \ref{HGW:ex_projspace_Jfn}, we had
    \[
        j^\textnormal{coh}(t_1,z) = e^{\frac{t_1 H}{z}} \sum_{d \geq 0} \frac{e^{t_1 d}}{\prod_{r=1}^d \left( H + rz\right)^{N+1}}
    \]
    Which was a solution of the differential equation
    \[
        \left[ (z \partial_{t_1})^{N+1} - e^{t_1} \right] j^\textnormal{coh}(t_1,z) = 0
    \]
    We obtain the $J$-function $\widetilde{J^\textnormal{coh}}(z,Q)$ from $j^\textnormal{coh}(t_1,z)$ by setting $Q := e^{t_1}$.
    Then, we have $\partial_{t_1} e^{t_1} = Q \partial_Q Q$.
\end{remark}

Let us describe the role played by the function $Q \mapsto Q^{\frac{H}{z}}$ for the cohomological $J$-function
\begin{defin}
    We introduce another $J$-function given by
    \[
        J^\textnormal{coh}(z,Q)
        =
        \sum_{d \geq 0} \frac{Q^d}{\prod_{r=1}^d \left( H + rz\right)^{N+1}}
    \]
\end{defin}

\begin{prop}
    The function $J^\textnormal{coh}$ is solution of the differential equation with cohomological coefficients
    \[
        \left[
            \left(H + z Q \partial_Q\right)^{N+1} - Q
        \right]
         J^\textnormal{coh}(z,Q)
         =0
    \]
\end{prop}

Therefore, we observe that just like $Q \mapsto P^{-\ell_q(Q)}$ in the $K$-theoretical setting, the function $Q \mapsto Q^{\frac{H}{z}}$ gives a $J$-function that satisfies a functional equation with complex coefficients.
We could expect the confluence of the function $Q \mapsto P^{-\ell_q(Q)}$ to be related to the function $Q \mapsto Q^{\frac{H}{z}}$.
In the next section, we will show that it is true.


    

\section{Confluence for small quantum $K$-theory of projective spaces}

\subsection{Statement of the theorem}

For us to apply $q$-difference equation's confluence phenomenon to Givental's $K$-theoretical $J$-function, we have to think the class $q \in K_{\mathbb{C}^*}(\textnormal{pt})$ as a parameter $q \in \mathbb{C}^*, |q| < 1$.
Then confluence will correspond to taking the limit of this parameter $q \to 1$ along a $q$-spiral.
Similarly, the class $z \in H_{\mathbb{C}^*}(\textnormal{pt})$ will be seen as a parameter $z \in \mathbb{C}^*$.

The goal of this subsection is to prove the following theorem:

\begin{thm}\label{qkqde:thm_JK_sol_confluence_statement}
    For $X = \mathbb{P}^N$, let $\widetilde{J^{K\textnormal{th}}}$ (resp. $\widetilde{J^\textnormal{coh}}$) be the small $K$-theoretical (resp. cohomological) $J$-function.
    \begin{enumerate}[label=(\roman*)]
        \item Making the $q$-difference equation satisfied by $\widetilde{J^{K\textnormal{th}}}$ (\ref{qkqde:eqn_JK}) confluent yields the differential system satisfied by $J^\textnormal{coh}$, (\ref{qkqde:eqn_pde_JH}).
        \item  Consider the ring automorphism
        $ \gamma : K\left( \mathbb{P}^N \right) \otimes \mathbb{Q} \to H^*\left( \mathbb{P}^N; \mathbb{Q}\right)$ defined by $\gamma(1-P^{-1})=H$ and
        let $\textnormal{confluence}\left(\widetilde{J^{K\textnormal{th}}}\right)$ be the result of confluence applied to the solution $\widetilde{J^{K\textnormal{th}}}$ of the above $q$-difference system.
        Then, we have
        \[
            \gamma\left(\textnormal{confluence}\left(\widetilde{J^{K\textnormal{th}}}\right)\right)(z,Q) = \widetilde{J^\textnormal{coh}}(z,Q)
        \]
    \end{enumerate}
\end{thm}

For a precise definition of the $K$-theoretical function $\textnormal{confluence}\left(\widetilde{J^{K\textnormal{th}}}\right)$, we will refer to Definition \ref{qkqde:def_confluence_applied_to_JK}.
The proof of this theorem is structured in three parts, which get their own subsection:
\begin{enumerate}[label=(\roman*)]
    \item Confluence of the $q$-difference equation (see Subsection \ref{qkqde:subsection_comparison_confluence_equation} and Proposition \ref{qkqde:prop_jk_eqn_pullback}).
    \item Confluence of the fundamental solution (see Subsection \ref{qkqde:subsection_comparison_confluence_solution} and Proposition \ref{qkqde:prop_JK_confluence_sol})
    \item Comparison between the confluence of the solution and the cohomological $J$-function (see Subsection \ref{qkqde:subsection_comparison_confluence_conclusion} and Proposition \ref{qkqde:prop_confluence_JK_equals_JH}).
\end{enumerate}

For the sake of comparison, we give the developments of the small $J$-functions in the case $X = \mathbb{P}^2$.
\begin{example}[Example \ref{qkqde:ex_computation_qpoch}]\label{qkqde:ex_JK_P2_again}
    In the case of $\mathbb{P}^2$, the partial decomposition of the $K$-theoretical $J$-function $\widetilde{J^{K\textnormal{th}}_{X=\mathbb{P}^2}}(q,Q)$ in the basis $\left(\mathds{1},1-P^{-1}, \left( 1 - P^{-1} \right)^2\right)$ of $K \left( \mathbb{P}^2 \right)$ is given by
    \[
        P^{-\ell_q(Q)}
        \sum_{d \geq 0}
        \frac{Q^d}{(q;q)_d^3}
        \left(
            \mathds{1}
            +
            \left( 1 - P^{-1} \right)
            \left[
                -3 \sum_{j=1}^d \frac{q^j}{1-q^j}
            \right]
            +
            \left( 1 - P^{-1} \right)^2
            \left[
                6 \left( \sum_{j=1}^d \frac{q^j}{1-q^j} \right)^2
                -
                3
                \sum_{0 \leq i < j \leq d}
                \frac{q^{i+j}}{ (1-q^i)(1-q^j) }
            \right]
        \right)
    \]
\end{example}

\begin{example}\label{qkqde:ex_JH_P2}
    In the case of $\mathbb{P}^2$, the partial decomposition of the cohomological $J$-function in the basis $\left(\mathds{1},H,H^2 \right)$ of $H^*\left(\mathbb{P}^2;\mathbb{Q}\right)$ is given by
    \[
        \widetilde{J^\textnormal{coh}_{X=\mathbb{P}^2}}(z,Q)
        =
        Q^{\frac{H}{z}}
        \sum_{d \geq 0} \frac{Q^d}{\left( z^d d! \right)^3}
        \left[
            \mathds{1}
            - H \left( 3\sum_{j=1}^d \frac{1}{jz} \right)
            + H^2 \left( 6 \left(\sum_{j=1}^d \frac{1}{jz}\right)^2 - 3 \sum_{1 \leq j_1 < j_2 \leq d} \frac{1}{j_1 j_2 z^2} \right)
        \right]
    \]
\end{example}

More precisely, through confluence of the $K$-theoretical $J$-function, we expect to obtain the correspondences

\[
    \begin{small}
    \begin{tikzcd}
        QK \left( \mathbb{P}^2 \right)
        &
        QH^* \left( \mathbb{P}^2 \right)
        \\
        \textnormal{Basis } \mathds{1}, 1-P^{-1}, \left( 1-P^{-1} \right)^2
            \arrow[r, mapsto, "\gamma"]
        &
        \textnormal{Basis } \mathds{1}, H, H^2
        \\
        q\textnormal{-character } P^{-\ell_q(Q)}
            \arrow[r, rightsquigarrow,"q^t \to 1"] 
        &
        \textnormal{Character }Q^{\frac{H}{z}}
        \\
        \sum_{d \geq 0} \frac{Q^d}{(q;q)_d^3}
            \arrow[r, rightsquigarrow,"q^t \to 1"] 
        &
        \sum_{d \geq 0} \frac{Q^d}{\left( z^d d! \right)^{3}}
        \\
        -3 \sum_{d \geq 0} \frac{Q^d}{(q;q)_d^3}
        \sum_{j=1}^d \frac{q^j}{1-q^j}
            \arrow[r, rightsquigarrow,"q^t \to 1"] 
        &
        -3 \sum_{d \geq 0} \frac{Q^d}{\left( z^d d! \right)^{3}}
        \left( \sum_{j=1}^d \frac{1}{jz} \right)
        \\
        \sum_{d \geq 0} \frac{Q^d}{(q;q)_d^3}
        \left[
            6 \left( \sum_{j=1}^d \frac{q^j}{1-q^j} \right)^2
            -
            3
            \sum_{0 \leq i < j \leq d}
            \frac{q^{i+j}}{ (1-q^i)(1-q^j) }
        \right]
            \arrow[r, rightsquigarrow,"q^t \to 1"] 
        &
        \sum_{d \geq 0} \frac{Q^d}{\left( z^d d! \right)^{3}}
        \left[
            6 \left(\sum_{j=1}^d \frac{1}{jz}\right)^2 - 3 \sum_{1 \leq j_1 < j_2 \leq d} \frac{1}{j_1 j_2 z^2}
        \right]
    \end{tikzcd}
    \end{small}
\]

\subsection{Confluence of the $q$-difference equation of the small $J$-function}\label{qkqde:subsection_comparison_confluence_equation}

\begin{notation}
    In this chapter the indices of the matrices will start at $0$ instead of $1$. This means we will write a matrix of size $(N+1) \times (N+1)$ in a vector space $V$ as
    \[
        A = (A_{i,j})_{i,j \in \{0,\dots,N\}} \in \textnormal{M}_{N+1}(V)
    \]
\end{notation}

\begin{remark}\label{qkqde:thm_JK_sol_confluence_rmk}
    To check the confluence of the equation with respect to Definition \ref{qde:def_confluentsys}, let us write our equations in matrix form.
    The differential equation (\ref{qkqde:eqn_pde_JH}) satisfied by the cohomological $J$-function can be translated to the differential system
    \begin{equation}\label{qkqde:eqn_jh_pde_system}
        Q \partial_Q
        \begin{pmatrix}
            f           \\
            Q\partial_Q f  \\
            \vdots      \\
            \left(Q\partial_Q\right)^N f
        \end{pmatrix}
        =
        \begin{pmatrix}
            0                       &   1       &   0       &   \cdots  &   0      \\
            0                       &   0       &   \ddots  &           &   \vdots  \\
            \vdots                  &   \vdots  &           &   \ddots  &   0       \\
            0                       &   \vdots   &          &           &   1       \\
            \frac{Q}{z^{N+1}}   &   0       &   \cdots  &   \cdots  &   0
        \end{pmatrix}
        \begin{pmatrix}
            f           \\
            Q\partial_Q f  \\
            \vdots      \\
            \left(Q\partial_Q\right)^N f
        \end{pmatrix}
    \end{equation}
    Write $\delta_q = \frac{\qdeop{Q}-\textnormal{Id}}{q-1}$. The $q$-difference equation (\ref{qkqde:eqn_JK}) becomes the $q$-difference system
    \begin{equation}\label{qkqde:eqn_jk_system}
        \delta_q
        \begin{pmatrix}
            f           \\
            \delta_q f  \\
            \vdots      \\
            \delta_q^N f
        \end{pmatrix}
        =
        \begin{pmatrix}
            0                       &   1       &   0       &   \cdots  &   0      \\
            0                       &   0       &   \ddots  &           &   \vdots  \\
            \vdots                  &   \vdots  &           &   \ddots  &   0       \\
            0                       &   \vdots   &          &           &   1       \\
            \frac{Q}{(1-q)^{N+1}}   &   0       &   \cdots  &   \cdots  &   0
        \end{pmatrix}
        \begin{pmatrix}
            f           \\
            \delta_q f  \\
            \vdots      \\
            \delta_q^N f
        \end{pmatrix}
    \end{equation}
\end{remark}

\begin{prop}\label{qkqde:prop_jk_eqn_pullback}
    Consider the $q$-difference equation (\ref{qkqde:eqn_JK}) :
    $
        \left( 1 - q^{Q \partial_Q} \right)^{N+1} f(q,Q) = Q f(q,Q)
    $
    Let $z \in \mathbb{C}^*$ and let $\varphi_{q,z}$ be the function
    \[
        \functiondesc{\varphi_{q,z}}{\mathbb{C}}{\mathbb{C}}{Q}{\left( \frac{z}{1-q} \right)^{N+1} Q}
    \]
    
    Then, the $q$-pullback of the $q$-difference equation (\ref{qkqde:eqn_JK}) by $\varphi_{q,z}$ is confluent, and its limit is the differential equation (\ref{qkqde:eqn_pde_JH}) satisfied by the small cohomological $J$-function
    \[
    \left(zQ \partial_Q\right)^{N+1} J^\textnormal{coh}(z,Q) = Q J^\textnormal{coh}(z,Q)
    \]
\end{prop}

\begin{proof}
    Replace the $q$-difference equation (\ref{qkqde:eqn_JK}) by the $q$-difference system (\ref{qkqde:eqn_jk_system})
    \[
        \delta_q
        \begin{pmatrix}
            f           \\
            \delta_q f  \\
            \vdots      \\
            \delta_q^N f
        \end{pmatrix}
        =
        \begin{pmatrix}
            0                       &   1       &   0       &   \cdots  &   0      \\
            0                       &   0       &   \ddots  &           &   \vdots  \\
            \vdots                  &   \vdots  &           &   \ddots  &   0       \\
            0                       &   \vdots   &          &           &   1       \\
            \frac{Q}{(1-q)^{N+1}}   &   0       &   \cdots  &   \cdots  &   0
        \end{pmatrix}
        \begin{pmatrix}
            f           \\
            \delta_q f  \\
            \vdots      \\
            \delta_q^N f
        \end{pmatrix}
    \]
    Then, the $q$-pullback of this $q$-difference system by the isomorphism $\varphi : Q \mapsto \left( \frac{z}{1-q} \right)^{N+1} Q$ is given by
    \[
        \delta_q
        \begin{pmatrix}
            f           \\
            \delta_q f  \\
            \vdots      \\
            \delta_q^N f
        \end{pmatrix}
        =
        \begin{pmatrix}
            0                       &   1       &   0       &   \cdots  &   0      \\
            0                       &   0       &   \ddots  &           &   \vdots  \\
            \vdots                  &   \vdots  &           &   \ddots  &   0       \\
            0                       &   \vdots   &          &           &   1       \\
            \frac{Q}{z^{N+1}}   &   0       &   \cdots  &   \cdots  &   0
        \end{pmatrix}
        \begin{pmatrix}
            f           \\
            \delta_q f  \\
            \vdots      \\
            \delta_q^N f
        \end{pmatrix}
    \]
    Recall that we have the formal limit $\lim_{q \to 1} \delta_q = Q \partial_Q$.
    By taking $q \to 1$, the $q$-difference system above has the formal limit
    \[
        Q \partial_Q
        \begin{pmatrix}
            f           \\
            Q\partial_Q f  \\
            \vdots      \\
            \left(Q\partial_Q\right)^N f
        \end{pmatrix}
        =
        \begin{pmatrix}
            0                       &   1       &   0       &   \cdots  &   0      \\
            0                       &   0       &   \ddots  &           &   \vdots  \\
            \vdots                  &   \vdots  &           &   \ddots  &   0       \\
            0                       &   \vdots  &           &           &   1       \\
            \frac{Q}{z^{N+1}}       &   0       &   \cdots  &   \cdots  &   0
        \end{pmatrix}
        \begin{pmatrix}
            f           \\
            Q\partial_Q f  \\
            \vdots      \\
            \left(Q\partial_Q\right)^N f
        \end{pmatrix}
    \]
    Which corresponds to the differential system (\ref{qkqde:eqn_jh_pde_system}) satisfied by the cohomological $J$-function.
\end{proof}

Let us make a few remarks on the $q$-pullback $\varphi_{q,z}$ we have used in this proposition.

\begin{remark}\label{qkqde:remark_formula_for_qpullback}
    \textit{(i)} \quad The $q$-difference system (\ref{qkqde:eqn_jk_system}) of the $K$-theoretical $J$-function is not confluent, so we look for a $q$-pullback $\varphi_{q,z}$ to make it confluent.
    From the point of view of $q$-difference equations, the natural pullback to use is $Q \mapsto (1-q)^{-(N+1)}Q$, cf. Remark \ref{qde:remark_when_system_is_not_confluent}.
    By doing so, we obtain as formal limit the differential system satisfied by the function $\left( Q \mapsto \widetilde{J^{\textnormal{coh}}}(1,Q) \right)$.
    
    \bigskip
    \textit{(ii)} \quad
    In the $\mathbb{C}^*$-equivariant cohomology $H_{\mathbb{C}^*}^*(\textnormal{pt})$, we recall that we have the relation $\textnormal{ch}(q) = e^{-z} = 1-z \in H_{\mathbb{C}^*}^*(\textnormal{pt})$ (\cite{CK_book}, (9.1)).
    An easy way to modify the $q$-pullback of \textit{(i)} is to consider instead the $q$-pullback $\varphi_{q,z}$ of the Proposition \ref{qkqde:prop_jk_eqn_pullback}: $\left( Q \mapsto (z/(1-q))^{N+1} Q\right)$.
    
    \bigskip
    \textit{(iii)} \quad
    Notice that we have in small quantum cohomology the relation $H^{\circ (N+1)} = Q$, which means the Novikov variable has degree $N+1$.
    This degree appears also in the $q$-difference system (\ref{qkqde:eqn_jk_system}) and in the exponent in the formula for $\varphi_{q,z}$.
\end{remark}

\begin{remark}\label{qkqde:remark_JK_qpullback_is_natural}
    The $q$-pullback $\varphi_{q,z}$ defined in Proposition \ref{qkqde:prop_jk_eqn_pullback} is natural in the following way: it is the only $q$-pullback of the form $Q \mapsto \left( \frac{z}{1-q} \right)^\lambda Q$, with $\lambda \in \mathbb{Z}$, which defines a confluent $q$-difference system whose formal limit is non zero.
    
    Let us verify that. Let $\varphi_{\lambda,q,z}$ be the function defined by
    \[
        \functiondesc{\varphi_{\lambda,q,z}}{\mathbb{C}}{\mathbb{C}}{Q}{\left( \frac{z}{1-q} \right)^\lambda Q}
    \]
    Then, the $q$-pullback of the $q$-difference system (\ref{qkqde:eqn_jk_system}) by $\varphi_{\lambda,q,z}$ is given by system
    \[
        \delta_q
        \begin{pmatrix}
            f           \\
            \delta_q f  \\
            \vdots      \\
            \delta_q^N f
        \end{pmatrix}
        =
        \begin{pmatrix}
            0                       &   1       &   0       &   \cdots  &   0      \\
            0                       &   0       &   \ddots  &           &   \vdots  \\
            \vdots                  &   \vdots  &           &   \ddots  &   0       \\
            0                       &   \vdots   &          &           &   1       \\
            \left( \frac{1-q}{z} \right)^\lambda \frac{Q}{(1-q)^{N+1}}   &   0       &   \cdots  &   \cdots  &   0
        \end{pmatrix}
        \begin{pmatrix}
            f           \\
            \delta_q f  \\
            \vdots      \\
            \delta_q^N f
        \end{pmatrix}
    \]
    The only value of $\lambda$ such that this $q$-difference system has a well defined and non trivial limit when $q$ tends to $1$ is for $\lambda = N+1$.
\end{remark}


\subsection{Confluence of the $J$-function}\label{qkqde:subsection_comparison_confluence_solution}

We recall that by Proposition \ref{qkqde:prop_JK_is_a_fund_sol}, we can see the $K$-theoretical $J$-function as a solution of the $q$-difference equation (\ref{qkqde:eqn_JK}).
Our goal is to obtain the confluence of the $q$-pullback by $\varphi_{q,z}$ of this fundamental solution.
Before giving a statement, we need to describe this fundamental solution.

\begin{remark}
    Consider the $q$-difference system (\ref{qkqde:eqn_jk_system}) associated to the $K$-theoretical $J$-function
    \[
        \delta_q
        \begin{pmatrix}
            f           \\
            \delta_q f  \\
            \vdots      \\
            \delta_q^N f
        \end{pmatrix}
        =
        \begin{pmatrix}
            0                       &   1       &   0       &   \cdots  &   0      \\
            0                       &   0       &   \ddots  &           &   \vdots  \\
            \vdots                  &   \vdots  &           &   \ddots  &   0       \\
            0                       &   \vdots   &          &           &   1       \\
            \frac{Q}{(1-q)^{N+1}}   &   0       &   \cdots  &   \cdots  &   0
        \end{pmatrix}
        \begin{pmatrix}
            f           \\
            \delta_q f  \\
            \vdots      \\
            \delta_q^N f
        \end{pmatrix}
    \]
    Take the decomposition
    \[
        \widetilde{J^{K\textnormal{th}}}(q,Q)
        =
        \sum_{i=0}^N \widetilde{J_i}(q,Q) \left(1-P^{-1} \right)^i
        \in
        K \left( \mathbb{P}^N \right) \otimes \mathbb{C}(q) [\![Q]\!]
    \]
    The fundamental solution (see Proposition \ref{qkqde:prop_JK_is_a_fund_sol}) associated to the $J$-function is the matrix
    \[
        \mathcal{X}^{K\textnormal{th}}(q,Q)
        =
        \begin{pmatrix}
            \widetilde{J_0}(q,Q)            &   \widetilde{J_1}(q,Q)            &    \cdots  &   \widetilde{J_N}(q,Q)               \\
            \delta_q \widetilde{J_0}(q,Q)   &   \delta_q  \widetilde{J_1}(q,Q)  &    \cdots  &   \delta_q \widetilde{J_N}(q,Q)      \\
            \vdots                          &   \vdots                          &    \ddots  &   \vdots                             \\
            \delta_q^N \widetilde{J_0}(q,Q) &   \delta_q^N \widetilde{J_1}(q,Q) &    \cdots  &   \delta_q^N \widetilde{J_N}(q,Q)    \\
        \end{pmatrix}
    \]
    The first line of this matrix is of particular interest for us as it contains the $J$-function.
\end{remark}


Recall that the $q$-pullback $\varphi_{q,z}$ was given by the function
    \[
        \functiondesc{\varphi_{q,z}}{\mathbb{C}}{\mathbb{C}}{Q}{\left( \frac{z}{1-q} \right)^{N+1} Q}
    \]
Therefore, the $q$-pullback by the $\varphi_{q,z}$ of the fundamental solution $\mathcal{X}^{K\textnormal{th}}$ is given by
\begin{equation}\label{qkqde:eqn_pullback_of_fundamental_solution_matrix_form}
    \mathcal{X}^{K\textnormal{th}} \left(q,\varphi_{q,z}^{-1}(Q) \right)
    =
    \begin{pmatrix}
        \widetilde{J_0}\left(q,\left( \frac{1-q}{z}\right)^{N+1} Q \right)            &   \widetilde{J_1}\left(q,\left( \frac{1-q}{z}\right)^{N+1} Q \right)            &    \cdots  &   \widetilde{J_N}\left(q,\left( \frac{1-q}{z}\right)^{N+1} Q \right)               \\
        \delta_q \widetilde{J_0}\left(q,\left( \frac{1-q}{z}\right)^{N+1} Q \right)   &   \delta_q  \widetilde{J_1}\left(q,\left( \frac{1-q}{z}\right)^{N+1} Q \right)  &    \cdots  &   \delta_q \widetilde{J_N}\left(q,\left( \frac{1-q}{z}\right)^{N+1} Q \right)      \\
        \vdots                          &   \vdots                          &    \ddots  &   \vdots                             \\
        \delta_q^N \widetilde{J_0}\left(q,\left( \frac{1-q}{z}\right)^{N+1} Q \right) &   \delta_q^N \widetilde{J_1}\left(q,\left( \frac{1-q}{z}\right)^{N+1} Q \right) &    \cdots  &   \delta_q^N \widetilde{J_N}\left(q,\left( \frac{1-q}{z}\right)^{N+1} Q \right)    \\
    \end{pmatrix}
\end{equation}

\begin{prop}\label{qkqde:prop_JK_confluence_sol}
    Let $\varphi_{q,z}$ be the $q$-pullback of Proposition \ref{qkqde:prop_jk_eqn_pullback}.
%
    There exists a $q$-constant matrix $P_{q,z} \in \textnormal{GL}_{N+1}\left( \mathcal{M}\left( \mathbb{E}_q \right) \right)$ such that the transformed fundamental solution $ \mathcal{X}^{K\textnormal{th}} \left(q,\varphi_{q,z}^{-1}(Q) \right) P_{q,z} $ obtained from Equation \ref{qkqde:eqn_pullback_of_fundamental_solution_matrix_form} is given by
    \[
        \left(
            \mathcal{X}^{K\textnormal{th}} \left(q,\varphi_{q,z}^{-1}(Q) \right)P_{q,z}
        \right)_{li}
        =
        \left(\delta_q\right)^l
        \sum_{\substack{
                0 \leq a, b \leq N  \\  a + b = i
            }}
        \left( \frac{q-1}{z} \right)^a 
        \binom{\ell_{q} \left( Q \right)}{a}
        \left( \frac{1-q}{z} \right)^b J_b \left(
            q,\left( \frac{1-q}{z}\right)^{N+1} Q
        \right)
    \]
    Where the functions $J_b$ are defined in Equation \ref{qkqde:eqn_decomposition_JK_taylor_series}.
    Moreover, this fundamental solution has a non trivial limit when $q^t$ tends to $1$.
\end{prop}

Before giving a proof, let us comment on the characterisation of the change of fundamental solution $P_{q,z}$.
While the condition on the first line of the new fundamental solution may seem arbitrary, we will see in Proposition \ref{qkqde:prop_confluence_JK_equals_JH} that it is closely related to the cohomological $J$-function.
Furthermore, we will discuss in Remark \ref{qkqde:remark_naturality_of_gauge_transform} the naturality of this transformation.
Let us detail these formulas in the concrete example of $\mathbb{P}^2$.

\begin{example}
    In the case of $\mathbb{P}^2$, let us apply the $q$-pullback $\varphi_{q,z}$ and the change of fundamental solution $P_{q,z}$ to the small $J$-function.
    Putting the first line of the fundamental solution $ \left( \mathcal{X}^{K\textnormal{th}}\left(q^t,\varphi_{q,z}^{-1}(Q) \right) \right) P_{q,z}$ as a $K$-theoretical function again, we obtain the function
    \begin{small}
    \begin{align*}
        &\left(
            \sum_{k \geq 0} (-1)^k 
            \left(
                \frac{1-q}{z} 
            \right)^k \binom{\ell_q(Q)}{k}
            \left( 
                1 - P^{-1}
            \right)^k
        \right)
        \left(
            \sum_{d \geq 0} Q^d 
            \frac{(1-q)^{3d}}{(q;q)_d^3}
        \right. \times
        \\
        & \times \left.
            \left(
                \mathds{1}
                +
                \left(1-P^{-1}\right) \left[
                    -\frac{3}{z} \sum_{m=1}^d \frac{q^m(1-q)}{1-q^m}
                \right]
                +
                \left(1-P^{-1}\right)^2  \frac{1}{z^2}
                \left[
                    6 \left(
                        \sum_{m=1}^d \frac{q^m(1-q)}{1-q^m}
                    \right)^2
                    -3 \sum_{0 \leq m_1 < m_2\leq d} \frac{q^{m_1+m_2} (1-q)^2}{(1-q^{m_1})(1-q^{m_2})}
                \right]
            \right)
        \right)
    \end{align*}
    \end{small}
    Which has the limit when $q^t \to 1$ given by
    \begin{small}
    \[
        \left(
            \sum_{k \geq 0} 
            \frac{\log(Q)}{z^k}
            \left( 
                1 - P^{-1} 
            \right)^k
        \right)
        \left(
            \sum_{d \geq 0}
            \frac{Q^d}{\left(z^d d!\right)^3}
            \left(
                \mathds{1}
                -
                \left(1-P^{-1}\right) 
                \frac{3}{z} \sum_{m=1}^d \frac{1}{m}
                +
                \left(1-P^{-1}\right)^2
                \frac{1}{z^2}
                \left( 6 \left(\sum_{m=1}^d \frac{1}{m}\right)^2 - 3 \sum_{1 \leq m_1 < m_2 \leq d} \frac{1}{m_1 m_2} \right)
            \right)
        \right)
    \]
    \end{small}
    We invite the reader to compare this limit to the cohomological $J$-function given in Example \ref{qkqde:ex_JH_P2}: after using the isomorphism $\gamma$ sending the class $\left( 1 - P^{-1} \right)^i$ to $\gamma\left(\left( 1 - P^{-1} \right)^i\right) = H^i$, the limit coincides with $\widetilde{J^\textnormal{coh}_{X=\mathbb{P}^2}}(z,Q)$.
\end{example}


\begin{proof}[Proof of Proposition \ref{qkqde:prop_JK_confluence_sol}]
    We recall that
    \[
    \widetilde{J_i} \left(
        q, \varphi^{-1}(Q)
    \right)
    =
    \sum_{\substack{
            a + b = i  \\  0 \leq a, b \leq N
        }}
        (-1)^a \binom{\ell_q \left(\left( \frac{1-q}{z}\right)^{N+1} Q \right)}{a}
        J_b \left(
            q,\left( \frac{1-q}{z}\right)^{N+1} Q
        \right) 
    \]
    Where the functions $J_b$ are obtained from the decomposition of $J^{K\textnormal{th}}$,
    \[
        J_b(q,Q) = \sum_{d \geq 0} \frac{Q^d}{(q;q)_d^{N+1}}
        \left(
            \sum_{k=0}^N \sum_{\substack{
            0 \leq j_1,\dots,j_N \leq N \\  j_1 + \cdots + j_N = k \\ j_1 + 2 j_2 + \cdots + N j_N = b
            }}
            (-1)^k \frac{(N+k)!}{N! j_1! \cdots j_N!}
            \prod_{l=1}^N \left(
                \sum_{1 \leq m_1 < \cdots < m_l \leq d} \frac{q^{m_1+\cdots+m_l}}{(1-q^{m_1}) \cdots (1-q^{m_l})}
            \right)^{j_l}
        \right)
    \]
    
    We are going to exhibit a $q$-constant matrix $P_{q,z} \in \textnormal{GL}_{N+1}\left( \mathcal{M}\left( \mathbb{E}_q \right) \right)$ such that the limit of the transformed fundamental solution
    \[
        \lim_{t \to 0} \mathcal{X}^{K\textnormal{th}}\left(q^t,\varphi_{q^t,z}^{-1}(Q) \right)P_{q^t,z}
    \]
    is well defined.
    The transformation $P_q$ will be constructed in two steps.
    \begin{enumerate}[label=(\roman*)]
        \item Notice that in the formula of $\widetilde{J_i}\left(q,\varphi^{-1}(Q)\right)$, the $q$-logarithms  $\ell_q \left(\left( \frac{1-q}{z}\right)^{N+1} Q \right)$, do not have a well defined limit.
        We are going to construct a first transformation $A_{q,z}$ to change them into the $q$-logarithms $\ell_q(Q)$.
        \item After this change, we will require to multiply $\widetilde{J_i}$ by $\left( \frac{1-q}{z} \right)^i$ to get a well defined limit that is not zero.
        We will construct a second transformation $B_{q,z}$ to do that.
    \end{enumerate}
    In the end, the function $\widetilde{J_i}\left(q,\varphi^{-1}(Q)\right)$ will be replaced by the function
    \[
        \left( \frac{1-q}{z} \right)^i
        \sum_{\substack{
            a + b = i  \\  0 \leq a, b \leq N
        }}
        (-1)^a \binom{\ell_q \left(Q \right)}{a}
        J_b \left(
            q,\left( \frac{1-q}{z}\right)^{N+1} Q
        \right) 
    \]
    Then, we will show that the transformation given by $P_{q,z} = A_{q,z} B_{q,z}$ satisfies the conditions imposed by the proposition.
    
    We begin by making a first change of fundamental solution to modify the $q$-logarithms $\ell_q \left(\left( \frac{1-q}{z}\right)^{N+1} Q \right)$.
    Notice that we can not use the asymptotic of the $q$-logarithm of Proposition \ref{qde:prop_specfns_limits}, as the input of the $q$-logarithm tends to 0.
    We recall that a $q$-logarithm is a solution of the $q$-difference equation $\qdeop{Q} f_q(Q) = f_q(Q) + 1$.
    To be able to apply Proposition \ref{qde:prop_specfns_limits} to the matrix $\mathcal{X}^{K\textnormal{th}} \left(q,\varphi_{q,z}^{-1}(Q) \right)$, we are going to take the other $q$-logarithm $\widetilde{\ell_{q,z}}$ defined by
    \[
        \widetilde{\ell_{q,z}}(Q)
        =
        \ell_q \left( \left( \frac{1-q}{z} \right)^{N+1} Q \right)
    \]
    By replacing in the matrix $\mathcal{X}^{K\textnormal{th}} \left(q,\varphi_{q,z}^{-1}(Q) \right)$ the $q$-logarithms $\ell_q$ by $\widetilde{\ell_{q,z}}$, we obtain a new fundamental solution $\widetilde{\mathcal{X}^{K\textnormal{th}}}(q,Q)$.
    In particular, for $i \in \{0 , \dots , N\}$, denote by $\left( \widetilde{\mathcal{X}^{K\textnormal{th}}} \right)_{0i}(q,Q)$ the coefficient of the new fundamental solution on the first row and the $(i+1)^\textnormal{th}$ column. We have
    \[
        \left( \widetilde{\mathcal{X}^{K\textnormal{th}}} \right)_{0i}(q,Q)
        =
        \sum_{\substack{
                a + b = i  \\  0 \leq a, b \leq N
            }}
            (-1)^a \binom{\ell_q \left( Q \right)}{a}
            J_b \left(
                q,\left( \frac{1-q}{z}\right)^{N+1} Q
        \right) 
    \]
    So we will be able to apply Proposition \ref{qde:prop_specfns_limits} to compute the asymptotics of the matrix $\widetilde{\mathcal{X}^{K\textnormal{th}}}(q,Q)$.
    Since the matrix $\widetilde{\mathcal{X}^{K\textnormal{th}}}(q,Q)$ is another fundamental solution of the same $q$-difference equation, there exists a $q$-constant matrix $A_{q,z} \in \textnormal{GL}_{N+1}\left( \mathcal{M}\left( \mathbb{E}_q \right) \right)$
    replacing the fundamental solution $\mathcal{X}^{K\textnormal{th}} \left(q,\varphi_{q,z}^{-1}(Q) \right)$ by the transformed solution $\widetilde{\mathcal{X}^{K\textnormal{th}}}(q,Q) = \mathcal{X}^{K\textnormal{th}} \left(q,\varphi_{q,z}^{-1}(Q) \right) A_{q,z}$.
    
    Our new goal is to compute $\lim_{t\to 0} \left( \widetilde{\mathcal{X}^{K\textnormal{th}}} \right)_{0i}(q^t,Q)$, or more precisely, to find a transformation after which these limits exist, then to compute them.
    We recall that
    \[
        \binom{\ell_q \left( Q \right)}{a}
        =
        \frac{1}{a!} \prod_{r=0}^{a-1}(\ell_q(Q)-r)
        =
        \frac{1}{a!} \ell_q(Q)^a + \cdots
    \]
    And
    \begin{align*}
        &J_b \left(
            q,\left( \frac{1-q}{z}\right)^{N+1} Q
        \right)
        =   \\
        &\sum_{d \geq 0} \frac{(1-q)^{d(N+1)}Q^d}{z^{d(N+1)} (q;q)_d^{N+1}}
        \left(
            \sum_{k=0}^N \sum_{\substack{
            0 \leq j_1,\dots,j_N \leq N \\  j_1 + \cdots + j_N = k \\ j_1 + 2 j_2 + \cdots + N j_N = b
            }}
            (-1)^k \frac{(N+k)!}{N! j_1! \cdots j_N!}
            \prod_{l=1}^N \left(
                \sum_{1 \leq m_1 < \cdots < m_l \leq d} \frac{q^{m_1+\cdots+m_l}}{(1-q^{m_1}) \cdots (1-q^{m_l})}
            \right)^{j_l}
        \right)
    \end{align*}
    For the coefficient $\binom{\ell_q \left( Q \right)}{a}$ to have a well defined limit when $q^t$ tends to $1$, according to Proposition \ref{qde:prop_specfns_limits}, we should multiply it by $(1-q)^a$. Then, by the same proposition, we have
    \[
        \lim_{t \to 0}
        \left( \frac{1-q^t}{z} \right)^a 
        \binom{\ell_{q^t} \left( Q \right)}{a}
        =
        (-1)^a
        \frac{\log(Q)^a}
        {z^a}
    \]
    As for the second one, $J_b \left(
        q,\left( \frac{1-q}{z}\right)^{N+1} Q
    \right)$, the summand inside the sum on the multi-index $(j_1, \dots, j_N)$ only has a well defined limit after being multiplied by $(1-q)^{j_1 + 2j_2 + \cdots Nj_N} = (1-q)^b$.
    Then, we have
    \begin{align*}
        &f_b(z,Q) := \lim_{t \to 0}
        \left( \frac{1-q^t}{z} \right)^b J_b \left(
            q^t,\left( \frac{1-q^t}{z}\right)^{N+1} Q
        \right)
        \\&=
        \sum_{d \geq 0}
        \frac{Q^d}{\left(d!z^d\right)^{N+1}}
        \left(
            \frac{1}{z^b}
            \sum_{k=0}^N \sum_{\substack{
            0 \leq j_1,\dots,j_N \leq N \\  j_1 + \cdots + j_N = k \\ j_1 + 2 j_2 + \cdots + N j_N = b
            }}
            (-1)^k \frac{(N+k)!}{N! j_1! \cdots j_N!}
            \prod_{l=1}^N \left(
                \sum_{1 \leq m_1 < \cdots < m_l \leq d} \frac{1}{m_1 \cdots m_l}
            \right)^{j_l}
        \right)
    \end{align*}
    So, for the first line to have a well defined limit, we need to modify the coefficient $\left( \widetilde{\mathcal{X}^{K\textnormal{th}}} \right)_{0i}(q,Q)$ by $\left( \frac{1-q}{z} \right)^i \left( \widetilde{\mathcal{X}^{K\textnormal{th}}} \right)_{0i}(q,Q)$.
    Since we are multiplying by a scalars, the matrix
    \begin{equation}\label{qkqde:eqn_proof_confluence_fn_final_result_of_gauge_transforms}
        \mathcal{Y}^{K \textnormal{th}}(z,q,Q)
        \begin{pmatrix}
            \left( \widetilde{\mathcal{X}^{K\textnormal{th}}} \right)_{00}(q,Q)            &   \left( \frac{1-q}{z} \right) \left( \widetilde{\mathcal{X}^{K\textnormal{th}}} \right)_{01}(q,Q)            &    \cdots  &   \left( \frac{1-q}{z} \right)^N \left( \widetilde{\mathcal{X}^{K\textnormal{th}}} \right)_{0N}(q,Q)               \\
            \delta_q \left( \widetilde{\mathcal{X}^{K\textnormal{th}}} \right)_{00}(q,Q)   &   \left( \frac{1-q}{z} \right) \delta_q  \left( \widetilde{\mathcal{X}^{K\textnormal{th}}} \right)_{01}(q,Q) &    \cdots  &   \left( \frac{1-q}{z} \right)^N \delta_q \left( \widetilde{\mathcal{X}^{K\textnormal{th}}} \right)_{0N}(q,Q)      \\
            \vdots                          &   \vdots                          &    \ddots  &   \vdots                             \\
            \delta_q^N \left( \widetilde{\mathcal{X}^{K\textnormal{th}}} \right)_{00}(q,Q) &   \left( \frac{1-q}{z} \right) \delta_q^N \left( \widetilde{\mathcal{X}^{K\textnormal{th}}} \right)_{01}(q,Q) &    \cdots  &   \left( \frac{1-q}{z} \right)^N \delta_q^N \left( \widetilde{\mathcal{X}^{K\textnormal{th}}} \right)_{0N}(q,Q)    \\
        \end{pmatrix}
    \end{equation}
    is another fundamental solution of the same $q$-difference system, which is confluent. Therefore, there exists a $q$-constant matrix $B_{q,z} \in \textnormal{GL}_{N+1}\left( \mathcal{M}\left( \mathbb{E}_q \right) \right)$ replacing the fundamental solution $\widetilde{\mathcal{X}^{K\textnormal{th}}}(q,Q)$ by the confluent solution $\mathcal{Y}^{K \textnormal{th}}(z,q,Q)$ given by Equation (\ref{qkqde:eqn_proof_confluence_fn_final_result_of_gauge_transforms}).
    
    Furthermore, the fundamental solution $\mathcal{Y}^{K \textnormal{th}}(z,q,Q)$ verifies the condition of the Proposition on the coefficients $\left(
            \mathcal{X}^{K\textnormal{th}} \left(q,\varphi_{q,z}^{-1}(Q) \right)P_{q,z}
        \right)_{li}$, since we have
    \[
        \left(
            B_{q,z} \widetilde{\mathcal{X}^{K\textnormal{th}}}
        \right)_{0i}
        =
        \sum_{\substack{
                0 \leq a, b \leq N  \\  a + b = i
            }}
        \left( \frac{q-1}{z} \right)^a 
        \binom{\ell_{q} \left( Q \right)}{a}
        \left( \frac{1-q}{z} \right)^b J_b \left(
            q,\left( \frac{1-q}{z}\right)^{N+1} Q
        \right)
    \]
    Next, we compute the limit of the matrix $\mathcal{Y}^{K \textnormal{th}}(z,q,Q)$ when $q^t$ tends to $1$.
    For the first row of this matrix, we have
    \[
        \lim_{t \to 0}  \left(\mathcal{Y}^{K \textnormal{th}}(z,q,Q) \right)_{0i}
        =
        \sum_{\substack{
                0 \leq a, b \leq N  \\  a + b = i
            }}
        \frac{\log(Q)^a}{z^a a!}
        f_b(z,Q)
    \]
    Since $\lim_{q \to 1} \delta_q = Q \partial_Q$, the other rows have a finite limit, and the fundamental solution $\mathcal{Y}^{K \textnormal{th}}(z,q,Q)$ has a well defined limit.
    
    Therefore, the $q$-constant transformation given by $P_{q,z} := A_{q,z} B_{q,z}$ satisfies the conditions imposed by the proposition, and the matrix $\mathcal{X}^{K\textnormal{th}} \left(q,\varphi_{q,z}^{-1}(Q) \right)P_{q,z}$ is given by the matrix $\mathcal{Y}^{K \textnormal{th}}(z,q,Q)$ of Equation (\ref{qkqde:eqn_proof_confluence_fn_final_result_of_gauge_transforms}).
\end{proof}

\begin{remark}
    Just like in the Proposition \ref{qkqde:prop_jk_eqn_pullback}, we could have multiplied the functions $\binom{\ell_q \left( Q \right)}{\gamma}$ and $
    J_\gamma \left(
            q,\left( \frac{1-q}{z}\right)^{N+1} Q
        \right)
    $
    by $(1-q)^\gamma$ only. By doing so, the limit we get will be related to $\widetilde{J^\textnormal{coh}}(1,Q)$ only, cf. Remark \ref{qkqde:remark_formula_for_qpullback}.
\end{remark}

\begin{defin}
    We can put the first row of the fundamental solution $\left( \mathcal{X}^{K\textnormal{th}}\left(q^t,\varphi_{q,z}^{-1}(Q) \right) \right) P_{q^t,z}$ back into a $K$-theoretical function, which we will denote by $P_{q,z} \cdot \varphi_{q,z}^* \widetilde{J^{K\textnormal{th}}}$.
    Denoting by $a_{0i}, i \in \{ 0, \dots, N \}$ the elements of the first row of $\left( \mathcal{X}^{K\textnormal{th}}\left(q^t,\varphi_{q,z}^{-1}(Q) \right) \right)P_{q^t,z} $, we have
    \[
        P_{q,z} \cdot \varphi_{q,z}^* \widetilde{J^{K\textnormal{th}}}(z,q,Q)
        =
        \sum_{i=0}^N \left( 1 - P^{-1} \right)^i a_{0i}
    \]
\end{defin}
The function $P_{q,z} \cdot \varphi_{q,z}^* \widetilde{J^{K\textnormal{th}}}$ decomposes in our basis of $K \left( \mathbb{P}^N \right)$ as follows:
    \begin{equation}\label{qkqde:eqn_decomposition_of_gauge_pullback_JK}
        P_{q,z} \cdot \varphi_{q,z}^* \widetilde{J^{K\textnormal{th}}}(z,q,Q)
        =
        \sum_{i=0}^N
        \left( 1 - P^{-1} \right)^i
        \sum_{\substack{
                0 \leq a, b \leq N  \\  a + b = i
            }}
        \left( \frac{q-1}{z} \right)^a 
        \binom{\ell_{q} \left( Q \right)}{a}
        \left( \frac{1-q}{z} \right)^b J_b \left(
            q,\left( \frac{1-q}{z}\right)^{N+1} Q
        \right)
    \end{equation}

\begin{remark}\label{qkqde:remark_naturality_of_gauge_transform}
    The $q$-constant transformation $P_{q,z}$ of Proposition \ref{qkqde:prop_JK_confluence_sol} could be natural in the following sense: if a fundamental solution is not confluent, it makes sense to look for a transformation which changes a solution $\varphi_{q,z}^* \widetilde{J_i}(z,q,Q)$ into the confluent solution $\left( \frac{1-q}{z} \right)^{k_i} \varphi_{q,z}^* \widetilde{J_i}(z;q,Q)$ for some integer $k_i \in \mathbb{Z}$ (cf Remark \ref{qkqde:remark_formula_for_qpullback}). The transformation $P_{q,z}$ is the only change of fundamental solution of this form such that the limits when $q$ tends to $1$ of the solutions $\left( \frac{1-q}{z} \right)^{k_i} \varphi_{q,z}^* \widetilde{J_i}(z;q,Q)$ are well defined and non zero.
    
    Let us explain why.
    For the change of fundamental solution $P_{q,z}$, we chose $k_i=i$ in the proof of Proposition \ref{qkqde:prop_JK_confluence_sol}.
    If instead we have $k_i \neq i$, in Equation \ref{qkqde:eqn_decomposition_of_gauge_pullback_JK}, this amounts to changing the exponents of $\left( \frac{q^t-1}{z} \right)^a $ and $\left( \frac{1-q^t}{z} \right)^b$.
    If these exponents were to be changed, we see that the limit of the coefficient in front of $\left(1-P^{-1}\right)^i$ when $q$ tends to $1$ would be either trivial or not defined.
\end{remark}

\subsection{Comparison between the confluence and the cohomological $J$-function}\label{qkqde:subsection_comparison_confluence_conclusion}

This subsection consists of the comparison between the limit of our transformed fundamental solution and the cohomological $J$-function $\widetilde{J^\textnormal{coh}}(z,Q)$ using the Chern character.

\begin{defin}\label{qkqde:def_confluence_applied_to_JK}
    We denote by $\textnormal{confluence}\left(\widetilde{J^{K\textnormal{th}}}\right)$ the function defined by the limit
    \[
        \textnormal{confluence}\left(\widetilde{J^{K\textnormal{th}}}\right)(z,Q)
        =
        \lim_{t \to 0}
         P_{q^t,z} \cdot \varphi_{q^t,z}^* \widetilde{J^{K\textnormal{th}}}(z,q^t,Q)
    \]
    The existence of this limit is given by Proposition \ref{qkqde:prop_JK_confluence_sol}.
\end{defin}

\begin{prop}\label{qkqde:prop_confluence_JK_equals_JH}
    Consider the ring automorphism
        $ \gamma : K\left( \mathbb{P}^N \right) \otimes \mathbb{Q} \to H^*\left( \mathbb{P}^N; \mathbb{Q}\right)$ defined by $\gamma(1-P^{-1})=H$. The confluence of the $K$-theoretical $J$-function is related to the cohomological $J$-function by the identity
    \[
        \gamma \left( \textnormal{confluence}\left(\widetilde{J^{K\textnormal{th}}}\right)(z,Q) \right)
        =
        \widetilde{J^\textnormal{coh}}(z,Q)
    \]
\end{prop}

Before giving a proof of this Proposition, we recall that a decomposition of the function $P_{q,z} \cdot \varphi_{q,z}^* \widetilde{J^{K\textnormal{th}}}$ is given by Equation (\ref{qkqde:eqn_decomposition_of_gauge_pullback_JK}).
We are going to compare this decomposition with the one of the cohomological $J$-function.

\begin{proof}
    The decomposition of the cohomological $J$-function in the basis $(1,H,\dots,H^N)$ is given by
    \begin{equation}\label{qkqde:eqn_jh_decomposition}
        \widetilde{J^\textnormal{coh}}(z,Q) =
        \sum_{i=0}^N H^i
        \sum_{\substack{
            0 \leq a, b \leq N  \\  a + b = i
        }}
        \frac{1}{a !} \left( \frac{\log(Q)}{z} \right)^a
        g_b(z,Q)
    \end{equation}
    Where
    \[
        g_b(z,Q)
        =
        \sum_{d \geq 0} \frac{Q^d}{\left( z^d d!\right)^{N+1}}
        \left(
            \frac{1}{z^b}
                \sum_{k=0}^N \sum_{\substack{
            0 \leq j_1,\dots,j_N \leq N \\  j_1 + \cdots + j_N = k \\ j_1 + 2 j_2 + \cdots + N j_N = b
            }}
            (-1)^k \frac{(N+k)!}{N! j_1! \cdots j_N!}
                \prod_{l=1}^N \left(
                \sum_{1 \leq m_1 < \cdots < m_l \leq d} \frac{1}{m_1 \cdots m_l)}
            \right)^{j_l}
        \right)
    \]
    Let us recall Equation (\ref{qkqde:eqn_decomposition_of_gauge_pullback_JK}). We have
    \[
    P_{q,z} \cdot \varphi_{q,z}^* \widetilde{J^{K\textnormal{th}}}(z,q,Q)
        =
        \sum_{i=0}^N
        \left( 1 - P^{-1} \right)^i
        \sum_{\substack{
                0 \leq a, b \leq N  \\  a + b = i
            }}
        \left( \frac{q^t-1}{z} \right)^a 
        \binom{\ell_{q^t} \left( Q \right)}{a}
        \left( \frac{1-q^t}{z} \right)^b J_b \left(
            q^t,\left( \frac{1-q^t}{z}\right)^{N+1} Q
        \right)
    \]
    Where
    \begin{align*}
        &J_b \left(
            q,\left( \frac{1-q}{z}\right)^{N+1} Q
        \right)
        =   \\
        &\sum_{d \geq 0} \frac{(1-q)^{d(N+1)}Q^d}{z^{d(N+1)} (q;q)_d^{N+1}}
        \left(
            \sum_{k=0}^N \sum_{\substack{
            0 \leq j_1,\dots,j_N \leq N \\  j_1 + \cdots + j_N = k \\ j_1 + 2 j_2 + \cdots + N j_N = b
            }}
            (-1)^k \frac{(N+k)!}{N! j_1! \cdots j_N!}
            \prod_{l=1}^N \left(
                \sum_{1 \leq m_1 < \cdots < m_l \leq d} \frac{q^{m_1+\cdots+m_l}}{(1-q^{m_1}) \cdots (1-q^{m_l})}
            \right)^{j_l}
        \right)
    \end{align*}
    
    We observe that
    \[
        \lim_{t \to 0}
        \left( \frac{1-q^t}{z} \right)^b J_b \left(
            q^t,\left( \frac{1-q^t}{z}\right)^{N+1} Q
        \right)
        =
        g_b(z,Q)
    \]
    Using Proposition \ref{qde:prop_specfns_limits}, we also observe that
    \[
        \lim_{t \to 0}
        \left( \frac{q^t-1}{z} \right)^a 
        \binom{\ell_{q^t} \left( Q \right)}{a}
        =
        \frac{1}{a !} \left( \frac{\log(Q)}{z} \right)^a
    \]
    Putting these two observations together, we get
    \[
        \textnormal{confluence}\left(\widetilde{J^{K\textnormal{th}}}\right)(z,Q)
        =
        \sum_{i=0}^N \left( 1-P^{-1}\right)^i
        \sum_{\substack{
            0 \leq a, b \leq N  \\  a + b = i
        }}
        \frac{1}{a !} \left( \frac{\log(Q)}{z} \right)^a
        g_b(z,Q)
    \]
    Comparing with Equation (\ref{qkqde:eqn_jh_decomposition}), we find that
    \[
        \gamma \left(
            \textnormal{confluence}\left(\widetilde{J^{K\textnormal{th}}}\right)(z,Q)
        \right)
        =
        \widetilde{J^\textnormal{coh}}(z,Q)
    \]
\end{proof}
We can now give a complete proof of Theorem \ref{qkqde:thm_JK_sol_confluence_statement}.

\begin{proof}[Proof of Theorem \ref{qkqde:thm_JK_sol_confluence_statement}]
    \textbf{Confluence of the equation}. Using the $q$-pullback $\varphi_{q,z}$ of Proposition \ref{qkqde:prop_jk_eqn_pullback}, we obtain a confluent $q$-difference system. Its limit is the differential equation $(z Q \partial_Q)^{N+1} \widetilde{f}(Q) = \widetilde{f}(Q)$, which is the differential equation associated to the small cohomological $J$-function.
    The choice of the formula for the $q$-pullback $\varphi_{q,z}$ is discussed in Remark \ref{qkqde:remark_JK_qpullback_is_natural}.
    
    \textbf{Confluence of the solution}. By Proposition \ref{qkqde:prop_JK_is_a_fund_sol}, we can see the $K$-theoretical $J$-function as a fundamental solution matrix $\mathcal{X}^{K\textnormal{th}}$ of the starting $q$-difference equation (\ref{qkqde:eqn_JK}).
    By Proposition \ref{qkqde:prop_JK_confluence_sol}, there exists a change of fundamental solution $P_{q,z} \in \textnormal{GL}_{N+1}\left( \mathcal{M}\left( \mathbb{E}_q \right) \right)$ such that the fundamental solution $\varphi_{q,z}^*\mathcal{X}^{K\textnormal{th}} P_{q,z}$ is confluent.
    The choice of the formula for the $q$-constant transformation $P_q$ is discussed in Remark \ref{qkqde:remark_naturality_of_gauge_transform}.
    
    \textbf{Comparison with quantum cohomology}.
    The first row of the fundamental solution $\varphi_{q,z}^*\mathcal{X}^{K\textnormal{th}}P_{q,z}$ defines another $K$-theoretical function, which we denote by $P_{q,z} \cdot \varphi_{q,z}^* \widetilde{J^{K\textnormal{th}}}$.
    Since the fundamental solution was confluent, this function has a well defined limit when $q^t \to 1$, which we denote by $\textnormal{confluence}\left(\widetilde{J^{K\textnormal{th}}}\right)$.
    Using Proposition \ref{qkqde:prop_confluence_JK_equals_JH}, we have
    \[
        \gamma \left( \textnormal{confluence}\left(\widetilde{J^{K\textnormal{th}}}\right)(z,Q) \right)
        =
        \widetilde{J^\textnormal{coh}}(z,Q)
    \]
\end{proof}

\section{Equivariant version of the comparison theorem}

In this section we state and prove an analogue of Theorem \ref{qkqde:thm_JK_sol_confluence_statement} in the context of equivariant $K$-theory and cohomology.
The two main differences compared to the previous section are:
\begin{enumerate}[label=(\roman*)]
    \item The formulas have to be adapted to the equivariant setting. In particular, since we see the generating classes $q \in K_{\mathbb{C}^*}(\textnormal{pt}), z\in H^*_{\mathbb{C}^*}(\textnormal{pt})$ as parameters, we will also need to see the generators of $K_{\left(\mathbb{C}^*\right)^{N+1}}(\textnormal{pt})$ as parameters.
    
    \item The bases of equivariant $K$-theory and cohomology are not related to the non equivariant bases. This means that the fundamental solution we will consider in the equivariant setting are different.
\end{enumerate}



\subsection{Equivariant Gromov--Witten theory}

The goal of this subsection is to introduce the definitions of the $J$-functions in the equivariant setting and explain the differences mentioned above.

\subsubsection{Equivariant $K$-theory}

\begin{notation}
    We denote by $T^{N+1}$ the torus $(\mathbb{C}^*)^{N+1}$.
\end{notation}

For $(t_0,\dots,t_N) \in T^{N+1}$ and $[z_0 : \cdots : z_N] \in \mathbb{P}^N$, we define an action of the torus $T^{N+1}=(\mathbb{C}^*)^{N+1}$ on $\mathbb{P}^N$ by the formula, 
\[
    (t_0,\dots,t_N) \cdot [z_0 : \cdots : z_N] = [ t_0 z_0 : \cdots : t_N z_N ]
\]
The elementary representations
\[
    \functiondesc
    {\rho_i}
    {(\mathbb{C}^*)^{N+1}}
    {\mathbb{C}^*}
    {(t_0,\dots,t_N)}
    {t_i}
\]
defines $N+1$ classes $\Lambda_0,\dots,\Lambda_N \in K_{T^{N+1}}(\textnormal{pt})$, where $- \Lambda_i$ is the line bundle on the point with an action of the group $T^{N+1}$ given by $\rho_i$. In the end, we get

\begin{prop}[\cite{Chriss_Ginzburg_book_eq_ktheory}]
    The $T^{N+1}$-equivariant $K$-theory of the point is the ring given by
    \[
    K_{T^{N+1}}(\textnormal{pt}) = \mathbb{Z}[\Lambda_0^{\pm 1},\cdots, \Lambda_N^{\pm 1}]
    \]
\end{prop}


\begin{prop}[\cite{Chriss_Ginzburg_book_eq_ktheory}]
    Denote by $P=\mathcal{O}_\textnormal{eq}(1) \in K_{T^{N+1}}\left( \mathbb{P}^N \right)$ the anti-tautological equivariant bundle.
    The $T^{N+1}$-equivariant $K$-theory of $\mathbb{P}^N$ is the ring
    \[
        K_{T^{N+1}}\left( \mathbb{P}^N \right) = \mathbb{Z}[\Lambda_0^{\pm 1},\dots,\Lambda_N^{\pm 1}][P^{\pm 1}] \left/\left( \left(1-\Lambda_0 P^{-1} \right) \cdots \left(1-\Lambda_N P^{-1} \right) \right) \right.
    \]
\end{prop}

\begin{coro}
    The $T^{N+1}$-equivariant $K$-theory of $\mathbb{P}^N$ is generated as a $K_{T^{N+1}}(\textnormal{pt})$-module by the classes
    \[
        \Psi_i = \prod_{j \neq i} \frac{1-\Lambda_j P^{-1}}{1-\Lambda_j \Lambda_i^{-1}} \in K_{T^{N+1}}\left( \mathbb{P}^N \right)
    \]
\end{coro}
    Notice the generating classes $\Psi_i \in K_{T^{N+1}}\left( \mathbb{P}^N \right)$ behave like the polynomials of Lagrange interpolation: for all $j \in \{0, \dots, N\}$, we have
    \[
        {\Psi_i}_{|P=\Lambda_j} = \delta_{i,j}
    \]

\begin{defin}
    There is a morphism of rings $\rho : K_{T^{N+1}}\left( \mathbb{P}^N \right) \to K\left( \mathbb{P}^N \right)$ defined by $\rho(P)=P$ and for all $i \in \{0, \dots, N\}, \, \rho(\Lambda_i)=1$.
    Given an equivariant class $\phi \in K_{T^{N+1}}\left( \mathbb{P}^N \right)$, its image $\rho(\phi)$ will be called its \textit{non equivariant limit}.
\end{defin}

\begin{remark}\label{QKQDE:rmk_equivariant_basis_nonequivariant_limit}
    The non equivariant limit of the basis $(\Psi_i)$ of $K_{T^{N+1}} \left( \mathbb{P}^N \right)$ is not a basis of $K\left( \mathbb{P}^N \right)$.
    For any $i \in \{ 0, \dots, N\}$, we have $\rho(\Psi_i) = \left( 1 - P^{-1} \right)^N$
\end{remark}

\subsubsection{Equivariant cohomology}

Introduce the classes
\begin{align*}
    \lambda_i = c_1 (\Lambda_i) \in H^2_{T^{N+1}}(\textnormal{pt}) = H^2(BT^{N+1})
    && z= -c_1(q) \in H^2_{T}(\textnormal{pt})    
\end{align*}

\begin{prop}[\cite{CK_book}, Section 9.1]
    We have the following equivariant cohomology rings
    \begin{align*}
        &H_{T^{N+1}}^*(\textnormal{pt}; \mathbb{Q}) = \mathbb{Q}[\lambda_0,\dots,\lambda_N]
        && H_{T}^*(\textnormal{pt};\mathbb{Q}) = \mathbb{Q}[z]
    \end{align*}
    Denote by $H=c_1( \mathcal{O}_\textnormal{eq}(1))$ the equivariant hyperplane class. We have
    \[
        H_{T^{N+1}}^*\left(\mathbb{P}^N ; \mathbb{Q} \right) = 
        \mathbb{Q}[\lambda_0,\dots,\lambda_N][H] \left/ \left( (H-\lambda_0) \cdots (H - \lambda_N) \right)  \right.
    \]
\end{prop}

\begin{coro}[\cite{CK_book}, Equation 9.4]
    The set $H_{T^{N+1}}^*\left(\mathbb{P}^N ; \mathbb{Q} \right)$ is generated as a $H_{T^{N+1}}^*(\textnormal{pt};\mathbb{Q})$-module by the classes
    \[
        Lag_i
        =
        \prod_{j \neq i}
        \frac{H-\lambda_j}{\lambda_i-\lambda_j}
        \in
        H_{T^{N+1}}^*\left(\mathbb{P}^N ; \mathbb{Q} \right)
    \]
\end{coro}

\begin{defin}
    There is a morphism of rings $\rho : H^*_{T^{N+1}}\left( \mathbb{P}^N ; \mathbb{Q} \right) \to H^*\left( \mathbb{P}^N , \mathbb{Q} \right)$ defined by $\rho(H)=H$ and for all $i \in \{0, \dots, N\}, \, \rho(\Lambda_i)=0$.
    Given an equivariant class $\phi \in K_{T^{N+1}}\left( \mathbb{P}^N \right)$, its image $\rho(\phi)$ will be called its \textit{non equivariant limit}.
\end{defin}

\subsubsection{Comparison between the equivariant classes}

Let us explain how to the actions by $T$ and $T^{N+1}$ are related. Our goal is to motivate a formula we will use later when dealing with $q$-difference equations.
The multiplicative group morphism
\[
    \functiondesc{f}{T^{N+1}}{T}{(z_0, \dots, z_N)}{z_0 \cdots z_N}
\]
Induces a ring morphism $f_{K\textnormal{th}} : K_{T}(\textnormal{pt}) \to K_{T^{N+1}}(\textnormal{pt})$ as follows: if $E$ denotes a complex vector space
and $\chi : T \times E \to E$ is a $T$-action,
then we can define a $\left(\mathbb{C}^*\right)^{N+1}$-action on $E$ 
by
\[
    (t_0, \dots, t_N) \cdot v = \chi(f(t_0, \dots, t_N),v) = \chi(t_0 \cdots t_N, v)
\]
In particular, by checking the group actions, recalling that $q$ is a generator of $K_T( \textnormal{pt})$, we have $f_\textnormal{Kth}(-q)=\Lambda_0 \otimes \cdots \otimes \Lambda_N$.
This map has $N+1$ sections $(s_i)$, with the section $s_i$ being given by the inclusion of $T$ in the $i^\textnormal{th}$ factor of $T^{N+1}$.
We have $s_j(\Lambda_i)=\delta_{i,j} q$.
The map $f_\textnormal{Kth}$ has a cohomological analogue $g_\textnormal{coh}$ constructed in the same way.
This data fit in the diagram
    \[
    \begin{tikzcd}
        \mathbb{Q}[q^{\pm 1}] = K_{T}(\textnormal{pt}) \otimes_\mathbb{Z} \mathbb{Q}
            \arrow[r, "f_{K\textnormal{th}}"]
            \arrow[d, "\textnormal{ch}"]
        &
        K_{T^{N+1}}(\textnormal{pt}) \otimes_\mathbb{Z} \mathbb{Q}
            \arrow[d, "\textnormal{ch}"]
        \\
        \mathbb{Q}[z]=H^*_{T}(\textnormal{pt}; \mathbb{Q})
            \arrow[r, "g_\textnormal{coh}"]
        &
        H^*_{T^{N+1}}(\textnormal{pt}; \mathbb{Q})
    \end{tikzcd}
    \]

\begin{prop}
    Up to degree 1 terms, we have the relation in $H_{T^{N+1}}^*( \textnormal{pt}; \mathbb{Q} )$
    \begin{equation}\label{qkqde:eqn_relation_between_equivariant_classes_of_point}
        \textnormal{ch}(\Lambda_i) = \textnormal{ch}(f_K(q))^{-\frac{\lambda_i}{g_H(z)}}
    \end{equation}
\end{prop}

\begin{proof}
    Since all the $K$-theoretical classes are line bundles, we have up to degree one terms
    \[
        \textnormal{ch}(f_K(q))^{-\frac{\lambda_i}{g_H(z)}}
        =
        1+\frac{\lambda_i}{\lambda_0 + \cdots + \lambda_N} c_1(\Lambda_0 \otimes \cdots \otimes \Lambda_N) + \cdots
        =
        1+\lambda_i + \cdots
        =
        \textnormal{ch}(\lambda_i)
    \]
\end{proof}


\subsubsection{Givental's equivariant $K$-theoretical $J$-function}

\begin{defin}[\cite{Giv_PermEquiv}, II]\label{qkqde:def_J_fn_eq}
    Let $X = \mathbb{P}^N$ and let $P=\mathcal{O}_\textnormal{eq}(1) \in K_{T^{N+1}}\left( \mathbb{P}^N \right)$ be the anti-tautological equivariant bundle. Givental's equivariant small $J$-function is the function
    \begin{equation}\label{qkqde:eqn_J_def_eq}
        J^{K\textnormal{th},\textnormal{eq}}(q,Q) 
        =
        P^{-\ell_q(Q)}
        \sum_{d \geq 0} \frac{Q^d}{\left(q\Lambda_0 P^{-1}, \dots,q\Lambda_N P^{-1} ;q \right)_d}
    \end{equation}
    Where
    \[
        \left(q\Lambda_0 P^{-1}, \dots,q\Lambda_N P^{-1} ;q \right)_d
        =
        \prod_{i=0}^N
        \left(q\Lambda_i P^{-1} ;q \right)_d
    \]
\end{defin}

\begin{prop}
    The small equivariant $J$-function (\ref{qkqde:eqn_J_def_eq}) is solution of the $q$-difference equation
    \begin{equation}\label{qkqde:eqn_JK_eq}
        \left[ \left( 1 - \Lambda_0 \qdeop{Q} \right) \cdots \left( 1 - \Lambda_N \qdeop{Q} \right) - Q \right]  J^{K\textnormal{th}}(q,Q) = 0
    \end{equation}
\end{prop}
The proof of this proposition is exactly the same as in Proposition \ref{qkqde:prop_JK_sol_qde_pre}, where we use instead the relation $\left(1-\Lambda_0 P^{-1} \right) \cdots \left(1-\Lambda_N P^{-1} \right)=0 \in K_{T^{N+1}} \left( \mathbb{P}^N \right)$.


\begin{remark}
    The non equivariant limit of the equivariant $J$-function $ J^{K\textnormal{th},\textnormal{eq}}$ is the usual $J$-function $\widetilde{J^{K\textnormal{th}}}$.
\end{remark}


Recall that the equivariant $K$-theory of projective spaces is generated by the classes 
\[
\Psi_i = \prod_{j \neq i} \frac{1-\Lambda_j P^{-1}}{1-\Lambda_j \Lambda_i^{-1}} \in K_{T^{N+1}}\left( \mathbb{P}^N \right)
\]
\begin{prop}[\cite{Giv_PermEquiv}, II]
    The decomposition of the $J$-function in this basis is given by
    \[
        J^{K\textnormal{th},\textnormal{eq}}(q,Q)
        =
        \sum_{i=0}^N J^{K\textnormal{th},\textnormal{eq}}_{|P=\Lambda_i}(q,Q) \Psi_i
    \]
\end{prop}

\begin{remark}\label{qkqde:rmk_jk_decomposition_and_nonequivariant_limit}
    The equivariant limit of this decomposition does not have a sense, since the basis $(\Psi_i)$ is not compatible with non equivariant limit.
    We have
    \[
        \rho \left( J^{K\textnormal{th},\textnormal{eq}}_{|P=\Lambda_i}(q,Q) \right)
        =
        \sum_{d \geq 0}
        \frac{Q^d}{(q;q)_d^{N+1}}
        =
        \widetilde{J_0}(q,Q)
    \]
    Where we recall that $\widetilde{J_0}(q,Q)$ is the coefficient in the decomposition
    \[
        \widetilde{J^{K\textnormal{th}}}(q,Q)
        =
        \sum_{i=0}^N \widetilde{J_i}(q,Q) \left(1-P^{-1} \right)^i
        \in
        K \left( \mathbb{P}^N \right) \otimes \mathbb{C}(q) [\![Q]\!]
    \]
\end{remark}


\subsubsection{Givental's equivariant $J$-function as a fundamental solution}

Let us consider the $q$-difference equation (\ref{qkqde:eqn_JK_eq})
\[
    \left[ \left( 1 - \Lambda_0 \qdeop{Q} \right) \cdots \left( 1 - \Lambda_N \qdeop{Q} \right) - Q \right]  J^{K\textnormal{th,eq}}(q,Q) = 0
\]
If we assume $\Lambda_0, \dots, \Lambda_N \in \mathbb{C}^*$ such that if $i \neq j, \Lambda_i \Lambda_j^{-1} \notin q^\mathbb{Z}$, then this $q$-difference equation is a hypergeometric $q$-difference equation which is regular singular and non resonant, see Definition \ref{qde:def_qhg_eqn} (here, non resonance is equivalent to the condition $\Lambda_i \Lambda_j^{-1} \notin q^\mathbb{Z}$).
Notice that, from the previous section, the non equivariant $q$-difference equation $
    \left[(zQ \partial_Q)^{N+1} - Q\right]
    \widetilde{J^\textnormal{coh}}(z,Q) = 0
$ is resonant.

\begin{prop}\label{prop:qkqde_JK_eq_is_a_fundamental_solution}
    The functions $\left( J^{K\textnormal{th},\textnormal{eq}}_{|P=\Lambda_i} \right)_{i \in \{0,\dots,N \}}$ form a $\mathcal{M}\left( \mathbb{E}_q \right)$-basis of the solutions of the $q$-difference  equation (\ref{qkqde:eqn_JK_eq}).
\end{prop}

\begin{proof}
    Since the $q$-difference equation (\ref{qkqde:eqn_JK_eq}) has no $K$-theoretical coefficient, the functions $J^{K\textnormal{th},\textnormal{eq}}_{|P=\Lambda_i}$ are solutions.
    The condition $\Lambda_i \Lambda_j^{-1} \notin q^\mathbb{Z}$ implies that the $q$-logarithms $\Lambda_i^{\ell_q(Q)}$ and $\Lambda_j^{\ell_q(Q)}$ are independent over $\mathcal{M}(\mathbb{E}_q)$.
    Using Proposition \ref{qde:prop_qchars_with_no_essential_sing}, if $i \neq j$, then $J^{K\textnormal{th},\textnormal{eq}}_{|P=\Lambda_i}$ and $J^{K\textnormal{th},\textnormal{eq}}_{|P=\Lambda_j}$ are $\mathcal{M}(\mathbb{E}_q)$-linearly independent.
\end{proof}

\begin{remark}
    Given a $q$-hypergeometric equation that is regular singular at $Q=0$ and non resonant, the theory of $q$-difference equation gives a fundamental solution.
    This fundamental solution depends on the choice of a $q$-character.
    For example, choosing the $q$-characters $\Lambda_i^{-\ell_q(Q)}$ in \cite{Adams_qde} gives a basis of solutions that is the same as $\left( J^{K\textnormal{th},\textnormal{eq}}_{|P=\Lambda_i} \right)_{i \in \{0,\dots,N \}}$. 
    We will explain how this basis is constructed in the next subsubsection, see Proposition \ref{qde:prop_qhg_eqn_fundsols}.
\end{remark}

\subsubsection{$q$-hypergeometric equations}

In this subsubsection, we give the definitions and the fundamental solutions for $q$-hypergeometric equations.

\begin{defin}\label{qde:def_qhg_eqn}
    Let $r,s \in \mathbb{Z}_{\geq 0}, a_1, \dots, a_r, b_1, \dots b_s \in \mathbb{C}^*$. The associated \textit{$q$-hypergeometric equation} is
    \begin{equation}\label{qde:eqn_qhypergeometric_equation}
        \left[
            Q \left( - \qdeop{Q} \right)^{1+s-r} 
            \prod_{i=1}^r \left( 1-a_i \qdeop{Q} \right)
            - (1 - \qdeop{Q})
            \prod_{j=1}^s \left( 1-\frac{b_j}{q} \qdeop{Q} \right)
        \right] f_q(Q) = 0
    \end{equation}
    The $q$-difference equation is said to be \textit{non resonant} if any ratio of two coefficients $a_1, \dots, a_r, b_1, \dots b_s$ is not in $q^\mathbb{Z}$.
\end{defin}

\begin{prop}
    This $q$-difference equation is singular regular at $0$ and $\infty$ if and only if $r+1=s$.
\end{prop}

\begin{proof}
    This is done by applying Proposition \ref{qde:prop_when_a_qde_equation_is_regular_singular} to the Equation (\ref{qde:eqn_qhypergeometric_equation}).
    Notice that the only coefficient depending on $Q$ in this $q$-difference is $Q \left( - \qdeop{Q} \right)^{1+s-r}$.
    If $r>s+1$ (resp. $r<s+1$), this $q$-difference equation is singular irregular at $Q=0$ (resp. $Q=\infty$). 
\end{proof}

By looking for a solution under the form of a Laurent series, we find the following definition.

\begin{defin}
    Let $a_1,...,a_r,b_1,...,b_s \in \mathbb{C}^*$. The associated \textit{$q$-hypergeometric series} is the formal Taylor series
    \begin{equation}\label{qde:eqn_qhypergeometric_series}
        \tensor[_{r}]{\varphi}{_s} \left(
            \begin{gathered}
                a_1 \quad \cdots \quad a_r \\
                b_1 \quad \cdots \quad b_s
            \end{gathered}
            \middle| q, Q
        \right)
        =
        \sum_{d \geq 0}
        \frac{\left(a_1,\dots,a_r ;q\right)_d}{\left(q,b_1,\dots,b_s ;q\right)_d}
        \left(
            (-1)^d q^{\frac{d(d-1)}{2}}
        \right)^{1+r-s}
        Q^d
    \end{equation}
    Where $\left(a_1,\dots,a_r ;q\right)_d = \left(a_1;q\right)_d \cdots \left(a_r;q\right)_d = \prod_{r \geq 0} \left(1-q^ra_1\right) \cdots \left(1-q^ra_r\right)$.
    If $b_1,\dots,b_s \notin q^{\mathbb{Z}_{\leq 0}}$, this defines an element of $\mathbb{C}[\![Q]\!]$.
\end{defin}

\begin{prop}
    The $q$-hypergeometric series (\ref{qde:eqn_qhypergeometric_series}) is a solution of the $q$-hypergeometric equation (\ref{qde:eqn_qhypergeometric_equation}).
    Moreover, assuming $|q|<1$, if $r<s+1$ (resp. $r>s+1$), this formal power series has convergent ray $\infty$ (resp. $0$).
\end{prop}

\begin{proof}
    We look for a solution of the $q$-hypergeometric equation (\ref{qde:eqn_qhypergeometric_equation}) of the form
    \[
        f_q(Q) = \sum_{d \in \mathds{Z}} a_d(q) Q^d
    \]
    Where $(a_d(q))_{d \in \mathbb{Z}} \in \mathbb{C}^\mathbb{Z}$.
    Assuming that $f_q$ is a solution of (\ref{qde:eqn_qhypergeometric_equation}), we obtain that
    \[
        f_{d+1}
        =
        \left(-q \right)^{1+s-r}
        \frac{(1-a_0 q^d) \cdots (1-a_r q^d)}{(1-q^{d+1})(1-b_0 q^d) \cdots (1-b_s q^d)}
    \]
    From this recursion relation, by setting the initial condition $f_0=1$, we deduce that the function $f_q$ corresponds to the $q$-hypergeometric series (\ref{qde:eqn_qhypergeometric_series}).
    In particular, the factor $(1-\qdeop{Q})$ inside the $q$-difference equation (\ref{qde:eqn_qhypergeometric_equation}) implies that $a_{-1}(q)=0$.
    
    Next, we compute the convergence ray of the Taylor series (\ref{qde:eqn_qhypergeometric_series}).
    We have
    \[
        \left| \frac{a_{d+1}}{a_d} \right|
        ~_{d \to \infty} 
        \left( q^d \right)^{1+r-s}
    \]
    We conclude by ratio test.
\end{proof}

\begin{prop}[\cite{Adams_qde}]\label{qde:prop_qhg_eqn_fundsols}
    Assume in this proposition that $r=s+1$. Then, the $q$-difference equation \ref{qde:eqn_qhypergeometric_equation} is singular regular at $Q=0,\infty$ of degree $r$. Assume furthermore that if $i \neq j, k \neq l$, then $a_i/a_{j}, b_k / b_{l}, q/b_1 \notin q^\mathbb{Z}$. A basis of independent solutions at $Q=0$ is given by taking the convergent power series (\ref{qde:eqn_qhypergeometric_series}) as a first solution
    and adding the functions indexed by $j \in \{ 1,\dots, s\}$
    \[
        y_j^0(Q)=
        \frac{\theta_q(-b_j Q/q)}{\theta_q(-Q)} \tensor[_{s+1}]{\varphi}{_s} \left(
            \begin{gathered}
                \frac{q a_1}{b_j} \quad \cdots \quad \frac{q a_{r}}{b_j} \\
                \frac{q^2}{b_j} \quad \frac{q b_1}{b_j} \quad \cdots \quad \widehat{\frac{q b_j}{b_j}} \quad \cdots \quad \frac{q b_s}{b_j}
            \end{gathered}
            \middle| q, Q
        \right)
    \]
    Where $\widehat{\cdot}$ means that this element is omitted.
    
    A basis of independent solutions at $Q=\infty$ is given by the solutions indexed by $i \in \{ 1,\dots, r\}$
    \[
        y_i^\infty(Q)=
        \frac{\theta_q(-a_i Q)}{\theta_q(-Q)} \tensor[_{s+1}]{\varphi}{_s} \left(
            \begin{gathered}
                a_i \quad \frac{a_i q}{b_1} \quad \cdots \quad \frac{a_i q}{b_s} \\
                \frac{a_i q}{a_1} \quad \cdots \quad \widehat{\frac{a_i q}{a_i}} \quad \cdots \quad \frac{a_i q}{a_r}
            \end{gathered}
            \middle| q, \frac{q b_1 \cdots b_s}{a_1 \cdots a_r Q}
        \right)
    \]
\end{prop}

\begin{proof}
    Our starting function was
    \[
        \tensor[_{r}]{\varphi}{_{r-1}} \left(
            \begin{gathered}
                a_1 \quad \cdots \quad a_r \\
                b_1 \quad \cdots \quad b_{r-1}
            \end{gathered}
            \middle| q, Q
        \right)
    \]
    It is obtained by looking a Taylor series solution $\sum_{d \geq 0} f_k(q) Q^k$ to the equation \ref{qde:eqn_qhypergeometric_equation} and setting the initial condition $f(0)=1$.
    
    To obtain the other solutions, we observe that if $b \in \mathbb{C}^* - q^\mathbb{Z}$ and $g_b$ is a solution of the $q$-difference equation $\left( 1 - \frac{\lambda}{b} \qdeop{Q} \right) g_b = 0$, then the function $f(Q) = e_{q,b^{-1}}(Q) g_b(Q)$ is a solution of the $q$-difference equation $\left( 1 - \lambda \qdeop{Q} \right) f = 0$.
    
    The independence of the families over $\mathcal{M}\left(\mathbb{E}_q\right)$ is a consequence of the hypothesis $a_i/a_{i+1}, b_j / b_{j+1}, q/b_1 \notin q^\mathbb{Z}$.
\end{proof}

\subsection{Statement of the equivariant comparison theorem}

We are going to relate Givental's equivariant $K$-theoretical $J$-function to its cohomological analogue by confluence.
Before stating the theorem, we define the equivariant cohomological $J$-function.

\begin{defin}[\cite{Givental_EquivariantGW}]
    The small equivariant cohomological $J$-function is given by the expression
    \begin{equation*}
        J^{\textnormal{coh},\textnormal{eq}}(z,Q) = Q^{\frac{H}{z}} \sum_{d \geq 0} \frac{Q^d}{\prod_{r=1}^d \left( H - \lambda_0 + rz\right) \cdots \left( H - \lambda_N + rz\right)}
        \in
        H_{T^{N+1}}^*\left( \mathbb{P}^N \right) \otimes \mathbb{C}[\![z,z^{-1}]\!]
    \end{equation*}
\end{defin}

\begin{prop}
    The function $J^{\textnormal{coh},\textnormal{eq}}$ is a solution of the differential equation
    \begin{equation}\label{qkqde:eqn_pde_jh_eq}
        \left[(-\lambda_0 + z Q \partial_Q) \cdots (- \lambda_N + z Q \partial_Q) - Q\right]
        J^{\textnormal{coh},\textnormal{eq}}(z,Q) = 0
    \end{equation}
\end{prop}

Recall that the equivariant cohomology of projective spaces has a basis given by
\[
    Lag_i
    =
    \prod_{j \neq i}
    \frac{H-\lambda_j}{\lambda_i-\lambda_j}
    \in
    H_{T^{N+1}}^*\left(\mathbb{P}^N ; \mathbb{Q} \right)
\]

\begin{prop}[\cite{CK_book}, Proposition 9.1.2]
    The function $J^{\textnormal{coh},\textnormal{eq}}$ is decomposed in the basis $(Lag_i)$ by the formula
    \begin{equation}\label{qkqde:eqn_decomposition_jh_eq}
        J^{\textnormal{coh},\textnormal{eq}}(z,Q)
        =
        \sum_{i=0}^N
        J^{\textnormal{coh},\textnormal{eq}}_{|H=\lambda_i}(z,Q)
        Lag_i
    \end{equation}
\end{prop}

For this the comparison theorem to make sense, we need to see the classes $q,z,\lambda_i,\Lambda_i$ as complex numbers.
Because of the Equation \ref{qkqde:eqn_relation_between_equivariant_classes_of_point}, we assume that if $i \neq j$, then $\lambda_i - \lambda_j \neq \mathbb{Z}$ and ask the parameters $\Lambda_i$ to satisfy
\begin{equation}\label{qkqde:eqn_relation_between_classes_as_numbers}
    \Lambda_i = q^{-\frac{\lambda_i}{z}}
\end{equation}
This relation will be necessary for confluence to work correctly.

\begin{thm}\label{qkqde:thm_confluence_jk_eq}
    Denote by $T^{N+1}$ the torus $\left( \mathbb{C}^* \right)^{N+1}$ and let $J^{K\textnormal{th},\textnormal{eq}}$ (resp. $J^{\textnormal{coh},\textnormal{eq}}$) be the small equivariant $K$-theoretical (resp. cohomological) $J$-function of $\mathbb{P}^N$. 
    \begin{enumerate}[label=(\roman*)]
        \item There is a $q$-pullback making the $q$-difference system (\ref{qkqde:eqn_JK_eq}) satisfied by $J^{K\textnormal{th},\textnormal{eq}}$ confluent. Its confluence yields the differential system satisfied by $J^{\textnormal{coh},\textnormal{eq}}$, (\ref{qkqde:eqn_pde_jh_eq}).
        \item  Consider the isomorphism of rings $\gamma_\textnormal{eq} : K_{T^{N+1}}\left( \mathbb{P}^N \right) \to H_{T^{N+1}}^* \left( \mathbb{P}^N ;\mathbb{Q} \right)$ given by $\gamma_\textnormal{eq}(\Psi_i)=Lag_i$ for all $i \in \{0, \dots, N\}$ and
        let $\textnormal{confluence}\left(J^{K\textnormal{th}}\right)$ be the result of confluence applied to the solution $J^{K\textnormal{th}}$ of the above $q$-difference system. Then, we have
        \[
            \gamma_\textnormal{eq}\left(\textnormal{confluence}\left(J^{K\textnormal{th},\textnormal{eq}}\right)\right)(z,Q) = J^{\textnormal{coh},\textnormal{eq}}(z,Q)
        \]
    \end{enumerate}
\end{thm}

For the formal definition of the function $\textnormal{confluence}\left(J^{K\textnormal{th},\textnormal{eq}}\right)$, see Definition \ref{qkqde:def_confluence_jk_equivariant}.
The proof of this theorem will again be in three steps:
\begin{enumerate}[label=(\roman*)]
    \item Confluence of the $q$-difference equation (see Subsection \ref{qkqde:subsection_comparison_confluence_equation_equivariant} and Proposition \ref{qkqde:prop_confluence_equation_JK_equivariant}).
    \item Confluence of the fundamental solution (see Subsection \ref{qkqde:subsection_comparison_confluence_solution_equivariant} and Proposition \ref{qkqde:prop_JK_confluence_sol})
    \item Comparison between the confluence of the solution and the cohomological $J$-function (see Subsection \ref{qkqde:subsection_comparison_confluence_conclusion_equivariant} and Proposition \ref{qkqde:prop_confluence_JK_equals_JH_equivariant}).
\end{enumerate}

\subsection{Confluence of the equation, equivariant version}\label{qkqde:subsection_comparison_confluence_equation_equivariant}

Before giving a statement on the confluence of the $q$-difference equation (\ref{qkqde:eqn_JK_eq}), let us write it in its matrix form.
We look for a matrix $B_q \in M_{N+1}(\mathbb{C}(Q))$ such that
    \[
        \delta_q
        \begin{pmatrix}
            J^{K\textnormal{th},\textnormal{eq}}(q,Q)            \\
            \delta_q J^{K\textnormal{th},\textnormal{eq}}(q,Q)   \\
            \vdots              \\
            \left( \delta_q \right)^N J^{K\textnormal{th},\textnormal{eq}}(q,Q)
        \end{pmatrix}
        =
        B_q(Q)
        \begin{pmatrix}
            J^{K\textnormal{th},\textnormal{eq}}(q,Q)            \\
            \delta_q J^{K\textnormal{th},\textnormal{eq}}(q,Q)   \\
            \vdots              \\
            \left( \delta_q \right)^N J^{K\textnormal{th},\textnormal{eq}}(q,Q)
        \end{pmatrix}
    \]
    We rewrite the $q$-difference equation (\ref{qkqde:eqn_JK_eq}) as
    \[
        \left[
            \left(
                1-\Lambda_0 + (1-q)\Lambda_0 \delta_q
            \right)
            \cdots
            \left(
                1-\Lambda_N + (1-q)\Lambda_N \delta_q
            \right)
            - Q
        \right]
        J^{K\textnormal{th},\textnormal{eq}}(q,Q) = 0
    \]
    We see the expression in the square brackets as a polynomial in $\delta_q$. For $i \in \{0, \dots, N+1\}$, the coefficient in front of $\delta_q^i$ in this polynomial, which we denote by $p_i$, is given by
    \[
        p_i :=
        -\delta_{i,0}Q
        +
        (1-q)^i \sum_{0 \leq j_1 < \cdots < j_i \leq N}
        \Lambda_{j_1} \cdots \Lambda_{j_i}
        \prod_{k \in \{0,\dots,N\}-\{j_1, \dots, j_i\}} (1-\Lambda_k)
    \]
    Where $\delta_{i,0}$ is the Kronecker symbol which is zero unless $i=0$. In particular, we have
    \begin{align*}
        p_0 &= -Q + \prod_{i=0}^N \left( 1 - \Lambda_i \right)    \\
        p_{N+1} &= (1-q)^{N+1} \Lambda_0 \cdots \Lambda_N
    \end{align*}
    So the $q$-difference equation (\ref{qkqde:eqn_JK_eq}) has the matrix form
    \begin{equation}\label{qkqde:eqn_jk_eq_as_system}
        \qdeop{Q}
        X
        =
        \begin{pmatrix}
            0       &   1       &   0         &   0       \\
            0       &   0       &   \ddots     &   0       \\
            \vdots  &   \vdots  &           &   \vdots  \\
            0       &   0       &   \cdots    &   1       \\
            \frac{p_0}{(1-q)^{N+1} \Lambda_0 \cdots \Lambda_N} &   \frac{p_1}{(1-q)^{N+1} \Lambda_0 \cdots \Lambda_N} &   \cdots  &   \frac{p_N}{(1-q)^{N+1} \Lambda_0 \cdots \Lambda_N}
        \end{pmatrix}
        X
    \end{equation}
    Where
    \[
        X=
        \begin{pmatrix}
            J^{K\textnormal{th},\textnormal{eq}}(q,Q)            \\
            \delta_q J^{K\textnormal{th},\textnormal{eq}}(q,Q)   \\
            \vdots              \\
            \left( \delta_q \right)^N J^{K\textnormal{th},\textnormal{eq}}(q,Q)
        \end{pmatrix}
    \]
    And
    \[
        p_i =
        -\delta_{i,0}Q
        +
        (1-q)^i \sum_{0 \leq j_1 < \cdots < j_i \leq N}
        \Lambda_{j_1} \cdots \Lambda_{j_i}
        \prod_{k \in \{0,\dots,N\}-\{j_1, \dots, j_i\}} (1-\Lambda_k)
    \]

\begin{prop}\label{qkqde:prop_confluence_equation_JK_equivariant}
    Consider the $q$-difference equation (\ref{qkqde:eqn_JK_eq})
    \[
        \left[ \left( 1 - \Lambda_0 \qdeop{Q} \right) \cdots \left( 1 - \Lambda_N \qdeop{Q} \right) - Q \right]  J^{K\textnormal{th}}(q,Q) = 0
    \]
    Let $\varphi_{q,z}$ be the function
    \[
        \functiondesc{\varphi_{q,z}}{\mathbb{C}}{\mathbb{C}}{Q}{\left( \frac{z}{1-q} \right)^{N+1} Q}
    \]
    Then, the $q$-pullback of the $q$-difference equation (\ref{qkqde:eqn_JK_eq}) by $\varphi_{q,z}$ is confluent, and its formal limit is the differential system satisfied by the small equivariant cohomological $J$-function (\ref{qkqde:eqn_pde_jh_eq}). 
 \end{prop}

\begin{proof}
    We consider the $q$-difference system (\ref{qkqde:eqn_jk_eq_as_system}). We recall that its coefficients on the last line are given by, for $i \in \{0, \dots, N\}$
    \[
        \frac{p_i}{p_{N+1}}
        =
        \frac{-\delta_{i,0}Q
        +
        (1-q)^i \sum_{0 \leq j_1 < \cdots < j_i \leq N}
        \Lambda_{j_1} \cdots \Lambda_{j_i}
        \prod_{k \in \{0,\dots,N\}-\{j_1, \dots, j_i\}} (1-\Lambda_k)}
        {(1-q)^{N+1} \Lambda_0 \cdots \Lambda_N}
    \]
    For $i \neq 0$, we have
    \[
        \frac{p_i}{p_{N+1}}
        =
        \frac{1}{\Lambda_0 \cdots \Lambda_N}
        \sum_{0 \leq j_1 < \cdots < j_i \leq N}
        \Lambda_{j_1} \cdots \Lambda_{j_i}
        \prod_{k \in \{0,\dots,N\}-\{j_1, \dots, j_i\}} \frac{1-\Lambda_k}{1-q}
    \]
    Using Equation (\ref{qkqde:eqn_relation_between_classes_as_numbers}) : $\Lambda_i=q^{- \lambda_i / z}$, we have
    \[
        \lim_{q \to 1} \frac{p_i}{p_{N+1}} 
        = 
        \frac{1}{(-z)^{N+1-i}}
        \sum_{0 \leq j_1 < \cdots < j_{N+1-i} \leq N} \lambda_{j_1} \cdots \lambda_{j_{N+1-i}}
    \]
    The remaining coefficient is
    \[
        \frac{p_0}{p_{N+1}} 
        = 
        \frac{1}{\Lambda_0 \cdots \Lambda_N} 
        \left(
            - \frac{Q}{(q-1)^{N+1}} +
            \prod_{j=0}^N \frac{1-\Lambda_j}{1-q}
        \right)
    \]
    Which has no limit as it stands.
    
    The $q$-pullback of this system only change the only coefficient of the matrix which depends on $Q$. The $q$-pullback is given by the system
    \[
        \qdeop{Q}
        X
        =
        \begin{pmatrix}
            0       &   1       &   0         &   0       \\
            0       &   0       &   \ddots     &   0       \\
            \vdots  &   \vdots  &           &   \vdots  \\
            0       &   0       &   \cdots    &   1       \\
            a_{N0}(z,q,Q) &   \frac{p_1}{(1-q)^{N+1} \Lambda_0 \cdots \Lambda_N} &   \cdots  &   \frac{p_N}{(1-q)^{N+1} \Lambda_0 \cdots \Lambda_N}
        \end{pmatrix}
        X
    \]
    Where
    \[
        a_{N0}(z,q,Q)
        =
        \frac{1}{\Lambda_0 \cdots \Lambda_N} 
        \left(
            - \frac{Q}{z^{N+1}} +
            \prod_{j=0}^N \frac{1-\Lambda_j}{1-q}
        \right)
    \]
    We have
    \[
        \lim_{q \to 1} a_{N0}(z,q,Q)
        = \frac{1}{z^{N+1}} \left( (-1)^{N+1}\lambda_0 \cdots \lambda_N - Q\right)
    \]
    Which is well defined.
    Therefore, the $q$-pullback by $\varphi_{q,z}$ of the system $(\ref{qkqde:eqn_jk_eq_as_system})$ is confluent.
    Moreover, its formal limit is the differential system associated to the differential equation
    \[
        \left[(-\lambda_0 + z Q \partial_Q) \cdots (- \lambda_N + z Q \partial_Q) - Q\right]
    f(z,Q) = 0
    \]
    Which is the differential equation satisfied by $J^{\textnormal{coh},\textnormal{eq}}$ (\ref{qkqde:eqn_pde_jh_eq}).
\end{proof}

\begin{remark}\label{qkqde:prop_JK_qpullback_is_natural_equivariant}
    The $q$-pullback $\varphi_{q,z}$ defined in Proposition \ref{qkqde:prop_confluence_equation_JK_equivariant} is the only $q$-pullback of the form $Q \mapsto \left( \frac{z}{1-q} \right)^\lambda Q$, with $\lambda \in \mathbb{Z}$, which defines a confluent $q$-difference system whose formal limit is non zero.
\end{remark}
The proof of this statement is the same as in Remark \ref{qkqde:remark_JK_qpullback_is_natural}.

\subsection{Confluence of the solution, equivariant version}\label{qkqde:subsection_comparison_confluence_solution_equivariant}

Let us mention a reason why confluence of the solution in the equivariant case should be easier than in the previous setting.
We recall that in Remark \ref{qkqde:rmk_jk_decomposition_and_nonequivariant_limit}, we had the non equivariant limit
\[
    \left( J^{K\textnormal{th},\textnormal{eq}}_{|P=\Lambda_i} \right)_{|\Lambda_0 = \cdots = \Lambda_N = 1} (q,Q) = \sum_{d \geq 0} \frac{Q^d}{(q;q)_d^{N+1}}
\]
In the proof of Proposition \ref{qkqde:prop_JK_confluence_sol}, the right hand side above had a well defined limit after the $q$-pullback without the need of a transformation.
We could expect that the limit when $q$ tends to $1$ of the functions $J^{K\textnormal{th},\textnormal{eq}}_{|P=\Lambda_i}(q,Q)$ to be well defined without the need of a gauge transform.

Let us introduce the fundamental solution of the $q$-pullback for which we want to compute the limit when $q$ tends to $1$.


\begin{remark}
    The $q$-pullback of the $q$-difference system (\ref{qkqde:eqn_jk_eq_as_system}) has a fundamental solution obtained from the equivariant $J$-function $J^{K\textnormal{th,eq}}(q,Q)$, which is given by 
    \begin{equation}\label{qkqde:eqn_fundamental_solution_before_confluence_equivariant}
        \mathcal{X}^{K\textnormal{th,eq}}\left(q,\varphi_{q,z}^{-1}(Q)\right)
        =
        \begin{pmatrix}
            J^{K\textnormal{th},\textnormal{eq}}_{|P=\Lambda_0}\left(q,\left( \frac{1-q}{z} \right)^{N+1} Q\right) &   \cdots  &   J^{K\textnormal{th},\textnormal{eq}}_{|P=\Lambda_N}\left(q,\left( \frac{1-q}{z} \right)^{N+1} Q\right)    \\
            \vdots                                                                                      &   \ddots  &   \vdots                                                      \\
            \delta_q^N J^{K\textnormal{th},\textnormal{eq}}_{|P=\Lambda_0}\left(q,\left( \frac{1-q}{z} \right)^{N+1} Q\right) &   \cdots  &  \delta_q^N  J^{K\textnormal{th},\textnormal{eq}}_{|P=\Lambda_N}\left(q,\left( \frac{1-q}{z} \right)^{N+1} Q\right)
        \end{pmatrix}
    \end{equation}
\end{remark}

\begin{prop}\label{qkqde:prop_jk_eq_confl_sol}          
    There exists a $q$-constant transformation $P_{q,z} \in  \textnormal{GL}_{N+1}\left( \mathcal{M}\left( \mathbb{E}_q \right) \right)$ such that the fundamental solution $\mathcal{X}^{K\textnormal{th,eq}}\left(q,\varphi_{q,z}^{-1}(Q)\right)P_{q,z}$ obtained from Equation \ref{qkqde:eqn_fundamental_solution_before_confluence_equivariant} is given by
    \[
        \left( 
            \mathcal{X}^{K\textnormal{th,eq}}\left(q,\varphi_{q,z}^{-1}(Q)\right)P_{q,z} 
        \right)_{li}
        =
        \left( \delta_q \right)^l
        \Lambda_i^{-\ell_q(Q)}
        \sum_{d \geq 0} 
        \frac{1}{z^{d(N+1)}}
        \frac{(1-q)^{d(N+1)}Q^d}{\left(q\Lambda_0 \Lambda_i^{-1}, \dots q, \dots ,q\Lambda_N \Lambda_i^{-1} ;q \right)_d}
    \]
    Moreover, this fundamental solution has a non trivial limit when $q$ tends to 1.
\end{prop}

\begin{proof}
    By the same argument as in Proposition \ref{qkqde:prop_JK_confluence_sol}, there is a transformation $P_{q,z} \in  \textnormal{GL}_{N+1}\left( \mathcal{M}\left( \mathbb{E}_q \right) \right)$ which changes the $q$-logarithms $\ell_q \left( \left(\frac{1-q}{z} \right)^{N+1} Q \right)$ in the fundamental solution $\mathcal{X}^{K\textnormal{th,eq}}\left(q,\varphi_{q,z}^{-1}(Q)\right)$ into the other $q$-logarithm $\ell_q(Q)$.
    Let us denote by $\left( 
            \mathcal{X}^{K\textnormal{th,eq}}\left(q,\varphi_{q,z}^{-1}(Q)\right)P_{q,z} 
        \right)_{0i}$ for $i\in \{ 0, \dots, N\}$ the elements on the first row of the fundamental solution $\mathcal{X}^{K\textnormal{th,eq}}\left(q,\varphi_{q,z}^{-1}(Q)\right) P_{q,z}$.
    We have
    \[
        \left( 
            \mathcal{X}^{K\textnormal{th,eq}}\left(q,\varphi_{q,z}^{-1}(Q)\right) P_{q,z}
        \right)_{0i}
         =
        \Lambda_i^{-\ell_q(Q)}
        \sum_{d \geq 0} 
        \frac{1}{z^{d(N+1)}}
        \frac{(1-q)^{d(N+1)}Q^d}{\left(q\Lambda_0 \Lambda_i^{-1}, \dots q, \dots ,q\Lambda_N \Lambda_i^{-1} ;q \right)_d}
    \]
    Therefore, the $q$-constant transformation $P_{q,z}$ satisfies the condition of the Proposition on the coefficient $\left( 
           \mathcal{X}^{K\textnormal{th,eq}}\left(q,\varphi_{q,z}^{-1}(Q)\right) P_{q,z}
        \right)_{li}$.
    
    Let us compute the limit of this new fundamental solution.
    Using Equation \ref{qkqde:eqn_relation_between_classes_as_numbers}, we have
    \begin{gather*}
        \lim_{q \to 1} \Lambda_i^{-\ell_q(Q)} 
        =
        \lim_{q \to 1}e^{\frac{\lambda_i}{z} \log(q) \ell_q(Q)}
        =
        Q^\frac{\lambda_i}{z}
        \\
        \lim_{q \to 1}
        \frac{1}{z^{d(N+1)}} \frac{(1-q)^{d(N+1)}}{\left(q\Lambda_0 \Lambda_i^{-1}, \dots q, \dots ,q\Lambda_N \Lambda_i^{-1} ;q \right)_d}
        =
        \lim_{q \to 1}
        \prod_{r=1}^d \prod_{j=0}^N \frac{1}{z} \frac{1-q}{1-q^r \Lambda_j \Lambda_i^{-1}}
        =
        \prod_{r=1}^d \prod_{j=0}^N \frac{1}{(\lambda_i-\lambda_j+rz)}
    \end{gather*}
    Thus,
    \[
        \lim_{q \to 1}
        \left( 
            P_q \mathcal{X}^{K\textnormal{th,eq}}\left(q,\varphi_{q,z}^{-1}(Q)\right)
        \right)_{0i}
        =
        Q^\frac{\lambda_i}{z}
        \sum_{d \geq 0} Q^d
        \prod_{r=1}^d \prod_{j=0}^N \frac{1}{(\lambda_i-\lambda_j+rz)}
    \]
    Since $\lim_{q \to 1} \delta_q = Q \partial_Q$, the fundamental solution $P_q \mathcal{X}^{K\textnormal{th,eq}}\left(q,\varphi_{q,z}^{-1}(Q)\right)$ is confluent.
\end{proof}

\subsection{Comparison between the confluence and the cohomological $J$-function}\label{qkqde:subsection_comparison_confluence_conclusion_equivariant}

\begin{defin}
    The first row of the fundamental solution $\mathcal{X}^{K\textnormal{th,eq}}\left(q,\varphi_{q,z}^{-1}(Q)\right) P_{q,z}$ of Proposition \ref{qkqde:prop_jk_eq_confl_sol} defines a $K$-theoretical function, which we denote by $P_{q,z} \cdot \varphi_{q,z}^* J^{K\textnormal{th,eq}}$.
\end{defin}

\begin{defin}\label{qkqde:def_confluence_jk_equivariant}
    By Proposition \ref{qkqde:prop_jk_eq_confl_sol}, the limit when $q$ tends to 1 of the function $P_{q,z} \cdot \varphi_{q,z}^* J^{K\textnormal{th,eq}}$ is well defined. We define the $K$-theoretical function $\textnormal{confluence}\left(J^{K\textnormal{th},\textnormal{eq}}\right)$ by
        \[
            \textnormal{confluence}\left(J^{K\textnormal{th},\textnormal{eq}}\right)(z,Q)
            =
            \lim_{t \to 0}
            P_{q^t,z} \cdot \varphi_{q^t,z}^* J^{K\textnormal{th,eq}}(q^t,Q)
        \]
\end{defin}

\begin{prop}\label{qkqde:prop_confluence_JK_equals_JH_equivariant}
    Consider the isomorphism of rings $\gamma_\textnormal{eq} : K_{T^{N+1}}\left( \mathbb{P}^N \right) \to H_{T^{N+1}}^* \left( \mathbb{P}^N ;\mathbb{Q} \right)$ given by $\gamma_\textnormal{eq}(\Psi_i)=Lag_i$ for all $i \in \{0, \dots, N\}$
    Then, we have
    \[
        \gamma_\textnormal{eq}\left( \textnormal{confluence}\left(J^{K\textnormal{th},\textnormal{eq}}\right)(z,Q) \right)
        =
        J^{\textnormal{coh},\textnormal{eq}}(z,Q)
    \]
\end{prop}

\begin{proof}
    On one hand, recall that by Equation \ref{qkqde:eqn_decomposition_jh_eq}, we have the decomposition
    \[
        J^{\textnormal{coh},\textnormal{eq}}(z,Q)
        =
        \sum_{i=0}^N
        J^{\textnormal{coh},\textnormal{eq}}_{|H=\lambda_i}(z,Q)
        Lag_i
    \]
    On the other hand we have
    \[
        P_{q,z} \cdot \varphi_{q,z}^* J^{K\textnormal{th,eq}}(t,z,Q)
        =
        \sum_{i=0}^N
        \left(
            \Lambda_i^{-\ell_q(Q)}
            \sum_{d \geq 0} 
            \frac{1}{z^{d(N+1)}}
            \frac{(1-q)^{d(N+1)}Q^d}{\left(q\Lambda_0 \Lambda_i^{-1}, \dots q, \dots ,q\Lambda_N \Lambda_i^{-1} ;q \right)_d}
        \right)
        \Psi_i
    \]
    Thus,
    \[
        \textnormal{confluence}\left(J^{K\textnormal{th},\textnormal{eq}}\right)(z,Q)
        =
        \sum_{i=0}^N
        \left(
            Q^\frac{\lambda_i}{z}
            \sum_{d \geq 0} Q^d
            \prod_{r=1}^d \prod_{j=0}^N \frac{1}{(\lambda_i-\lambda_j+rz)}
        \right)
        \Psi_i
    \]
    We conclude using $\gamma_\textnormal{eq}(\Psi_i) = Lag_i$ and
    \[
        J^{\textnormal{coh},\textnormal{eq}}_{|H=\lambda_i}(z,Q)
        =
        Q^\frac{\lambda_i}{z}
        \sum_{d \geq 0} Q^d
        \prod_{r=1}^d \prod_{j=0}^N \frac{1}{(\lambda_i-\lambda_j+rz)}
    \]
\end{proof}

We can now give a complete proof of Theorem \ref{qkqde:thm_confluence_jk_eq}

\begin{proof}[Proof of Theorem \ref{qkqde:thm_confluence_jk_eq}]
    \textbf{Confluence of the equation}. Using the $q$-pullback $\varphi_{q,z}$ of Proposition \ref{qkqde:prop_confluence_equation_JK_equivariant}, we obtain a confluent $q$-difference system. Its limit is the differential equation associated to the small equivariant cohomological $J$-function.
    The naturality of the $q$-pullback $\varphi_{q,z}$ is discussed in Remark \ref{qkqde:prop_JK_qpullback_is_natural_equivariant}.
    
    \textbf{Confluence of the solution}. By Proposition \ref{prop:qkqde_JK_eq_is_a_fundamental_solution}, we can encode the equivariant $K$-theoretical $J$-function as a fundamental solution of the $q$-pullback of the system (\ref{qkqde:eqn_jk_eq_as_system}), which we denote by the matrix $\mathcal{X}^{K\textnormal{th,eq}}\left(q,\varphi_{q,z}^{-1}(Q)\right)$ in Equation \ref{qkqde:eqn_fundamental_solution_before_confluence_equivariant}.
    By Proposition \ref{qkqde:prop_JK_confluence_sol}, there exists a $q$-constant transformation $P_{q,z} \in  \textnormal{GL}_{N+1}\left( \mathcal{M}\left( \mathbb{E}_q \right) \right)$ such that the fundamental solution $\mathcal{X}^{K\textnormal{th,eq}}\left(q,\varphi_{q,z}^{-1}(Q)\right) P_{q,z}$ is confluent.
    
    \textbf{Comparison with quantum cohomology}.
    The first row of the transformed fundamental solution $ \mathcal{X}^{K\textnormal{th,eq}}\left(q,\varphi_{q,z}^{-1}(Q)\right) P_{q,z}$ defines another $K$-theoretical function, which we denote by $P_{q,z} \cdot \varphi_{q,z}^*J^{K\textnormal{th,eq}}$.
    Since the fundamental solution was confluent, this function has a well defined limit when $q^t \to 1$, which we denote by $\textnormal{confluence}\left(J^{K\textnormal{th,eq}}\right)$.
    Using Proposition \ref{qkqde:prop_confluence_JK_equals_JH_equivariant}, we have
    \[
        \gamma_\textnormal{eq} \left( \textnormal{confluence}\left(J^{K\textnormal{th},\textnormal{eq}}\right)(z,Q) \right)
        =
        J^{\textnormal{coh,eq}}(z,Q)
    \]
\end{proof}

We remark that from this theorem we can recover the non equivariant version of the theorem, by taking the non equivariant limits in cohomology and $K$-theory.
The way to proceed can be summed up in the informal diagram below. 
\begin{center}
    \begin{tikzcd}
        QK_{T^{N+1}} \left( \mathbb{P}^N \right)
            \arrow[r,rightsquigarrow,"\ref{qkqde:thm_confluence_jk_eq}","\textnormal{confluence}"']
            \arrow[d, rightsquigarrow,"\underline{\Lambda} \to 1"]
        &
        QH_{T^{N+1}} \left( \mathbb{P}^N \right)
            \arrow[d, rightsquigarrow,"\underline{\lambda} \to 0"] 
        \\
        QK \left( \mathbb{P}^N \right)
            \arrow[r,rightsquigarrow,"\ref{qkqde:thm_JK_sol_confluence_statement}","\textnormal{confluence}"']
        &
        QH \left( \mathbb{P}^N \right)
    \end{tikzcd}
\end{center}
Note that non equivariant limits are only defined for the expressions of $J^{K\textnormal{th},\textnormal{eq}}$ and $J^{\textnormal{coh},\textnormal{eq}}$ which are not decomposed in our bases, cf. Remarks \ref{QKQDE:rmk_equivariant_basis_nonequivariant_limit} and \ref{qkqde:rmk_jk_decomposition_and_nonequivariant_limit}.


%
%

\bibliographystyle{alpha}
\small\bibliography{Bibliography}


\end{document}